\documentclass[12pt]{article}
\usepackage{amsfonts,mathbbol}
\usepackage{amscd}
\usepackage{amssymb}
\usepackage{amsthm}
\usepackage{amsmath, xspace}
\usepackage{blkarray}
\usepackage{bm}
\usepackage{cancel}
\usepackage{color}
\usepackage{graphics}
\usepackage{graphicx}
\usepackage{enumitem}
\usepackage{fancyhdr}
\usepackage{mathdots}
\usepackage{mathrsfs}
\usepackage{multicol}
\usepackage{stmaryrd}
\usepackage{ytableau}
\usepackage{young}
\usepackage[all]{xy}
\usepackage{dsfont}
\usepackage[plainpages,backref]{hyperref}
\usepackage{lscape}
\usepackage[table]{xcolor}

\theoremstyle{plain}
\newtheorem{theorem}{Theorem}[section]
\newtheorem{lemma}[theorem]{Lemma}
\newtheorem{proposition}[theorem]{Proposition}
\newtheorem{corollary}[theorem]{Corollary}
\theoremstyle{definition}
\newtheorem{definition}[theorem]{Definition}
\newtheorem{remark}[theorem]{Remark}
\newtheorem{example}{Example}[section]
\newtheorem{problem}{Problem}
\newtheorem{conjecture}[theorem]{Conjecture}
\theoremstyle{remark}

\newtheorem*{notation}{Notation}

\numberwithin{equation}{section}
   
\newcommand{\Sp}{\operatorname{Sp}}

\usepackage{ytableau}

\DeclareMathOperator{\Res}{Res}
\DeclareMathOperator{\Ind}{Ind}

\DeclareMathOperator{\lcm}{lcm}

\DeclareMathOperator{\sgn}{sgn}

\DeclareMathOperator{\DP}{DP}
\DeclareMathOperator{\OP}{OP}
\DeclareMathOperator{\Gal}{Gal}
\usepackage{tikz}
\usepackage{float}

\newcommand{\jacobi}[2]{\left( \frac{#1}{#2} \right)}

\title{Minimal Polynomials in Spin Representations\\of Symmetric and Alternating Groups}
\author{Amritanshu Prasad\footnote{The Institute of Mathematical Sciences, Chennai, Homi Bhabha National Institute, Mumbai.\\E-mail: amri@imsc.res.in}, Velmurugan S\footnote{Indian Institute of Science, Bangalore.\\E-mail: velmurugan.math@gmail.com}, Alexey Staroletov\footnote{Sobolev Institute of Mathematics, Novosibirsk, Russia. \\E-mail:staroletov@math.nsc.ru}}
\date{\vspace{-5ex}}
\begin{document}
	\maketitle
    \begin{abstract}
		We determine minimal polynomials of all elements of the double covers of symmetric $\tilde S_n$ and alternating groups $\tilde A_n$ in all irreducible complex spin representations.
        We also solve the problems of finding the orders of elements of $\tilde S_n$ and determining conjugacy classes for powers of given element of $\tilde S_n$.
	\end{abstract}


\section{Introduction}

Let $G$ be a double cover of either the symmetric group or the alternating group.
In this paper, we determine the minimal polynomials of elements of $G$ in all irreducible complex spin representations.
Our main results (Theorems~\ref{theorem:main} and~\ref{theorem:main_alt}) show that, for most elements $g\in G$ and most irreducible spin representations $\rho$ of $G$, the minimal polynomial of $\rho(g)$ is $x^k\pm 1$, where $k$ is the order of $g$ modulo the center of $G$.
The exceptional cases consist of four infinite families and finitely many sporadic cases.

The problem of determining minimal polynomials of group elements in irreducible representations goes back to the works of Hall--Higman~\cite{Hall_Higman} and Shult~\cite{Shult}.
A general formulation was articulated by Tiep and Zalesskii~\cite{Tiep_Zalesski}.

\begin{problem}[Tiep--Zalesskii]\label{problem:tiep_zalesskii}
Given a finite group $G$, classify all pairs $(\rho,g)$, where $\rho$ is an irreducible representation of $G$ and $g\in G$, such that the degree $d$ of the minimal polynomial of $\rho(g)$ satisfies $1<d<o(g),$ where $o(g)$ is the order of $g$ in $G/Z(G)$. In the first instance, under the condition that $o(g)$ is a prime power.
\end{problem}

This problem is especially interesting for groups close to simple groups (for example, almost simple and quasi-simple groups) as one expects the minimal polynomial of $\rho(g)$ to be of degree $o(g)$ for most pairs $(\rho,g)$ and the exceptions can be explicitly described.
In survey articles~\cite{survey_Zalesski_algebraic_Chevalley_2009,survey_Zalesski_eigenvalue_1_2025}, Zalesskii explains the importance 
of this problem and its connections to arithmetic geometry, algebraic geometry, topology, combinatorics, game theory and group theory.
At present, the only families for which this problem is solved for all elements and all irreducible complex representations are the symmetric and alternating groups.
This was achieved in a sequence of works: Swanson~\cite{Swanson} (for long cycles in symmetric groups), Yang--Staroletov~\cite{Yang_Staroletov} and Velmurugan~\cite{vel_dec_2024} (for specific families of permutations), Amrutha--Prasad--Velmurugan~\cite{p2024cyclic,Inv_vectors,fpsac2024} (for the occurrence of eigenvalue $1$), and finally Staroletov~\cite{Staroletov_all_eigenvalues} (complete solution).
In positive characteristic, even less is known. One of the main results is due to Thompson~\cite{Thompson_composition_factors_2008}, who described the minimal polynomials of $q$-cycles in $A_q$ in characteristic $p$ with $p<q<2p$ and $q$ prime.
In the case of the double covers $\tilde S_n$ and $\tilde A_n$, known results include $p$-cycles in characteristic $p$~\cite{Kleshchev_Zalesski_minimal_prime_elements_char_p_double_alt_2004} and $3$-cycles in arbitrary characteristic~\cite{Wales_3_cycle_minimal_polynomial}.

Our double cover of the symmetric group has the following defining relations:
\begin{align*}
\tilde{S}_n & =\langle z,t_1,t_2,\dotsc,t_{n-1}~|~z^2=1, t_i^2=z, (t_jt_{j+1})^3=z, (t_jt_k)^2=z  \text{ for } |j-k|\geq 2 \rangle.
\end{align*}
The irreducible spin representations of these groups are parametrized by strict partitions.
Let $\DP_n$ denote the set of strict partitions of $n$, and write $\DP_n^+$ (respectively, $\DP_n^-$) for those with an even (respectively, odd) number of even parts.
In the case of $\tilde S_n$: for each $\lambda\in\DP_n^+$, there exists a unique irreducible spin representation $(\rho_\lambda,V_\lambda)$ indexed by $\lambda$, and for each $\lambda\in\DP_n^-$, there exist exactly two non-isomorphic irreducible spin representations $(\rho_\lambda^\pm,V_\lambda^\pm)$ indexed by $\lambda.$
For a tuple $\alpha$ of positive integers, let $\lcm(\alpha)$ denote the least common multiple of its entries.

We now state the main result for $\tilde S_n$.

\begin{theorem}\label{theorem:main}
	Let $(\rho,V)$ be an irreducible spin representation of $\tilde S_n$ indexed by $\lambda\in \DP_n$. Let $g\in \tilde S_n$ have an image in $S_n$ with cycle type $\alpha=(\alpha_1,\dots,\alpha_l)$, and set $k=\lcm(\alpha)$.
	Let $\varepsilon\in\{\pm 1\}$ be such that $g^k=\varepsilon$.
	Then the minimal polynomial of $\rho(g)$ is $x^k-\varepsilon$, except in the following cases.

	For $n\le 8$, $(\lambda,\alpha)$ is one of the pairs listed in Table~\ref{table:exceptional_cases_1_8}.

	For $n>8$, $(\lambda,\alpha)$ lies in one of the following four infinite families:
	\begin{enumerate}
		\item $\lambda=(n)$, $\alpha$ contains $3$, all other parts of $\alpha$ are not divisible by $3$, and $5$ is not the unique part of $\alpha$ divisible by $5$.
		In this case, the minimal polynomial is $\frac{x^{k}- \varepsilon}{x^{k/3}- \varepsilon}.$
		\item $\lambda=(n)$, $\alpha$ contains $5$, all other parts of $\alpha$ are not divisible by $5$, and $3$ is not the unique part of $\alpha$ divisible by $3$.
		In this case, the minimal polynomial is $\frac{x^{k}- \varepsilon}{x^{k/5}- \varepsilon}$.
		\item $\lambda=(n)$, $\alpha$ contains both $3$ and $5$, while all other parts of $\alpha$ are not divisible by $3$ or $5$.
		In this case, the minimal polynomial is $\frac{(x^{k}- \varepsilon)(x^{k/15}- \varepsilon)}{(x^{k/3}- \varepsilon)(x^{k/5}- \varepsilon)}$.
		\item $\lambda=(n-1,1)$, $\alpha$ contains both $3$ and $5$, while all other parts of $\alpha$ are not divisible by $3$ or $5$.
		In this case, the minimal polynomial is $\frac{x^{k} - \varepsilon}{x^{k/15} - \varepsilon}$.
	\end{enumerate}
\end{theorem}
The double cover of the alternating group is the unique index $2$ subgroup of $\tilde S_n$.
The representation theory of double cover of alternating group is similar to that of alternating group:
 for each $\lambda\in\DP_n^-$, there exists a unique irreducible spin representation which is obtained by restricting the irreducible representation $(\rho_\lambda,V_\lambda)$ of $\tilde S_n$, denoted by $(\rho_\lambda,V_\lambda)$ again, while for each $\lambda\in\DP_n^+$, there are exactly two non-isomorphic irreducible spin representations $(\rho_\lambda^\pm,V_\lambda^\pm)$ indexed by $\lambda$ with $\Res^{\tilde S_n}_{\tilde A_n} V_\lambda = V_\lambda^+ \oplus V_\lambda^-$.
The main theorem for double cover of the alternating group $\tilde A_n$ is the following.

\begin{theorem}\label{theorem:main_alt}
	Let $(\rho,V)$ be an irreducible spin representation of $\tilde A_n$ indexed by $\lambda\in\DP_n$.
	Let $g\in \tilde A_n$ have an image in $A_n$ of cycle type $\alpha=(\alpha_1,\alpha_2,\dotsc,\alpha_l)$, and set $k=\lcm(\alpha)$.
	Let $\varepsilon \in \{\pm 1\}$ be such that $g^k=\varepsilon$.
	Then the minimal polynomial of $\rho(g)$ is $x^k-\varepsilon$, except in the following cases.

	For $n\leq 9$, $(\lambda,\alpha)$ is one of the pairs listed in Table~\ref{table:exceptional_cases_1_8_alt}.

	For $n\geq 10$, $(\lambda,\alpha)$ lies in one of the following four infinite families:
\begin{enumerate}
		\item $\lambda = (n)$, $\alpha$ contains $3$, all other parts of $\alpha$ are not divisible by $3$, and $5$ is not the unique part of $\alpha$ divisible by $5$.
		In this case, the minimal polynomial is $\frac{x^{k}-\varepsilon}{x^{k/3}-\varepsilon}$.
		\item $\lambda = (n)$, $\alpha$ contains $5$, all other parts of $\alpha$ are not divisible by $5$, and $3$ is not the unique part of $\alpha$ divisible by $3$.
		In this case, the minimal polynomial is $\frac{x^{k}-\varepsilon}{x^{k/5}-\varepsilon}$.
		\item $\lambda = (n)$, $\alpha$ contains both $3$ and $5$, while all other parts of $\alpha$ are not divisible by $3$ or $5$.
		In this case, the minimal polynomial is $\frac{(x^{k}-\varepsilon)(x^{k/15}-\varepsilon)}{(x^{k/3}-\varepsilon)(x^{k/5}-\varepsilon)}$.
		\item $\lambda = (n-1,1)$, $\alpha$ contains both $3$ and $5$, while all other parts of $\alpha$ are not divisible by $3$ or~$5$.
		In this case, the minimal polynomial is $\frac{x^{k}-\varepsilon}{x^{k/15}-\varepsilon}$.
	\end{enumerate}
\end{theorem}

The paper is organized as follows.
In Section~\ref{section:Preliminaries for the double cover of the symmetric group}, we recall the computation of spin characters for $\tilde{S}_n$ and $\tilde{A_n}$, and in Section~\ref{section:combinatorics_on_tableaux} we explain the analogues of the Littlewood--Richardson rule for spin characters.
Section~\ref{section:minimal_polynomial_of_n_cycles} is devoted to the proof of Theorem~\ref{theorem:main} for elements of cycle type $(n)$.
Theorems~\ref{theorem:main} and~\ref{theorem:main_alt} are proved from these results using the same general approach that worked for $S_n$ and $A_n$, however, several new difficulties arise. Computing the power map on conjugacy classes of $\tilde S_n$ and $\tilde A_n$ is, in general, a subtle question that we resolve in Section~\ref{sec:conju-classes-powers}. 
Sections~\ref{section:minimal_polynomial_some_family} and~\ref{section:minimal_poly_of_any_element_in_A_n} contain the proofs of Theorems~\ref{theorem:main} and~\ref{theorem:main_alt}, respectively.
\section{Computing Spin Character Values}\label{section:Preliminaries for the double cover of the symmetric group}

A double cover $\tilde S_n$ of the symmetric group is a group $G$ for which there exists a short exact sequence:
\begin{displaymath}
1\rightarrow \mathbb{Z}_2\rightarrow G\rightarrow S_n\rightarrow 1,
\end{displaymath}	
that does not split, and such that the image of $\mathbb{Z}_2$ is the center of $G$.

Schur~\cite{Schur} showed that double covers of $S_n$ exist if and only if $n\geq 4$.
He defined two families of groups by presentation:
\begin{equation}\label{eq:relations}
\begin{split}
\tilde{S}_n  &=\langle z,t_1,t_2,\dotsc,t_{n-1}~|~z^2=1, t_i^2=z, (t_jt_{j+1})^3=z, (t_jt_k)^2=z  \text{ for } |j-k|\geq 2 \rangle,\\
\nonumber \tilde{S}'_n  &=\langle z,t_1',t_2',\dotsc,t_{n-1}'~|~z^2=1, t_i'^2=1, (t_j't_{j+1}')^3=1, (t_j't_k')^2=z  \text{ for } |j-k|\geq 2 \rangle.
\end{split}
\end{equation}
For $n\geq 4$, these groups are double covers of $S_n$, and they are non-isomorphic except when $n=6$.
Moreover, up to isomorphism, they are the only double covers of $S_n$.

We shall identify free generators $t_j$ (resp. $t_j'$) with their images in $\tilde{S}_n$ (resp. $\tilde{S}'_n$).
But we denote the image of $z$ by $-1$. 

The map $t_j\mapsto it_j'$ gives rise to an isomorphism $\mathbb{C}[\tilde S_n]\to \mathbb{C}[\tilde S_n']$ of group algebras.
We will focus on $\tilde S_n$ since all results can easily be transferred to $\tilde S_n'$ using this isomorphism.

An element $\sigma\in \tilde{S}_n$ is said to be even (resp. odd) if its projection to $S_n$ is an even (resp. odd) permutation.
Denote by $\tilde A_n$ the subgroup of even elements of $\tilde S_n$.
The group $\tilde A_n$ is a double cover of $A_n$ for $n\geq 4$.

The representation theory of double covers $\tilde{S}_n$ and $\tilde{A}_n$ 
was originally developed by Schur~\cite{Schur}. 
We will follow the basic notation and terminology of Stembridge~\cite{Stembridge_Shifted},\cite{Stembridge_Survey}.

In the following two subsections, we discuss basic facts about the conjugacy classes and representation theory of $\tilde S_n$ and $\tilde A_n$.
In particular, we state an analog of Murnaghan--Nakayama rule, called Morris' recursion formula.
\subsection{Morris' Recursion Formula for Spin Characters}
Recall that a \emph{composition} of a positive integer $n$ of length $k$ is a sequence $\lambda=(\lambda_1,\lambda_2,\dots,\lambda_k)$ of positive integers whose sum equals $n$.
We write $\ell(\lambda)=k$ for the length of $\lambda$.
The composition $\lambda$ of $n$ is called a \emph{partition} of $n$ if $\lambda_1\geq \lambda_2\geq \cdots \geq \lambda_k$.
In this case, numbers $\lambda_i$ are called \emph{parts} of $\lambda$ and we write $\lambda\vdash n$.

Let us briefly discuss the description of the conjugacy classes of $\tilde S_n$.
Let $\pi:\tilde S_n\to S_n$ denote the projection map that sends $t_i$ to the simple transposition $(i,i+1)$.
For $\alpha\vdash n$, let $D_\alpha$ denote the conjugacy class of $S_n$ consisting of permutations with cycle type $\alpha$.
The preimage $C_\alpha:=\pi^{-1}(D_\alpha)$ is either a single conjugacy class of $\tilde S_n$ or it splits into two conjugacy classes of equal size.
We say that the cycle type of any element of $C_\alpha$ is $\alpha$.

Denote the set of all partitions of $n$ whose parts are all odd by $\OP_n$.
A partition $\lambda$ of $n$ is called \emph{strict} if all its parts are distinct. The set of all strict partitions of $n$
is denoted by $\DP_n$. Let $\DP^+_n$ (respectively, $\DP^-_n$) denote the set of all partitions $\lambda\in\DP_n$ such that $\lambda$ has an even (respectively, odd) number of even parts.

The following lemma (see \cite[p.\ 172]{Schur}) describes the conjugacy classes of $\tilde S_n$.
\begin{lemma}
\label{lemma:conjugacy_classes_of_double_cover}
Let $\alpha\vdash n$.
Then $C_\alpha$ splits into two conjugacy classes of $\tilde S_n$ of equal size if and only if $\alpha\in\OP_n \cup \DP^-_n$; otherwise, $C_\alpha$ is a single conjugacy class in $\tilde S_n$.
\end{lemma}
We choose a canonical representative for each conjugacy class based on the decomposition of a cycle into a product of transpositions in the symmetric group:
\begin{displaymath}
(k+1,k+2,\ldots,l)=(k+1,k+2)(k+2,k+3)\cdots(l-1,l) \quad \text{for } 0\leq k<l\leq n.
\end{displaymath}
\begin{definition}[Canonical representatives]\label{defn:canonical-rep}
For each composition $\alpha = (\alpha_1,\dotsc,\alpha_l)$ of $n$, define
\begin{displaymath}
	\sigma_j = t_{\alpha_1+\dotsb + \alpha_{j-1}+1} t_{\alpha_1+\dotsb+\alpha_{j-1}+2} \dotsb t_{\alpha_1 + \dotsb + \alpha_{j-1}+\alpha_j-1}
\end{displaymath}
for $j=1,\dotsc,l$.
The image of $\sigma_j$ in $S_n$ is a cycle of length $\alpha_j$.
Define
\begin{displaymath}
\sigma_\alpha=\sigma_1\sigma_2\dotsb \sigma_l.
\end{displaymath}
\end{definition}
The image of $\sigma_\alpha$ in $S_n$ is a permutation with cycle type $\alpha$.
For example, when $\alpha=(4,3,2,1)$, we have $\sigma_\alpha=t_1t_2t_3t_5t_6t_8$.

If $\alpha\in\OP_n\cup \DP^-_n$, set
\begin{displaymath}
\sigma_\alpha^+:=\sigma_\alpha \quad \text{and} \quad \sigma_\alpha^-:=-\sigma_\alpha^+ = -\sigma_\alpha.
\end{displaymath}
By Lemma~\ref{lemma:conjugacy_classes_of_double_cover}, a complete set of representatives for the conjugacy classes in $\tilde S_n$ is given by
\begin{displaymath}
\{\sigma_\alpha \mid \alpha\notin\OP_n\cup\DP_n^-\} \cup \{ \sigma_\alpha^\pm \mid \alpha\in\OP_n\cup\DP_n^-\}.
\end{displaymath}
As was mentioned in Introduction, the irreducible spin representations of $\tilde{S}_n$ are
\begin{displaymath}
	\{(\rho_\lambda, V_\lambda)\mid \lambda\in\DP_n^+\} \cup \{(\rho_\lambda^\pm, V^\pm_\lambda) \mid \lambda\in\DP_n^-\}.
\end{displaymath}
Let $\phi_\lambda$ denote the character of $\rho_\lambda$ for $\lambda\in \DP_n^+$ and the character of $\rho_\lambda^+$ for $\lambda\in \DP_n^-$.
We will also write $\phi_\lambda^\pm$ for the character of $\rho_\lambda^\pm$ when $\lambda\in\DP_n^-$.

We now state an analog of the Murnaghan--Nakayama rule for the double cover of the symmetric group.
There are several versions of the Murnaghan--Nakayama rule for $\tilde S_n$ in the literature, see~\cite{Morris_spin_Canadian_1965,Mac_sym,Hoffman_Humphreys,Bessendrodt_rule_of_morris}.
The version we shall use is the one given in Bessenrodt~\cite{Bessendrodt_rule_of_morris}, which is easy to state and apply.

Given a partition $\lambda\in\DP_n$, a \emph{bar length} of a cell in the shifted Young diagram of $\lambda$ is the hook length of that cell in the double diagram (or shift-symmetric diagram) of $\lambda$.
For example, when $\lambda=(7,5,4)$, the bar lengths of each cell in the shifted Young diagram of $\lambda$ are shown below inside the double diagram of $\lambda$:
\ytableausetup{smalltableaux}
\begin{displaymath}
	\begin{ytableau}
		\empty & 12 & 11 & 7 & 6 & 5 & 4 & 1\\
		\empty & \empty & 9 & 5 & 4 & 3 & 2 \\
		\empty & \empty & \empty & 4 & 3 & 2&1 \\
		\empty & \empty & \empty\\
		\empty & \empty & \empty\\
		\empty & \empty & \empty\\
		\empty
	\end{ytableau}
\end{displaymath}
For any partition $\eta$ of $n$, let $\epsilon_\eta = (-1)^{n-\ell(\eta)}$.

We now explain the process of bar removal.
We will denote by $(i,j)$ the cell in the $i$-th row and $j$-th column of the double diagram of $\lambda$, and its hook length by $h(i,j)$.
We will only consider hook lengths of cells that lie in the shifted diagram of $\lambda$.
If $h(i,j)\leq \lambda_i$, then the bar corresponding to the cell $(i,j)$ consists of the last $h(i,j)$ cells in the $i$-th row of the shifted diagram of $\lambda$.
On the other hand, if $h(i,j)>\lambda_i$, then the number of cells in the hook of $(i,j)$ that lie in the shifted diagram of $\lambda$ is $\lambda_i$ and $h(i,j)=\lambda_i+\lambda_j$.
The bar corresponding to such a cell consists of the cells in the $i$-th row and $j$-th rows of the shifted diagram of $\lambda$.
We denote the bar corresponding to the cell $(i,j)$ by $H(i,j)$.

To remove a bar from a shifted diagram, we simply delete these cells and rearrange the remaining cells to obtain a shifted diagram of a strict partition.
We denote the resulting diagram by $\lambda\setminus H(i,j)$.
The leg length $L(i,j)$ of the associated hook is the number of cells in the hook that lie strictly below the cell $(i,j)$ in the double diagram of $\lambda$.

\begin{theorem}[Morris' Recursion Formula] \label{theorem:spin_Murnaghan_Nakayama_rule}
	Let $\lambda\in\DP_n$ and $\alpha\in\OP_n$. If $l$ is a part of $\alpha$, then
	\begin{displaymath}
		\phi_\lambda(\sigma_\alpha) = \sum\limits_{h(i,j)=l} (-1)^{L(i,j)} 2^{m(i,j)} \phi_{\lambda\setminus H(i,j)}(\sigma_{\alpha\setminus l}),
	\end{displaymath} 
	where
		$m(i,j) = 
		\begin{cases}
			 1 & \text{if } \epsilon_{\lambda\setminus H(i,j)}-\epsilon_\lambda = 1,\\
			 0 & \text{otherwise}.
		\end{cases}$
\end{theorem}
\begin{example}
Let $\lambda=(7,5,4)$.
There are two cells in the double diagram of $\lambda$ with hook length $3$, namely $(2,6)$ and $(3,5)$.
The $3$-bar corresponding to cell $(2,6)$ consists of the last three cells in the second row of the shifted diagram of $\lambda$, while the $3$-bar corresponding to cell $(3,5)$ consists of the last three cells in the third row of the shifted diagram of $\lambda$.

Taking $l=9$, we see that there the cells $(1,4)$ and $(2,3)$ in the double diagram of $\lambda$ have hook length $9$.
The $9$-bar corresponding to cell $(1,4)$ consists of first and fourth rows, while the $9$-bar corresponding to cell $(2,3)$ consists of the second and third rows.
Morris' recursion formula would give, for example,
\begin{displaymath}
	\phi_{(7,5,4)}(\sigma_{(7,5,3,1)}) = - \phi_{(7,4,2)}(\sigma_{(7,5,1)}) +  \phi_{(7,5,1)}(\sigma_{(7,5,1)}).
\end{displaymath}

Taking $l=11$, the cell $(1,3)$ is the unique cell in the double diagram of $\lambda$ with hook length $11$.
The $11$-bar corresponding to cell $(1,3)$ consists of the first and third rows of the shifted diagram of $\lambda$.
Morris' recursion formula would give, for example,
\begin{displaymath}
	\phi_{(7,5,4)}(\sigma_{(11,5)}) = -\phi_{(5)}(\sigma_{(5)}) = -1.
\end{displaymath}
\end{example}
We end this subsection by explicitly describing the remaining character values~\cite{Stembridge_Shifted}.
\begin{lemma}
	\label{lemma:spin_character_values}
	Let $\lambda\in\DP_n$ and $\alpha\vdash n$.
	Then the character value $\phi_\lambda(\sigma_\alpha)$ can be computed by the following rules:
	\begin{enumerate}
		\item $\phi_\lambda(-\sigma)= -\phi_\lambda(\sigma)$ for all $\sigma\in \tilde S_n$.
		\item $\phi_\lambda(\sigma_\alpha)=0$ for all $\alpha\notin\OP_n\cup \DP^-_n$.
		\item For $\alpha\in\OP_n$, we apply Morris' recursion formula.
				In particular, we have $\phi_\lambda(\sigma_\alpha)\in \mathbb{Z}$ for all $\alpha\in\OP_n$.
		\item For $\lambda \in\DP_n^+$ and $\alpha \notin\OP_n$, we have
			$\phi_\lambda(\sigma_\alpha) = 0.$
		\item $\phi_{\lambda}^+(\sigma_\alpha)=(-1)^{\epsilon_\alpha}\phi_{\lambda}^-(\sigma_\alpha)$ for all $\lambda\in \DP^-_n$.
		\item $\phi_\lambda^\pm(\sigma_\alpha) = 
		\begin{cases}
				\pm i^{(n-\ell(\lambda)+1)/2} \sqrt{\tfrac{1}{2}\prod_j \lambda_j} & \text{if }\lambda=\alpha\in\DP_n^-,\\
				0 & \text{if } \lambda\neq \alpha \text{ and } \alpha \notin\OP_n.
		\end{cases}$    
	\end{enumerate}
\end{lemma}

\subsection{Spin Characters of the Alternating Group}\label{section:spin_characters_of_A_n}

We begin with Schur's description of conjugacy classes in $\tilde A_n$.

\begin{lemma}\cite[Theorem 2.7]{Stembridge_Shifted}
	\label{lemma:conjugacy_classes_A_n}
	Let $\alpha$ be a partition of $n$ with $(-1)^{n-\ell(\alpha)}=1$.
	Then the conjugacy class $C_\alpha$ of $\tilde S_n$ splits into two conjugacy classes of $\tilde A_n$ if and only if $\alpha \in\DP_n^+$.
	In particular, if $\alpha\in\OP_n\cap\DP_n^+$, then each of $C_\alpha^\pm$ splits into two conjugacy classes.
\end{lemma}

\begin{notation}
Fix an element $s\in \tilde S_n\setminus \tilde A_n$.
For $\alpha\in\DP_n^+\setminus\OP_n$, write
$
\sigma_\alpha^+ := \sigma_\alpha,
$
$
\sigma_\alpha^- := s\sigma_\alpha s^{-1},
$
as representatives of the two $\tilde A_n$-conjugacy classes
$
C_\alpha = C_\alpha^+ \sqcup C_\alpha^-.
$

For $\alpha\in\OP_n$, we already defined
$
\sigma_\alpha^+ := \sigma_\alpha,
$
$
\sigma_\alpha^- := -\sigma_\alpha.
$
There is no ambiguity in the notation $\sigma_\alpha^-$, since these definitions apply to disjoint cases.

For $\alpha\in\DP_n^+\cap\OP_n$, define
\[
\sigma_\alpha^{++}:=\sigma_\alpha^+,\quad
\sigma_\alpha^{+-}:=s\sigma_\alpha^+s^{-1},\quad
\sigma_\alpha^{-+}:=\sigma_\alpha^-,\quad
\sigma_\alpha^{--}:=s\sigma_\alpha^-s^{-1},
\]
as representatives of $C_\alpha^{++},C_\alpha^{+-},C_\alpha^{-+},C_\alpha^{--}$, where
$
C_\alpha^+ = C_\alpha^{++}\sqcup C_\alpha^{+-},
$
$
C_\alpha^- = C_\alpha^{-+}\sqcup C_\alpha^{--}.
$
\end{notation}

The description of irreducible spin representations of $\tilde A_n$ follows from the standard representation theory of index $2$ subgroup $H$ of any group $G$.
But for the sake of completeness, we give a more explicit description of them.


Let $\sgn$ denote the sign representation of $\tilde S_n$, obtained by composing the sign representation of $S_n$ with $\pi:\tilde S_n\to S_n$.  
For a representation $V$ of $H \leq \tilde S_n$, we denote the representation $(\Res^{\tilde S_n}_{H}\sgn) \otimes V$ by $V'$.

A spin representation $(\rho, V)$ of $\tilde S_n$ is said to be \emph{self-associate} if $V\cong V'$.
For a non self-associate representation $V$, the module $\sgn \otimes V$ is called its \emph{associate}.
It is known that irreducible spin representations indexed by $\DP_n^+$ are self-associate, while those indexed by $\DP_n^-$ are not; see~\cite{Stembridge_Shifted}.
If $\lambda\in\DP_n^-$, then we denote by $V_\lambda^\pm$ modules affording characters $\phi_\lambda^\pm$ from Lemma~\ref{lemma:spin_character_values}, respectively. Their restrictions to $\tilde A_n$ are equal and irreducible.

If $(\rho,V)$ is self-associate, there exists $T\in GL(V)$ such that
$T\rho(t_j)(v)=-\rho(t_j)T(v)$ for all $v\in V$ and $j\in [n-1]$.
Hence, $T^2$ commutes with $\rho(g)$ for all $g\in \tilde S_n$, so by Schur's lemma $T^2$ is scalar. After rescaling, we may assume $T^2=1$.
The operator $T$ is called an \emph{associator} of $V$.
The \emph{difference operator} is
\begin{displaymath}
	\Delta(\sigma)=\text{tr}(T\sigma|_V),\quad \sigma\in \tilde S_n.
\end{displaymath}
By Schur's lemma, $T$ (hence $\Delta$) is unique up to sign.
\begin{lemma}[see {\cite[Lemma 4.1]{Stembridge_Shifted}}]\label{lemma:restriction_to_A_n}
	Let $(\rho,V)$ be a self-associate irreducible representation of $\tilde S_n$, and let $T$ be an associator of $V$.
	Then
	$
	\Res_{\tilde A_n}^{\tilde S_n}V=V^+\oplus V^-,
	$
	$
	V^\pm=\{v\in V\mid T(v)=\pm v\},
	$
	where $V^+$ and $V^-$ are non-isomorphic irreducible representations of $\tilde A_n$.
	Moreover, for $\sigma\in\tilde A_n$,
	$
	\Delta(\sigma)=\phi^+(\sigma)-\phi^-(\sigma),
	$
	where $\phi^\pm$ are the characters of $V^\pm$.
\end{lemma}

The next theorem gives explicit values of the difference character.

\begin{theorem}\label{theorem:Delta_details_double_alternating}\cite[Theorem 7.4]{Stembridge_Shifted}
	Let $\lambda\in\DP_n^+$, and let $\Delta_\lambda$ be the difference operator on $V_\lambda$. Then
	\begin{displaymath}
	\Delta_\lambda(\sigma)= \phi_\lambda^+(\sigma) - \phi_\lambda^-(\sigma) =
	\begin{cases}
	\pm i^{(n-\ell(\lambda))/2} \sqrt{\prod_i \lambda_i} & \text{if the cycle type of } \sigma\text{ is }\lambda,\\
	0 & \text{otherwise}.
	\end{cases}
	\end{displaymath}
\end{theorem}

Fix an associator as in Theorem~\ref{theorem:Delta_details_double_alternating}. Then
\begin{align}
	\phi_\lambda^\pm(\sigma_\lambda) =
	\begin{cases}
		\pm i^{(n-\ell(\lambda))/2}\sqrt{\prod_i \lambda_i}/2 & \text{if }\lambda \in \DP_n^+\setminus\OP_n,\\
		(\phi_\lambda(\sigma_\lambda)\pm i^{(n-\ell(\lambda))/2}\sqrt{\prod_i \lambda_i})/2 & \text{if } \lambda \in\OP_n \cap \DP_n^+.
	\end{cases}
\end{align}

\noindent If $\lambda\in \DP_n^+\setminus\OP_n$,
\begin{align*}
	\phi_\lambda^\pm(\sigma_\lambda^+) = \phi_\lambda^\pm(\sigma_\lambda),
	\qquad
	\phi_\lambda^\pm(\sigma_\lambda^-) = \phi_\lambda^\mp(\sigma_\lambda^+).
\end{align*}
If $\lambda \in \DP_n^+ \cap\OP_n$,
\begin{displaymath}
	\phi_{\lambda}^{\pm} (\sigma_\lambda^{+-}) = \phi_\lambda^\mp(\sigma_\lambda^{++}), \quad
	\phi_{\lambda}^{\pm} (\sigma_\lambda^{-+}) = -\phi_\lambda^\pm(\sigma_\lambda^{++}), \quad
	\phi_{\lambda}^{\pm} (\sigma_\lambda^{--}) = -\phi_\lambda^\mp(\sigma_\lambda^{++}).
\end{displaymath}
\begin{notation}
    For subsets $A,B$ of $\mathbb{C}$, we define $A \times B$ to be $\{ab \mid a\in A, b\in B\}$.
\end{notation}
\begin{lemma}
	\label{lemma:min_same_on_V_+_V_-}
	Let $\lambda\in \DP_n^+$ and $\sigma \in \tilde A_n$ such that the cycle type of $\sigma$ is not $\lambda$.
	Then the minimal polynomial of $\sigma$ on $V_\lambda^+$, $V_\lambda^-$, and $V_\lambda$ is the same.
    Consequently, $\Sp_{V_\lambda}(\sigma T)= \{\pm 1\} \times \Sp_{V_\lambda}(\sigma)$.
\end{lemma}
\begin{proof}
	No element of $\langle \sigma\rangle$ has cycle type $\lambda$: if $(k,|\langle \sigma\rangle|)=1$, then $\sigma^k$ has the same cycle type as $\sigma$; if $(k,|\langle \sigma \rangle|)>1$, then $\sigma^k$ has repeated cycle lengths.
	Hence, by Theorem~\ref{theorem:Delta_details_double_alternating}, we have
	$\phi_\lambda^+(\tau)=\phi_\lambda^-(\tau)=\phi_\lambda(\tau)/2$ for all $\tau\in \langle \sigma\rangle$.
	Therefore, for any linear character $\delta$ of $\langle \sigma\rangle$,
	\begin{displaymath}
		\langle \phi_\lambda^+|_{\langle \sigma \rangle}, \delta \rangle
		=
		\langle \phi_\lambda^-|_{\langle \sigma \rangle}, \delta \rangle
		=
		\tfrac12\langle \phi_\lambda|_{\langle \sigma \rangle}, \delta \rangle.
	\end{displaymath}
	So the same linear characters occur in all three restrictions, and the minimal polynomials coincide.
    Since $T|_{V_\lambda^\pm}=\pm I$, the latter statement follows.
\end{proof}

\section{The Analogue of the Littlewood--Richardson Rule}
\label{section:combinatorics_on_tableaux}
In this section, we describe reduced Clifford products following~\cite{Stembridge_Shifted}, which is an analogue of induction from Young subgroups. We then state an analogue of the Littlewood--Richardson rule for the double cover of the symmetric group. Finally, we prove some corollaries of these constructions that are necessary for using induction in proving our main results.

\subsection{Reduced Clifford Product}
We begin with the following definition.
\begin{definition}[Spin representation]
Let $H$ be a subgroup of $\tilde S_n$. We call a (irreducible) representation $(\rho,V)$ of $H$ a \emph{(irreducible) spin representation} if $-1 \notin H$ or $\rho(-1)=-I$.
\end{definition}

Let $\beta=(\beta_1,\dotsc,\beta_l)$ be a composition of $n$. For $k\in \{1,\dotsc,l\}$, define 
\begin{displaymath}
	J_k=\{j\in [n] \mid \beta_1+\dotsb+\beta_{k-1}<j < \beta_1+\dotsb+\beta_k\} \text{ and } J=\bigcup_{k=1}^l J_k.
\end{displaymath}
Let $\tilde S_\beta$ be the subgroup of $\tilde S_n$ generated by $-1$ and $\{t_j \mid j\in J\}$.
We sometimes write $\tilde S_\beta$ as $\tilde S_{\beta_1}\times_c \tilde S_{\beta_2}\times_c \cdots \times_c \tilde S_{\beta_l}$.
For each $k$, let $\tilde S_{J_k}$ be the subgroup of $\tilde S_\beta$ generated by $-1$ and $\{t_j \mid j\in J_k\}$.
Note that $S_\beta=\pi(\tilde S_{\beta})$ and $S_{J_k}=\pi(\tilde S_{J_k})$ are Young subgroups in $S_n$.
Since $S_\beta=S_{\beta_1} \times S_{\beta_2} \times \cdots \times S_{\beta_l}$, the irreducible representations of $S_\beta$ are tensor products of irreducible representations of $S_{\beta_i}$, $i\in \{1,\dotsc,l\}$.

Since $\tilde S_\beta$ is not a direct product of the $\tilde S_{\beta_i}$ its irreducible representations are not tensor products of irreducible representations of the $\tilde S_{\beta_i}$.
Instead, they are the so-called \emph{reduced Clifford products}.

Let $V_i$ be an irreducible representation of $\tilde S_{\beta_i}$ for each $i=1,\dotsc,l$.
Suppose that among them $r$ representations are self-associate and $s$ are non self-associate, so $r+s=l$.
For simplicity, assume that $V_1, V_2,\dotsc,V_r$ are self-associate, while $V_{r+1}, V_{r+2},\dotsc,V_l$ are non self-associate.
Consider the Clifford algebra
\begin{displaymath}
	\mathscr {C}_s = \mathbb{C}\langle \xi_{r+1},\xi_{r+2},\dotsc,\xi_l \mid \xi_k^2=1, \xi_k\xi_j=-\xi_j\xi_k \text{ for } r+1\leq k,j \leq l, k\neq j\rangle.
\end{displaymath}
Fix an irreducible $\mathscr{C}_s$-module $V$ and associators $T_j\in GL(V_j)$ (as defined in Section~\ref{section:spin_characters_of_A_n}) for each $V_j$ with $j \in [r]$.
The \emph{reduced Clifford product} of $V_1, V_2,\dotsc,V_l$ with respect to $V$ and $\{T_j\}_{j=1}^r$ is the tensor product $V\otimes V_1\otimes V_2\otimes \cdots \otimes V_l$, with $\tilde S_\beta$-module structure defined by
\begin{equation}\label{eq:action_in_reduced_Clifford_product}
	t_j(v\otimes v_1\otimes v_2\otimes \cdots \otimes v_l)=
	\begin{cases}
		v\otimes A_1 v_1\otimes A_2 v_2\otimes \cdots \otimes A_l v_l & \text{if } j\in J_k, 1\leq k \leq r,\\
		\xi_k(v)\otimes B_1 v_1\otimes B_2 v_2\otimes \cdots \otimes B_l v_l & \text{if } j\in J_k, r< k \leq l,
	\end{cases}
\end{equation}
where the operators $A_i$ and $B_j$ are defined as follows:
\begin{align*}
	(A_1,\dotsc,A_l)&=(T_1,\dotsc,T_{k-1},t_j,1,\dotsc,1), \\
	(B_1,\dotsc,B_l)&=(T_1,\dotsc,T_{r},1, \dotsc,1,t_j,1,\dotsc,1).
\end{align*}
In both cases, the operator $t_j$ (with $j\in J_k$) appears in the $k$-th position.
We denote the resulting representation of $\tilde S_\beta$ by $V_1\otimes_c\dots\otimes_c V_l$.
It follows from the definition that the reduced Clifford
products for $\tilde{S_\beta}$ may be 
labeled  by $l$-tuples $\lambda=(\lambda^1,\ldots,\lambda^l)$
of partitions with $\lambda^j\in\DP_{\beta_j}$.
When the number of odd elements of $\lambda$ is
odd, there are two mutually associate reduced Clifford products (see the second case in the following proposition).
\begin{proposition}{\em\cite[Proposition 4.2]{Stembridge_Shifted}}
	\label{prop:reduced_Clifford_product}
	Let $V_1, V_2,\dotsc,V_l$ be irreducible spin representations of $\tilde S_{\beta_1}, \tilde S_{\beta_2},\dotsc,\tilde S_{\beta_l}$, respectively, and let $s$ be the number of non self-associate representations.
	Let $\phi$ be the character of the reduced Clifford product of $V_1, V_2,\dotsc,V_l$ with respect to $V$ and $\{T_j\}_{j=1}^r$.
	For each $j\in [l]$, let $\pi_j \in \tilde S_{\beta_j}$.
	Then the following statements hold true.
	\begin{enumerate}
		\item If $\pi_j$ is even for all $j\in [l]$, then
		\begin{displaymath}
			\phi(\pi_1\pi_2\cdots \pi_l)= 2^{\lfloor s/2 \rfloor}\phi_1(\pi_1)\phi_2(\pi_2)\cdots \phi_l(\pi_l).
		\end{displaymath}
		\item If $s$ is odd and $\Delta_j$ is the difference operator of $V_j$, then
		\begin{displaymath}
			\phi(\pi_1\pi_2\cdots \pi_l)= \pm(2i)^{\lfloor s/2 \rfloor}\Delta_1(\pi_1)\Delta_2(\pi_2)\cdots \Delta_r(\pi_r)\phi_{r+1}(\pi_{r+1})\cdots \phi_l(\pi_l),
		\end{displaymath}
		provided that $\pi_j$ is even for all $j\in [r]$ and $\pi_j$ is odd for all $j\in [r+1,l]$.
		\item In all other cases, $\phi(\pi_1\pi_2\cdots \pi_l)=0$.
	\end{enumerate}
\end{proposition}

Note that the sign $\pm$ in the second case depends only on the choice of the Clifford module $V$ and the operators $T_j$, not on the $\pi_j$.

\subsection{The Shifted Littlewood--Richardson Rule}

To state an analogue of the Littlewood--Richardson rule for $\tilde S_n$, we first recall several definitions which are mainly taken from \cite[Section 3]{Stembridge_Survey}.

The \emph{shifted diagram} of shape $\lambda$ consists of cells
\begin{displaymath}
D_\lambda' := \{(i,j)\in \mathbb{Z}^2 \mid i\le j<\lambda_i+i,\ 1\le i\le \ell(\lambda)\}.
\end{displaymath}
Fix the ordered alphabet
\begin{displaymath}
P'=\{1'<1<2'<2<3'<3<\cdots\}.
\end{displaymath}
The letters $1',2',3',\ldots$ are called marked. For $a\in P'$, denote by $|a|$ its unmarked version.

A \emph{shifted tableau} of shape $\lambda$ is a map $T:D_\lambda'\to P'$ such that:
\begin{enumerate}
	\item[(R1)] $T(i,j)\le T(i+1,j)$ and $T(i,j)\le T(i,j+1)$.
	\item[(R2)] At most one $k$ appears in each column ($k=1,2,3,\ldots$).
	\item[(R3)] At most one $k'$ appears in each row ($k'=1',2',3',\ldots$).
\end{enumerate}
A shifted tableau is represented by filling each cell $(i,j)$ of $D_\lambda'$ by the letter $T(i,j)$.
For example, when $\lambda=(4,3,1)$,
\begin{displaymath}
T = \begin{ytableau}
1' & 1 & 1 & 3'\\
\none & 2' & 2 & 3'\\
\none & \none & 3
\end{ytableau}
\end{displaymath}
is a shifted tableau of shape $\lambda$.
The \emph{content} of a shifted tableau $T$ is the sequence $\gamma=(\gamma_1,\gamma_2,\ldots)$, where $\gamma_k$ is the number of boxes $(i,j)\in D_\lambda'$ with $|T(i,j)|=k$.

A \emph{skew shifted diagram} is an array of boxes of the form
\begin{displaymath}
D'_{\lambda/\mu}:=D'_\lambda \setminus D'_\mu
\end{displaymath}
for strict partitions $\lambda$ and $\mu $ with $\mu\subseteq\lambda$.  
A shifted tableau of shape $\lambda/\mu$ is a map $T:D'_{\lambda/\mu}\to P'$ satisfying (R1)--(R3).

Let $w=w_1\cdots w_k$ be a word in the alphabet $P'$.  
Denote by $w^r$ its reverse, $w_k\cdots w_1$, and by $\hat w$ the word obtained by toggling all marks (for example, if $w=3'21$, then $\hat w=32'1'$).

For $i\ge 1$ and $0\le j\le k$, let $n_i(w,j)$ be the number of occurrences of the letter $i$ (not $i'$) among $w_1,\ldots,w_j$, with the convention $n_i(w,0)=0$.

The word of a shifted tableau $T$ is the sequence $w(T)$ obtained by reading the rows of $T$ from left to right, starting from the bottom row.
Its extended word is $e(T)=w^r\hat w$. 
For the preceding example,
\begin{displaymath}
w(T)=32'23'1'113',\quad e(T)=3'111'3'22'33'22'311'1'3.
\end{displaymath}

The tableau $T$ is said to satisfy the \emph{shifted lattice property} if $e=e(T)=e_1\cdots e_{2n}$ satisfies the following condition for all $i>1$ and $0\le j<2n$:
\begin{displaymath}
\text{(SLP)}\quad n_i(e,j)=n_{i-1}(e,j)\ \Longrightarrow\
\begin{cases}
e_{j+1}\neq i,i', & \text{if } 0\le j<n,\\
e_{j+1}\neq i,(i-1)', & \text{if } n\le j<2n.
\end{cases}
\end{displaymath}
For a strict partition $\lambda$, define
\begin{displaymath}
c_\lambda=
\begin{cases}
\sqrt{2}, & \text{if } \lambda \text{ is odd},\\
1, & \text{if } \lambda \text{ is even}.
\end{cases}
\end{displaymath}
\begin{definition}
	\label{lemma:shifted_Littlewood_Richardson_coefficients}
	The shifted Littlewood-Richardson coefficient $f_{\mu\nu}^\lambda$ is the number of shifted tableaux $T$ of shape $\lambda/\mu$ and content $\nu$ such that:
	\begin{itemize}
		\item $T$ satisfies \textup{(SLP)}.
		\item For each $1\le i\le \ell(\nu)$, the leftmost occurrence of $i$ in $|w(T)|$ is unmarked in $w(T)$.
	\end{itemize}
\end{definition}
A tableau is called a Yamanouchi tableau if it satisfies the above two properties.

Shifted Littlewood-Richardson coefficients are the structure constants of the algebra of Schur $Q$-functions~{\cite[Theorem 8.3]{Stembridge_Shifted}}.
The following is the analog of the Littlewood–-Richardson rule for spin characters.
\begin{theorem}{\em\cite[Theorem 8.1]{Stembridge_Shifted}}\label{theorem:Littlewood-Richardson-rule}
	Let $\mu\in \DP_k$, $\nu\in \DP_{n-k}$, and $\lambda\in \DP_n$.
	Let $\phi_\mu\times_c\phi_\nu$ be a reduced Clifford product of $\phi_\mu$ and $\phi_\nu$.
	Then
	\begin{equation}
		\langle \phi_\mu\times_c\phi_\nu,\phi_\lambda\rangle
		=
		\frac{1}{c_\lambda c_{\mu\cup\nu}}\,
		2^{(\ell(\mu)+\ell(\nu)-\ell(\lambda))/2}\,
		f_{\mu\nu}^\lambda,
	\end{equation}
	unless $\lambda$ is odd and $\lambda=\mu\cup\nu$. In that case,
	$\langle \phi_\mu\times_c\phi_\nu,\phi_\lambda^++\phi_\lambda^-\rangle=1$.
	The choices in the construction of the Clifford product determine which of $\phi_\lambda^+$ and $\phi_\lambda^-$ occurs in $\phi_\mu\times_c\phi_\nu$.
\end{theorem}
The commutativity of Schur $Q$-functions implies
\begin{corollary}
	\label{corollary:commutativity_of_shifted_Littlewood_Richardson_coefficients}
	For $\mu\in\DP_k$, $\nu\in\DP_{n-k}$, and $\lambda\in\DP_n$, we have
	\begin{displaymath}
	f_{\mu\nu}^\lambda=f_{\nu\mu}^\lambda.
	\end{displaymath}
\end{corollary}
\begin{lemma}\label{lemma:most_dominant_alpha_beta}
	Let $n=m+k$ for some positive integers $m$ and $k$.
	If $\lambda\in\DP_n$ and $\mu\in\DP_m$ with $\mu\subseteq \lambda$, then there exists a strict partition $\nu$ of $k$ such that $f_{\mu\nu}^\lambda>0$.
\end{lemma}
\begin{proof}
	Let $\lambda\in\DP_n$ and $\mu\in\DP_m$ be a partition of $m$ contained in $\lambda$.
	Recall that $Q_{\lambda/\mu}=\sum_{\nu} f_{\mu\nu}^\lambda Q_\nu$ (see~\cite{Mac_sym}).
	In order to prove that $f_{\mu\nu}^\lambda\neq 0$ for some $\nu$, it suffices to show that $Q_{\lambda/\mu}$ is nonzero.
	It turns out that $Q_{\lambda/\mu} = \sum_{T} 2^{b(T)} x^T$, where the sum is over shifted tableaux of shape $\lambda/\mu$, and $b(T)$ is the number of border strips into which $T$ is decomposed.

	Clearly, there is at least one shifted tableau of shape $\lambda/\mu$, obtained by filling each column of the shifted skew shape $\lambda/\mu$ with the entries $1,2,\dots$ from top to bottom, so that the entries are strictly increasing down each column, starting with $1$.
	Hence, $Q_{\lambda/\mu}$ is nonzero.
\end{proof}
\begin{lemma}\label{lemma:combinatorial_lemma_for_double_A_n}
Let $\lambda\in \DP_n$ with $\ell(\lambda)\ge 3$, and let $p\leq \lambda_1$.
Then there exists $\alpha \in \DP_{n-p}$ such that
$f_{(p)\alpha}^{\lambda}>0$ and $\alpha \neq (n-p)$.
If $p>2$ and $n>6$, then there exists $\beta \in \DP_{n-p}$ such that
$f_{(p-1,1)\beta}^{\lambda}>0$ and $\beta\neq (n-p)$.
\end{lemma}

\begin{proof}
Since $p\leq \lambda_1$ and $\ell(\lambda)\geq 2$, the shifted diagram of $\lambda$ contains $(p)$ and $(p-1,1)$.
By Lemma~\ref{lemma:most_dominant_alpha_beta}, there exist strict partitions
$\alpha,\beta\in\DP_{n-p}$ such that $f_{(p)\alpha}^{\lambda}>0 $ and $f_{(p-1,1)\beta}^{\lambda}>0.$
If neither $\alpha$ nor $\beta$ is equal to $(n-p)$, there is nothing to prove.

Suppose first that $\alpha=(n-p)$.
Then there exists a shifted tableau $T$ of shape $\lambda/(p)$ with content
$(n-p)$.
Hence every entry of $T$ is either $1'$ or $1$.
By conditions (R1)--(R3), every cell of $T$, except possibly the first cell in
the second row, must contain $1$.
Since $\ell(\lambda)\ge3$, the tableau has a third row, whose entries must be
strictly greater than $1$, a contradiction.
Therefore, $\alpha\neq(n-p)$.

Now suppose that $\beta=(n-p)$.
Then there exists a shifted tableau of shape $\lambda/(p-1,1)$ with content
$(n-p)$.
Since $\ell(\lambda)\ge3$, conditions (R1)--(R3) imply that $\lambda_3=1$.
Depending on whether the $p$-th column of $\lambda/(p-1,1)$ consists of one or
two cells, we obtain one of the two tableaux shown below.
Hence, $f_{(p-1,1)(n-p-1,1)}^\lambda>0.$

\ytableausetup{boxsize=1em}
\[
\begin{ytableau}
*(blue) & *(blue) & *(blue) & *(blue) & 1' & 1 & 1 & 1\\
\none & *(blue) & 1 & 1 & 1\\
\none & \none & 2
\end{ytableau}
\hspace{4cm}
\begin{ytableau}
*(blue) & *(blue) & *(blue) & *(blue) & 1 & 1 & 1 & 1\\
\none & *(blue) & 1 & 1\\
\none & \none & 2
\end{ytableau}
\]
\end{proof}

\begin{lemma}\label{lemma:(3)_(2,1)_one_different_from_basic}
    Let $\lambda \in \DP_n$ with $\lambda_2\geq 2$ and $n\geq 7$.
    Then there exist partitions $\mu,\nu \in \DP_{n-3}\setminus\{(n-3)\}$ such that
    $f_{(3)\mu}^{\lambda}>0$ and $f_{(2,1)\nu}^{\lambda}>0$.
\end{lemma}

\begin{proof}
Since $\ell(\lambda)\ge2$, the shifted diagram of $\lambda$ contains both $(3)$ and $(2,1)$.
By Lemma~\ref{lemma:most_dominant_alpha_beta}, there exist partitions
$\mu,\nu\in\DP_{n-3}$ such that
$f_{(3)\mu}^{\lambda}>0$ and $f_{(2,1)\nu}^{\lambda}>0$.
If neither $\mu$ nor $\nu$ is equal to $(n-3)$, there is nothing to prove.

Suppose first that $\mu=(n-3)$.
Then there exists a skew shifted Young tableau of shape $\lambda/(3)$ with content $(n-3)$.
Conditions (R1)--(R3) imply that $\lambda_3=1$ and $\lambda_2\leq3$, so we find that $\lambda=(n-2,2)$ or $(n-3,3)$.
In either case, changing the rightmost entry in the second row from $1$ to $2$ yields, $f_{(3)(n-4,1)}^{\lambda}>0,$ as illustrated in the first two diagrams below.

\ytableausetup{boxsize=1em}
\[
\begin{ytableau}
*(blue) & *(blue) & *(blue) & 1' & 1 & 1 & 1\\
\none 1 & 1 & 1 & 2
\end{ytableau}
\hspace{3cm}
\begin{ytableau}
*(blue) & *(blue) & *(blue) & 1 & 1 & 1 & 1\\
\none 1 & 1 & 2
\end{ytableau}
\hspace{3cm}
\begin{ytableau}
*(blue) & *(blue) & 1 & 1 & 1 & 1\\
\none & *(blue) & 2
\end{ytableau}
\]

Now suppose that $\nu=(n-3)$.
Then there exists a skew shifted Young tableau of shape $\lambda/(2,1)$ with content $(n-3)$.
Again, conditions (R1)--(R3) imply that $\lambda=(n-2,2)$.
The third diagram implies that $f_{(2,1)(n-4,1)}^{\lambda}>0.$
This completes the proof.
\end{proof}

\subsection{Spectra in the Reduced Clifford Product}
For any operator $g$ on a vector space $V$, let $\Sp_V(g)$ denote the set of eigenvalues of $g$ in $V$.
When $V=V_\lambda$ for $\lambda\in\DP_n$, we simply write $\Sp_\lambda(g)$ for $\Sp_{V_\lambda}(g)$.
For $\pi\in \tilde S_n$, define $d(\pi)=0$ if $\pi$ is even and $d(\pi)=1$ if $\pi$ is odd.
\begin{lemma}\label{lemma:spectrum_module_with_associator}
Let $U$ be a self-associate irreducible representation of $\tilde S_n$ and $T\in GL(U)$ be an associator for $U$. For each $\pi\in \tilde{S}_n$,
\begin{enumerate}
\item If $d(\pi)=0$, then $\Sp_U(\pi T)\subseteq\Sp_U(\pi)\cup-\Sp_U(\pi)$.
Moreover for each $\eta\in \Sp_U(\pi)$, at least one of $\pm\eta$ belongs to $\Sp_U(\pi T)$.
\item If $d(\pi)=1$, then $\Sp_U(\pi T)=\{\pm i\}\times \Sp_U(\pi)$.
\end{enumerate}
\end{lemma}
\begin{proof}
	Let $U = E_1(T) \oplus E_{-1}(T)$ be the decomposition of $U$ into the eigenspaces of $T$ corresponding to the eigenvalues $1$ and $-1$, respectively.
	Suppose that $d(\pi)=0$.
	Then $\pi T= T \pi$.
	Therefore, the eigenspaces of $T$ are invariant subspaces for $\pi$.
	Hence, we have the following decomposition for $U$ as a $\langle \pi \rangle$-module:
	\begin{align*}
		U &= \bigoplus_{\eta \in\Sp_U(\pi)} (E_1(T) \cap E_\eta(\pi)) \oplus (E_{-1}(T) \cap E_\eta(\pi)).
	\end{align*}
	Hence, the eigenvalues of $\pi T$ are $\pm \eta$, where $\eta$ runs through the eigenvalues of $\pi$ in $U$.

	Suppose that $d(\pi)=1$.
	Then $\pi T = - T \pi$.
	Let $u\in U$ be an eigenvector of $\pi$ corresponding to the eigenvalue $\eta$. It is easy to see that $Su$ is an eigenvector of $\pi$ corresponding to the eigenvalue $-\eta$, so $u$ and $T u$ are linearly independent vectors. Therefore, we find that the subspace $W$ generated by $u$ and $Tu$ is an invariant subspace for $\pi$ and $T$.
	With respect to the basis $\{u, Tu\}$ of $W$, we have
	\begin{align*}
		\pi T|_W &= \begin{pmatrix} 0 & \eta \\ -\eta & 0 \end{pmatrix}.
	\end{align*}
	Therefore, the eigenvalues of $\pi T$ on $W$ are $\pm i \eta$.
	This completes the proof of the lemma.
\end{proof}


\begin{lemma}\label{lemma:spectra_of_product_of_two_modules}
Suppose that $V_1$ and $V_2$ are irreducible spin representations for $\tilde S_m$ and
$\tilde S_k$ respectively, with $m+k=n$, $\pi_1\in\tilde S_m$ and $\pi_2\in\tilde S_k$.
Let $d_1=d(\pi_1)$, and $d_2=d(\pi_2)$.
Then $\Sp_{V_1\otimes_c V_2}(\pi_1\pi_2)$ is given by Table~\ref{table:product_of_two}, where $T_1$ is an associator for $V_1$.
If $V_1$ is self-associate but $V_2$ is not, then $\Sp_{V_1\otimes_c V_2'}(\pi_1\pi_2)=\Sp_{V_1\otimes_c V_2}(\pi_1\pi_2)$ unless $(d_1,d_2)=(0,1)$, in which case $\Sp_{(V_1\otimes_c V_2)\oplus (V_1\otimes_c V_2')}=\{\pm 1\}\times\Sp_{V_1}(\pi_1)\times\Sp_{V_2}(\pi_2)$.
\end{lemma}
\begin{table}[h]
\centering
\renewcommand{\arraystretch}{1.2}
\scalebox{0.92}{
\begin{tabular}{|c|ccc|}
\hline
  & \text{$V_1$ self-associate} &   \text{$V_1$ not self-associate} &  \text{$V_1$ self-associate} \\
$(d_1,d_2)$  & \text{$V_2$ self-associate} &  \text{$V_2$ not self-associate} &  \text{$V_2$ not self-associate}\\ \hline
$(0,0)$ & $\Sp_{V_1}(\pi_1)\times\Sp_{V_2}(\pi_2)$ & $\Sp_{V_1}(\pi_1)\times\Sp_{V_2}(\pi_2)$ & $\Sp_{V_1}(\pi_1)\times\Sp_{V_2}(\pi_2)$\\
$(1,0)$ & $\Sp_{V_1}(\pi_1)\times\Sp_{V_2}(\pi_2)$  & $\{\pm1\}\times\Sp_{V_1}(\pi_1)\times\Sp_{V_2}(\pi_2)$ &  $\Sp_{V_1}(\pi_1)\times\Sp_{V_2}(\pi_2)$\\
$(0,1)$ & $\Sp_{V_1}(\pi_1)\times\Sp_{V_2}(\pi_2)$ & $\{\pm1\}\times\Sp_{V_1}(\pi_1)\times\Sp_{V_2}(\pi_2)$ & $\Sp_{V_1}(\pi_1T_1)\times\Sp_{V_2}(\pi_2)$\\
$(1,1)$ & $\{\pm i\} \times \Sp_{V_1}(\pi_1)\times\Sp_{V_2}(\pi_2)$  & $\{\pm{i}\}\times\Sp_{V_1}(\pi_1)\times\Sp_{V_2}(\pi_2)$ & $\{\pm{i}\}\times\Sp_{V_1}(\pi_1)\times\Sp_{V_2}(\pi_2)$\\
\hline
\end{tabular}
}
\caption{$\Sp_{V_1\otimes_c V_2}(\pi_1\pi_2)$}\label{table:product_of_two}
\end{table}
\begin{proof}
Suppose that $V_1$ and $V_2$ are self-associate.
Fix associators $T_1\in GL(V_1)$ and $T_2\in GL(V_2)$.
By the definition of the reduced Clifford product, $(\pi_1\pi_2)(x\otimes u\otimes v)=x\otimes \pi_1T_1^{d_2}(u)\otimes\pi_2(v)$. If $d_2=0$, then
this formula simplifies to $(\pi_1\pi_2)(x\otimes u\otimes v)=x\otimes \pi_1(u)\otimes\pi_2(v)$,
so we get that $\Sp_W(\pi_1\pi_2)=\Sp_{V_1}(\pi_1)\times\Sp_{V_2}(\pi_2)$. If $(d_1,d_2)=(0,1)$,
we can swap the roles of $V_1$ and $V_2$ by looking at the representation
$(\pi_1\pi_2)(x\otimes u\otimes v)=x\otimes \pi_1(u)\otimes T_2^{d_1}\pi_2(v)$.
By Proposition~\ref{prop:reduced_Clifford_product}, this representation has the same character as the initial one. Therefore, once again $\Sp_W(\pi_1\pi_2)=\Sp_{V_1}(\pi_1)\times\Sp_{V_2}(\pi_2)$. Finally, if $(d_1,d_2)=(1,1)$, the spectrum of $\pi_1T_1$ is equal to $\{\pm i\}\times\Sp_{V_1}(\pi_1)$ by Lemma~\ref{lemma:spectrum_module_with_associator}, which completes the proof in this case.

Suppose that $V_1$ is self-associated and $V_2$ is not.
The Clifford algebra $\mathcal{C}_1$ generated by $\xi_1$, and we may assume that $\xi_1\mapsto 1$ for the $\mathcal{C}_1$-module $V$. We also fix an associator $T_1\in GL(V_1)$ for $V_1$. By definition of $V_1\otimes_cV_2$  (or $V_1\otimes_cV_2'$), we
 see that $(\pi_1\pi_2)(x\otimes u\otimes v)=x\otimes \pi_1T_1^{d_2}(u)\otimes\pi_2(v)$.
 This implies that if $d_2=0$, then
 $\Sp_{V_1\otimes_c V_2}(\pi_1\pi_2)=\Sp_{V_1}(\pi_1)\times\Sp_{V_2}(\pi_2)$.
If $d_2=1$, then
 $(\pi_1\pi_2)(x\otimes u\otimes v)=x\otimes\pi_1T_1(u)\otimes\pi_2(v)$. 
 Now Lemma~\ref{lemma:spectrum_module_with_associator} implies that 
 $\Sp_{V_1\otimes_c V_2}(\pi_1\pi_2)=\{\pm i\}\times\Sp_{V_1}(\pi_1)\times\Sp_{V_2}(\pi_2)$ if $d_1=1$
 and $\Sp_{(V_1\otimes_c V_2)\oplus(V_1\otimes_c V_2')}(\pi_1\pi_2)=\{\pm 1\}\times\Sp_{V_1}(\pi_1)\times\Sp_{V_2}(\pi_2)$ if $d_1=0$.

Suppose that $V_1$ and $V_2$ are non self-associate.
We may assume that the $\mathcal{C}_2$-module $V$ is given by
$$\xi_1\mapsto\left(\begin{matrix} 0 & 1 \\ 1 & 0\end{matrix}\right), \xi_2\mapsto\left(\begin{matrix} 0 & i \\ -i & 0\end{matrix}\right),$$ so
 $\xi_1\xi_2\mapsto\left(\begin{matrix} -i & 0 \\ 0 & i\end{matrix}\right)$. By the definition of $V_1\otimes_c V_2$, we
 see that $(\pi_1\pi_2)(x\otimes u\otimes v)=\xi_1^{d_1}\xi_2^{d_2}x\otimes \pi_1(u)\otimes\pi_2(v)$.
 Then $\Sp_{V_1\otimes_c V_2}(\pi_1\pi_2)=\Sp(\xi_1^{d_1}\xi_2^{d_2})\times \Sp_{V_1}(\pi_1)\times\Sp_{V_2}(\pi_2)$.
 The spectra of $\xi_1$, $\xi_2$, $\xi_1\xi_2$ are $\{\pm1\}$, $\{\pm1\}$, and $\{\pm i\}$, respectively. Therefore, we get the spectrum as in the corresponding cell in Table~\ref{table:product_of_two}.
 \end{proof}

\begin{remark}
    Since $V_1\otimes_c V_2$ and $V_2\otimes_c V_1$ correspond to equivalent representations of $\tilde{S}_m\times_c\tilde{S}_k$, the statement of Lemma~\ref{lemma:spectra_of_product_of_two_modules} covers all cases.
\end{remark}
The following result will be useful on many occasions in the upcoming proofs.
For any positive integer $n$, define $\Omega(n,\pm 1) = \{x \in \mathbb{C} : x^n = \pm 1\}$.
\begin{proposition}\label{proposition:product_of_two_sets_cardiality_in_complex_numbers}
Let $a,b$ be positive integers and $c,d \in \{\pm 1\}$. 
Then $\Omega(a,c) \times  \Omega(b,d)$ has exactly $\operatorname{lcm}(a,b)$ elements.
\end{proposition}
\begin{proof}
Fix $\zeta, \xi \in \mathbb{C}$ such that $\zeta^a = c$ and $\xi^b = d$.
Let $l = \lcm(a,b)$.
Then it is easy to see that $\Omega(a,c) = \{\zeta\} \times \Omega(a,1)$ and similarly, $\Omega(b,d) = \{\xi\} \times \Omega(b,1)$.

It is well-known that $\Omega(a,1) \times \Omega(b,1) = \Omega(l,1)$. Therefore
\begin{displaymath}
\Omega(a,c) \times \Omega(b,d) = \{\zeta\} \times \Omega(a,1) \times \{\xi\} \times \Omega(b,1) = \{\zeta\xi\} \times \Omega(l,1).
\end{displaymath}
Hence, $\Omega(a,c) \times \Omega(b,d)$ has exactly $l$-many elements.
\end{proof}
 
 Recall that, for $g\in G$, $o(g)$ denotes the order of the image of $g$ in $G/Z(G)$.
 \begin{corollary}\label{corollary:product_of_two}
Suppose that $V_1$ and $V_2$ are irreducible spin representations for $\tilde S_m$ and
$\tilde S_k$ respectively, with $m+k=n$, $\pi_1\in\tilde S_m$ and $\pi_2\in\tilde S_k$ are such that $|\Sp_{V_i}(\pi_i)|=o(\pi_i)$ for $i=1,2$, then $|\Sp_{V_1\otimes_c V_2}(\pi_1\pi_2)|=o(\pi_1\pi_2)$.
\end{corollary}
\begin{proof}
Suppose that $\pi_i$ has $o(\pi_i)$ distinct eigenvalues in $V_i$ for $i=1,2$. By Lemma~\ref{lemma:spectra_of_product_of_two_modules}, we know that either 
$\Sp_{V_1\otimes_cV_2}(\pi_1\pi_2)=M\times\Sp_{V_1}(\pi_1)\times \Sp_{V_2}(\pi_2)$, where $M\in\{\{1\},\{\pm1\},\{\pm i\}\}$, or, up to the reordering of $\pi_1$ and $\pi_2$,  $\Sp_{V_1\otimes_cV_{2}}(\pi_1\pi_2)=\Sp_{V_1}(\pi_1T)\times\Sp_{V_{2}}(\pi_{2})$, where 
$T$ is an associator for $V_1$. In the latter case, $\pi_2$ is odd and hence $o(\pi_2)$ is even. Then $\Sp_{V_2}(\pi_2)=\{\pm1\}\times \Sp_{V_2}(\pi_2)$. 
By Lemma~\ref{lemma:spectrum_module_with_associator}(1), we get that $\Sp(\pi_1 T) \times \{\pm 1\} = \Sp(\pi_1) \times \{\pm 1\}$.
Therefore, 
it suffices to show that $|\Sp_{V_1}(\pi_1)\times\Sp_{V_2}(\pi_2)|=o(\pi_1\pi_2)$ to prove the statement in all cases.
But this already follows from Proposition~\ref{proposition:product_of_two_sets_cardiality_in_complex_numbers}.
\end{proof}

\section{Powers of Conjugacy Classes in $\tilde S_n$}\label{sec:conju-classes-powers}
In this section, we determine the power map on conjugacy classes in $\tilde S_n$.
We begin by establishing the orders of certain elements that lie above single cycles:
\begin{theorem}\label{c:value_of_powers}
For $a\geq 1$ and $n\geq 2$, let $\pi_n=t_at_{a+1}\ldots t_{a+n-2}\in \tilde S_m$.
Then $\pi_n^{n}=(-1)^{\lfloor(n+2)/4\rfloor}$. 
\end{theorem}
Without loss of generality, assume that $a=1$.
The theorem will be established with the help of the following recursive result:
\begin{lemma}\label{l:sign_of_powers_of_cycles}
\[
(t_1 t_2 \cdots t_n)^{n+1}
=
(-1)^{\frac{n(n+1)(n+2)}6}\,(t_1 t_2 \cdots t_{n-1})^{n}.
\]
\end{lemma}

\begin{proof}
Set
$
u=t_1t_2\cdots t_n,$ $ v=t_1t_2\cdots t_{n-1}.
$
We need to show that
$
u^{n+1}=(-1)^{\frac{n(n+1)(n+2)}6}\,v^n.
$
Since $u=vt_n$, we have
\[
u^{n+1}=v\,t_n\,u^n.
\]
The strategy is to repeatedly transform the expression $v t_n u^n$ by moving one factor of $u$ from the right to the left, converting it into an extra factor of $v$, while keeping track of the powers of $(-1)$ that appear.
More precisely, we shall prove the following intermediate identity: for every $0\le j\le n-2$,
\begin{equation}\label{eq:step}
v^{\,j+1}\,t_n t_{n-1}\cdots t_{n-j}\,u^{\,n-j}
=
(-1)^{(j+1)(n-j)}\,
v^{\,j+2}\,t_n t_{n-1}\cdots t_{n-j}t_{n-j-1}\,u^{\,n-j-1}.
\end{equation}

Indeed, write $u^{n-j}=t_1t_2\cdots t_n\,u^{n-j-1}$. Then

\begin{align}
v^{j+1}\,t_n\cdots t_{n-j}\,u^{n-j} 
&= v^{j+1}\,t_n\cdots t_{n-j}\,(t_1t_2\cdots t_n)\,u^{n-j-1}.
\end{align}

Moving the block $t_n\cdots t_{n-j}$ past $t_1\cdots t_{n-j-2}$  gives a factor $(-1)^{(j+1)(n-j-2)}.$
Hence we get
\begin{align}
\label{eq:order_n_first}
v^{j+1}t_n\cdots t_{n-j}(t_1t_2\cdots t_n)u^{n-j-1}
&= (-1)^{(j+1)(n-j-2)}
   v^{j+1}(t_1\cdots t_{n-j-2}) \notag\\
&\qquad
   (t_n\cdots t_{n-j+1})\,  (t_{n-j}t_{n-j-1}t_{n-j}) \,
   (t_{n-j+1}\cdots t_n\,u^{n-j-1}).
\end{align}
The only nontrivial part of the right-hand side is the conjugated braid block
\[
w_j:=(t_n\cdots t_{n-j+1})
(t_{n-j}t_{n-j-1}t_{n-j})
(t_{n-j+1}\cdots t_n).
\]
Write $$x_j = t_n\cdots t_{n-j+1} \text{ and } y_j = t_{n-j+1}\cdots t_n, \text{ so that } w_j = x_j (t_{n-j} t_{n-j-1} t_{n-j}) y_j .$$

Using the braid relation $t_it_{i+1}t_i=t_{i+1}t_it_{i+1}$ together with the anti-commutation relations $t_it_k=-t_kt_i$ whenever $|i-k|\ge2$, we obtain the recursion
\begin{align*}
    w_j &= x_j \, (t_{n-j} t_{n-j-1} t_{n-j}) \, y_j 
        =  x_j \, (t_{n-j-1} t_{n-j} t_{n-j-1}) \, y_j\\
        &=  t_{n-j-1}\, (x_j t_{n-j} y_j) \, t_{n-j-1}
        =  t_{n-j-1}\, w_{j-1} \,t_{n-j-1}
\end{align*}

Now replacing $j$ with $j-1,j-2, \dots ,1$, we obtain
\begin{align*}
    w_j & =  (t_{n-j-1} t_{n-j-2} \dots t_{n-2})\, (t_n t_{n-1} t_{n}) \, (t_{n-2} \dots t_{n-j-2} t_{n-j-1})\\
        & =  (t_{n-j-1} t_{n-j-2} \dots t_{n-2})\, (t_{n-1} t_{n} t_{n-1}) \, (t_{n-2} \dots t_{n-j-2} t_{n-j-1})\\
        &=  (t_{n-j-1} \dots t_{n-1}) (t_n \dots t_{n-j-1})
\end{align*}

Substituting the above expression for $w_j$ into Eq.~\eqref{eq:order_n_first} gives
\begin{align*}
    v^{j+1} t_n t_{n-1} \dots t_{n-j} u^{n-j} = (-1)^{(j+1)(n-j-2)} v^{j+2} t_n \dots t_{n-j-1} v^{n-j-1},
\end{align*}
which gives Eq.~\eqref{eq:step}.
Applying Eq.~\eqref{eq:step} successively for $j=0,1,\ldots,n-2$ moves every remaining factor of $u$ to the left, producing
\begin{align*}
    u^n &= (-1)^k v^{n} t_n t_{n-1} \dots t_1 u = (-1)^{k+n} v^n,
\end{align*}
where $k=\sum\limits_{j=0}^{n-2} (j+1)(n-j)$. Therefore, $k+n = \frac{n(n+1)(n+2)}6$.

Hence, $u^n = (-1)^{\frac{n(n+1)(n+2)}6} v^n$. This completes the proof of the lemma.
\end{proof}

\begin{proof}[Proof of Theorem~\ref{c:value_of_powers}]
It is easy to see that $\frac{n(n+1)(n+2)}6$ is odd if and only if $n\equiv1\pmod 4$.
Now $\pi_2^2=t_a^2=-1$. 
By Lemma~\ref{l:sign_of_powers_of_cycles}, the number of times the sign of $\pi_n^n$ changes between $2$ and $n$ is the number of integers $2<i<n$ such that $i\equiv 1\mod 4$, which is $\lfloor\frac{n-2}4\rfloor$.
Hence $\pi_n^n = -(-1)^{\lfloor(n-2)/4\rfloor} = (-1)^{\lfloor(n+2)/4\rfloor}$.
\end{proof}

This theorem helps to find the order of a preimage in $\tilde S_n$
and to understand in which conjugacy classes the powers of a given $n$-cycle lie for $n$ odd.

\begin{example}
Consider $\sigma=\sigma_{(15)}^+=t_1t_2\ldots t_{14}\in\tilde S_{15}$, a preimage of the cycle $(1,2,\ldots,15)$. By Theorem~\ref{c:value_of_powers}, $\sigma^{15}=1$, so $(-\sigma_{15})^{15}=-1$.
Hence $\sigma^i$ is conjugate to $\sigma$ if $i$ is coprime to $15$.
Now $(\sigma^3)^{5}=1$ and $\sigma^3$ is conjugate to either $\sigma_{(5,5,5)}^+$ or $\sigma_{(5,5,5)}^-$. Since $(t_1t_2t_3t_4)^5=-1$, we find that $(\sigma_{(5,5,5)}^+)^5=-1$, so $\sigma^3$ is conjugate to $\sigma_{(5,5,5)}^-$.
Similarly, since $(t_1t_2)^3=-1$, we see that $\sigma^5$ is conjugate to
$\sigma_{(3,3,3,3,3)}^-$.
\end{example}


Now we shall consider the $n$-cycles in $\tilde S_n$ when $n$ is even.
To solve the conjugacy problem for powers of $
\sigma_{(n)}$, we need to use quadratic Gauss sums.

\begin{definition}
If $n$ and $k$ are integer, where $n>0$, then 
the sum $G(k;n)=\sum\limits_{r=1}^ne^\frac{2\pi ikr^2}{n}$
is called a \emph{quadratic Gauss sum}.
\end{definition}

We collect basic properties of $G(k;n)$ in the following lemma.

\begin{lemma}\label{l:gauss_sum_properties}
The following statements are true.
\begin{enumerate}
\item $G(k; mn)=G(km;n)G(kn;m)$ whenever $(m,n)=1$. This reduces the study of quadratic Gauss sums to the special case $G(k; p^a)$, where $p$ is prime. 
\item If $p$ is an odd prime, $k$ is coprime to $p$,
and $a\geq2$, then 
$$G(k; p^a)=\begin{cases}
 p^{a/2}, \text{if }a\text{ is even}, \\
 p^{(a-1)/2}G(k;p), \text{if }a\text{ is odd}. 
\end{cases}$$
\item If $p$ is an odd prime and $n$ is coprime to $p$, then $G(n;p)=\jacobi{n}{p}G(1;p)$.
\item If $k$ is odd, then $G(k;2)=0$, $G(k;4)=2+2i^k$, 
$G(k;8)=4e^{\pi ik/4}$.
\item If $m$ is a positive integer, then
$$G(1;m)=\frac{1}{2}\sqrt{m}(1+i)(1+e^{-\pi im/2})=\begin{cases}
\sqrt{m}, \text{ if }m\equiv1\pmod{4},\\
0, \text{if }m\equiv2\pmod{4},\\
i\sqrt{m}, \text{ if }m\equiv3\pmod{4},\\
(1+i)\sqrt{m}, \text{ if }m\equiv0\pmod{4}.
\end{cases}$$
\end{enumerate}
\end{lemma}
\begin{proof}
Items 1) and 2) are \cite[Exercise 8.16]{Apostol},
while items 3) and 5) are formulas (29) and (30) in \cite[9.10]{Apostol}, respectively. 
Finally, item (4) can be verified by hands.
\end{proof}

\begin{corollary}\label{c:gauss_sum_odd}
If $n$ and $k$ are coprime positive integers, where $n$ is odd, then $G(k;n)=\jacobi{k}{n}G(1;n)$, where $\jacobi{k}{n}$ means the Jacobi symbol for $k$ and $n$.
\end{corollary}
\begin{proof}
We prove the statement by induction on $n$. 
First, we consider the case when $n$ is a prime power.
If $n=p$, then the assertion is true by Lemma~\ref{l:gauss_sum_properties}(3). If $n=p^a$, where $a\geq2$ and $a$ is even, then Lemma~\ref{l:gauss_sum_properties}(2) implies that $G(k;p^a)=G(1;p^a)=\jacobi{k}{n}G(1;p^a)$.
If $n=p^a$, where $a\geq2$ and $a$ is odd, then Lemma~\ref{l:gauss_sum_properties}(2) implies that $G(k;p^a)=p^{(\alpha-1)/2}G(k;p)=
p^{(\alpha-1)/2}\jacobi{k}{p}G(1;p)$.
Applying this to $k=1$, we get that $G(k;p^\alpha)=\jacobi{k}{p}G(1;p^a)=\jacobi{k}{p^a}G(1;p^a)$, as required.
Suppose now that $n=p^am$, where $p$ is a prime, 
$m>1$ and $a$ are positive integers, and $(m,p)=1$.
Using Lemma~\ref{l:gauss_sum_properties}(1), we get that
\begin{align*}
G(k;n)=G(k;p^am)=G(kp^a;m)G(km; p^a)&=\jacobi{kp^a}{m}\jacobi{km}{p^a}G(1;m)G(1;p^a)\\&=\jacobi{k}{n}\jacobi{p^a}{m}\jacobi{m}{p^a}G(1;m)G(1;p^a).
\end{align*}
Applying this to $k=1$, we infer that $G(k;n)=\jacobi{k}{n}G(1;n)$.
\end{proof}

The following lemma lists well-known properties of the Jacobi symbol.
\begin{lemma}~\label{lemma:jacobi_basic_facts}
Let $m$ and $k$ be odd integers. Then the following statement holds.
    \begin{enumerate}
        \item $\jacobi{-1}{k} = (-1)^{(k-1)/2}$;
        \item $\jacobi{2}{k} = (-1)^{(k^2-1)/8}$;
        \item $\jacobi{m}{k} = (-1)^{\big(\tfrac{(m-1)(k-1)}{4}\big)} \jacobi{k}{m}$ (the law of quadratic reciprocity);
        \item If $m\equiv 3 \, (\text{mod }4)$, then $\jacobi{k}{m} = \jacobi{-m}{k} $;
        \item If $k \equiv 1 \, (\text{mod }4)$, then $\jacobi{2}{k} = (-1)^{(k-1)/2}$;
        \item If $k \equiv 3 \, (\text{mod }4)$, then $\jacobi{2}{k} = (-1)^{(k+1)/2}$;
        \end{enumerate}
\end{lemma}
\begin{lemma}~\label{lemma:tau_k_on_sqrt_2}
Let $N,k$ be positive odd integers with $(N,k)=1$ and $8 \mid N$.
Let $\eta=e^{2\pi i/N}$, and let
$\tau\in\Gal(\mathbb{Q}(\eta)/\mathbb{Q})$ be the unique automorphism satisfying
$\tau(\eta)=\eta^k$.
Then
$
\tau(\sqrt{2})=\jacobi{2}{k}\sqrt{2}.
$
\end{lemma}

\begin{proof}
Since $8\mid N$, we have $\zeta_8:=e^{2\pi i/8}=\eta^{N/8}\in\mathbb{Q}(\eta)$ and hence
$\tau(\zeta_8)=\zeta_8^k.$
Using the equality $\zeta_8^k+\zeta_8^{-k}=2\cos\!\left(\frac{k\pi}{4}\right),$ and $k$ is odd, we get that
\[
2\cos\!\left(\frac{k\pi}{4}\right)
=
\begin{cases}
\sqrt{2}, & k\equiv1,7\pmod8,\\
-\sqrt{2}, & k\equiv3,5\pmod8.
\end{cases}
\]
On the other hand, Lemma~\ref{lemma:jacobi_basic_facts} implies that
\[
\jacobi{2}{k}
=
\begin{cases}
1, & k\equiv1,7\pmod8,\\
-1, & k\equiv3,5\pmod8.
\end{cases}
\]
Therefore,
\[
\tau(\sqrt{2})
=\tau(\zeta_8)+\tau(\zeta_8^{-1})
=\zeta_8^k+\zeta_8^{-k}=\jacobi{2}{k}\sqrt{2},
\]
as required.
\end{proof}
Let $n$ be even and write $n=2m$, then we define
$    D(n) :=(-1)^mm= \begin{cases}
        -m, & \text{ if } m \text{ is odd,}\\
        m,  & \text{otherwise.}
    \end{cases}
$

\begin{proposition}~\label{proposition:conjugacy_classes_of_powers_of_n_cycles_even}
Suppose that $n=2m$ and $k$ is a positive integer 
coprime to $n$.
Then $\sigma_{(n)}^k$ is conjugate to $\sigma_{(n)}$
if and only if $\jacobi{D(n)}{k}=1$.
\end{proposition}
\begin{proof}
    Recall that $\sigma_{(n)}^k$ is conjugate to $\sigma_{(n)}$ or $-\sigma_{(n)}$.
    Since $\chi := \phi_{(n)}^+$ does not vanish on either of these three elements,
    $\sigma_{(n)}^k$ is conjugate to $\sigma_{(n)}$ (resp. $-\sigma_{(n)}$) if and only if $\chi(\sigma_{(n)}^k)$ is equal to $\chi(\sigma_{(n)})$ (resp. $\chi(-\sigma_{(n)}^k)$).
    Notice that $\chi(-\sigma_{(n)}) = - \chi(\sigma_{(n)}) \neq \chi(\sigma_{(n)})$ due to $\chi(\sigma_{(n)}) \neq 0$.

    Let $N$ be the order of $\sigma_{(n)}$ and set $\zeta = e^{{2\pi i}/{N}}$.
    Since $n$ is even, $(k,n)=1$ if and only if $(k,N)=1$.
    Observe that $\chi(\sigma_{(n)}) = \zeta^{i_1} + \zeta^{i_2} + \cdots + \zeta^{i_r}$, where $i_1,\dots, i_r$ are nonnegative integers and $r=\chi(1)$.
    Therefore, $\chi(\sigma_{(n)}^k) = \zeta^{ki_1} + \cdots + \zeta^{ki_r}$.

    Let $\tau \in \Gal(\mathbb{Q}(\eta)/\mathbb{Q})$ such that $\tau(\eta) = \eta^k$.
    Then $\tau(\chi(\sigma_{(n)})) = \chi(\sigma_{(n)}^k)$. From Lemma~\ref{lemma:spin_character_values}, $\chi(\sigma_{(n)}) = i^m \sqrt{m}$, where $m = n/2$.
    Thus it suffices to check whether $\tau(\chi(\sigma_{(n)}))$ is equal to $i^m \sqrt{m}$ or $-i^m \sqrt{m}$.
    
    Suppose that $m$ is odd.
    If $m\equiv 1 (\text{mod } 4)$, then $\tau(\chi(\sigma_{(n)})) = \tau(i^m \sqrt{m}) = \tau(i \sqrt{m}).$
    Using Theorem~\ref{c:value_of_powers}, we get $4 \mid N$, which is equivalent to that $i$ is an $N$-th root of unity.
    At the same time, using Lemma~\ref{l:gauss_sum_properties}, we get that $\sqrt{m} = G(1;m)$.
    Hence, by Corollary~\ref{c:gauss_sum_odd} and Lemma~\ref{lemma:jacobi_basic_facts}, we obtain 
    \begin{align*}
    \tau(\sigma_{(n)}) 
    &= i^k \tau(G(1;m)) = i^k G(k;m) = (-1)^{(k-1)/2} i \jacobi{k}{m} \sqrt{m}\\
    &= \jacobi{-1}{k} \jacobi{m}{k}\chi(\sigma_{(n)}) = \jacobi{-m}{k} \chi(\sigma_{(n)})=\jacobi{D(n)}{k} \chi(\sigma_{(n)}).
    \end{align*}
    Hence, the lemma is proved in this case.

    If $m \equiv 3 (\text{mod } 4)$, then the same series of calculation yields the following:
    \begin{align*}
        \tau(\chi(\sigma_{(n)})) 
        &= \tau(i^m \sqrt{m}) = \tau(-i \sqrt{m}) = \tau(-G(1;m)) = -G(k;m) \\
        &= \jacobi{k}{m} (-G(1;m)) = \jacobi{k}{m} \chi(\sigma_{(n)}) = \jacobi{-m}{k} \chi(\sigma_{(n)}).
    \end{align*}

    Let us consider the case $m = 2^{2a}b$ with $b$ odd and $a\geq 1$.
    Then $i^m\sqrt{m}=\sqrt{m}$. Now we see that $\sqrt{m}=2^{a}\sqrt{b}$.
    If $b\equiv1\pmod 4$, then 
    \begin{displaymath}
        \tau_k(\sqrt{m})=\tau_k(2^{a})\tau_k(G(1;b))=2^aG(k;b)=2^a\jacobi{k}{b}G(1;b)=\jacobi{k}{b}\sqrt{m} = \jacobi{b}{k} \sqrt{m}.
    \end{displaymath}
    Since $\jacobi{m}{k} = \jacobi{b}{k}$, the result follows in this case.
    
    If $b\equiv3\pmod 4$, then
    \begin{displaymath}
        \tau_k(\sqrt{m})=2^a\tau_k(i^{-1}G(1;b))= 2^ai^{-k}G(k;b)=2^a(-i)^k\jacobi{k}{b}i\sqrt{b}=(-1)^{(k-1)/2}\jacobi{k}{b}\sqrt{m}.
    \end{displaymath}
    Using Lemma~\ref{lemma:jacobi_basic_facts}, we have $(-1)^{(k-1)/2} \jacobi{k}{b} = \jacobi{b}{k} = \jacobi{m}{k}$ and the result follows.

Suppose that $m=2^{2a+1}b$, where $b$ is odd and $a\ge1$.
Then $i^m\sqrt m=\sqrt m$, and
$
\sqrt m=2^a\sqrt2\,\sqrt b.
$
Hence
\[
\tau_k(\sqrt m)
=
2^a\tau_k(\sqrt2)\tau_k(\sqrt b).
\]

By Lemma~\ref{lemma:tau_k_on_sqrt_2}, $\tau_k(\sqrt2)=\jacobi{2}{k}\sqrt2.$

Using the proof of $m$ odd case, we get the following.

If $b\equiv1\pmod4$, then $\tau_k(\sqrt b)=\jacobi{b}{k}\sqrt b.$

If $b\equiv3\pmod4$, then $\tau_k(\sqrt b)=(-1)^{(k-1)/2}\jacobi{k}{b}\sqrt b=\jacobi{b}{k} \sqrt{b}.$

Since $\jacobi{2}{k} = \jacobi{2^{2a+1}}{k}$, we have $\jacobi{2}{k} \jacobi{b}{k} = \jacobi{2^{2a+1}}{k} \jacobi{b}{k}  = \jacobi{m}{k}$.
Hence,  $    \tau(\sqrt{m}) = \jacobi{m}{k} \sqrt{m}.$
This completes the proof in this case.

Finally, if $m=2b$, then $i^m\sqrt m=-\sqrt m$ and
$
\sqrt m=\sqrt2\,\sqrt b.
$
Notice that $8\mid N$ in this case by Theorem~\ref{c:value_of_powers}. Using Lemma~\ref{lemma:tau_k_on_sqrt_2}, we get $\tau_k(\sqrt{2}) = \jacobi{2}{k}\sqrt{2}$.
Now the above argument applies verbatim with $a=0$, completing the proof.      
\end{proof}
We record a simple observation.

\begin{lemma}\label{lemma:conjugate_by_an_even_permutation}
Let $n$ be even and let $k$ be a positive integer with $(k,n)=1$.
Then there exists $g\in \tilde A_n$ such that
\[
g\sigma_{(n)}^kg^{-1}
=
\jacobi{D(n)}{k}\sigma_{(n)}.
\]
\end{lemma}

\begin{proof}
Since $\pi(\sigma_{(n)})$ and $\pi(\sigma_{(n)}^k)$ are conjugate in $S_n$, there exists
$h\in S_n$ such that
$
h\,\pi(\sigma_{(n)}^k)\,h^{-1}
=
\pi(\sigma_{(n)}).
$
If $h$ is odd, replace it with $h':=h\pi(\sigma_{(n)})$.
Since $\pi(\sigma_{(n)})$ commutes with
$\pi(\sigma_{(n)}^k)$, we still have
$
h'\,\pi(\sigma_{(n)}^k)\,h'^{-1}
=
\pi(\sigma_{(n)}),
$  
and $h'$ is even. Thus, there exists an even permutation
$h\in A_n$ satisfying
$
h\,\pi(\sigma_{(n)}^k)\,h^{-1}
=
\pi(\sigma_{(n)}).
$
Let $g$ be one of the preimages of $h$ in $\tilde A_n$. By
Proposition~\ref{proposition:conjugacy_classes_of_powers_of_n_cycles_even},
$
g\sigma_{(n)}^kg^{-1}
=
\jacobi{D(n)}{k}\sigma_{(n)},
$
as required.
\end{proof}
We need the following well-known lemma.
\begin{lemma}~\label{lemmma:conjugacy_class_splitting_index_2}
Let $G$ be a group and $H \leq G$ a subgroup of index $2$. Let $x \in H$ and let $C$ denote the conjugacy class of $x$ in $G$ which splits as two conjugacy classes of $H$. Suppose that $g x g^{-1} = y$ for some $g\in G$.
Then $y$ is conjugate to $x$ in $H$ if and only if $g \in H$.
\end{lemma}
\begin{proof}
Let $C^+$ denote the conjugacy class of $x$ in $H$.
One direction is immediate. Indeed, if $g \in H$, then $y\in C^+$.
For the other direction, assume the contrary that $y\in C$ but 
$g \notin H$.
Since $y \in C^+$, there exists $h\in H$ such that $h y h^{-1} = x$.
Thus, $hg \in Z_{G}(x)$. 
Since $hg \notin H$, $Z_{H}(x)$ is an index $2$ subgroup of $Z_{G}(x)$.
Therefore, $|C^+| = \frac{|G|}{|Z_{G}(x)|} = \frac{2|H|}{2|Z_{H}(x)|} = \frac{|H|}{|Z_{H}(x)|} = |C|$, a contradiction to $C$ splits.
This completes the proof of the claim.
\end{proof}
Let us explain how to compute the minimal polynomial of an element $g\in \tilde S_n$.
The main focus is on determining the signs that need to be attached to powers of a given element.

Recall that $\sigma_\lambda$ and $-\sigma_\lambda$ are not conjugate if and only if either $\lambda$ has only odd parts, or $\lambda$ has distinct parts with an odd number of even parts.
Using Lemma~\ref{lemma:spin_character_values}, we can find the character values of all irreducible spin representations at all $\sigma_\lambda^\pm$.
Therefore, in order to calculate the minimal polynomial of $\sigma_\lambda$ in an irreducible spin representation, we need to determine whether $\sigma_\lambda^i$ is conjugate to $\sigma_{\lambda^i}$ or $-\sigma_{\lambda^i}$ (here $\lambda^i$ is the cycle type of the $i$-th power of a permutation with cycle type $\lambda$).

We start with computing the order of $\sigma_\lambda$ in general. For a partition $\lambda$, let $e(\lambda)$ denote the number of even parts of $\lambda$.
\begin{lemma}
	\label{lemma:powers_of_pi_lambda}
	Let $\lambda=(\lambda_1,\dotsc,\lambda_k)$ be a partition of $n$ with length $k$.
	Then for a positive integer~$i$, 
	\begin{displaymath}
		\sigma_\lambda^i = (-1)^{\binom{i}{2}\binom{e(\lambda)}{2}}\sigma_{1}^i\sigma_{2}^i\cdots \sigma_{k}^i,
	\end{displaymath}
	where $\sigma_j= t_{a+1}t_{a+2}\cdots t_{a+\lambda_j-1}$ and $a=\lambda_1+\ldots+\lambda_{j-1}$ for each $j$.

	In particular, $\sigma_\lambda^{\lcm(\lambda)} = (-1)^{\binom{\lcm(\lambda)}{2}\binom{e(\lambda)}{2}}\prod\limits_{j=1}^k(-1)^{\lfloor(\lambda_j+2)/4\rfloor\lcm(\lambda)/\lambda_j}$.
\end{lemma}
\begin{proof}
	We prove the statement by induction on the length $k$ of $\lambda$.
	For $k=1$, the statement reads $\sigma_{(\lambda_1)}^i = \sigma_{1}^i$, which is true.
	Suppose that $k>1$ and the statement is true for all partitions of length less than $k$.
	Then we can write $\sigma_\lambda = \sigma_1\sigma_2\cdots \sigma_k$, where $\sigma_j = t_{a+1}t_{a+2}\cdots t_{a+\lambda_j-1}$, with $a=\lambda_1+\ldots+\lambda_{j-1}$.
    Denote $\lambda_j-1$ by $l_j$ for each $j$.
	Then
	\begin{align*}
		\sigma_\lambda^i &= (\sigma_1\cdots \sigma_k)^i\\
		&= (-1)^{l_1(l_2+\dotsb+l_k)(1+2+\dotsb+(i-1))}\sigma_1^i(\sigma_2\dotsc\sigma_k)^i\\
		&= (-1)^{\binom{i}{2}l_1(l_2+\dotsb+l_k)}\sigma_1^i(\sigma_2\cdots \sigma_k)^i,
	\end{align*}
	since the $j$-th $\sigma_1$, consisting of $l_1$ letters, has been moved past $(l_2+\dotsb+l_k)(j-1)$ letters.
	By induction on $k$, we get
	\begin{align*}
		(\sigma_1\cdots \sigma_k)^i = (-1)^{\binom{i}{2}(l_1(l_{2}+\dotsb + l_k) + \binom{e(\tilde\lambda)}{2})}\sigma_1^i\sigma_2^i\cdots \sigma_k^i = (-1)^{\binom{i}{2}\binom{e(\lambda)}{2}}\sigma_1^i\sigma_2^i\cdots \sigma_k^i,
	\end{align*}
    where $\tilde \lambda = (\lambda_2,\dots,\lambda_k)$.

    The latter statement follows from the former statement and Theorem~\ref{c:value_of_powers}.
\end{proof}

From Theorem~\ref{c:value_of_powers}, we know the orders of $\sigma_{(n)}$, and thus, by Lemma~\ref{lemma:powers_of_pi_lambda}, we can determine the order of $\sigma_\lambda$ for any partition $\lambda$.

For a partition $\lambda\vdash n$, let $\alpha=\lambda^i$ denote the cycle type of $\sigma_\lambda^i$.
Then $\sigma_\lambda^i$ is conjugate to $\sigma_\alpha$, $-\sigma_\alpha$ or both $\sigma_\alpha$ and $-\sigma_\alpha$.

If $\alpha \notin\OP_n\cup \DP_n$, then $\sigma_\alpha$ and $-\sigma_\alpha$ are conjugate, so the conjugacy class of $\sigma_\lambda^i$ contains both $\sigma_\alpha$ and $-\sigma_\alpha$.
In this case, all spin characters vanish at $\sigma_\lambda^i$ by Lemma~\ref{lemma:spin_character_values}.

Suppose $\alpha \in\OP_n$.
Then the order of $\sigma_\alpha$ can be determined by Lemma~\ref{lemma:powers_of_pi_lambda}, and the order of $\sigma_\lambda^i$ is equal to the order $N$ of $\sigma_\lambda$ divided by $(i,N)$.
With these two information, we can determine whether $\sigma_{\lambda^i}$ is conjugate to $\sigma_\alpha$ or $-\sigma_\alpha$.


Now we shall consider, more generally, $\alpha \in \DP_n$.
Assume that $(i,N)=1$.
From Lemma~\ref{lemma:powers_of_pi_lambda}, we have
\begin{align}\label{eq:sigma_lambda^i_exprsession_in_sigma_la^i}
	\sigma_\lambda^i = (-1)^{\binom{i}{2}\binom{e(\lambda)}{2}}\sigma_{\lambda_1}^i\sigma_{\lambda_2}^i\cdots \sigma_{\lambda_k}^i.
\end{align}
Since $i$ is coprime to the order of $\sigma_{\lambda_j}$, $\sigma_{\lambda_j}^i$ is conjugate to $\sigma_{\lambda_j}$ or $-\sigma_{\lambda_j}$ for each $j$.

Suppose that $\lambda_j$ is even.
Then by Lemma~\ref{lemma:conjugate_by_an_even_permutation}, there exists $g_j\in \tilde A_{\lambda_j}$ such that $g_j\sigma_{\lambda_j}^ig_j^{-1} = \jacobi{D(\lambda_j)}{j}\sigma_{\lambda_j}$.
Therefore,
\begin{align*}
	g_j \sigma_{\lambda_1}^i\cdots \sigma_{\lambda_j}^i\cdots \sigma_{\lambda_k}^i g_j^{-1} 
	&= \sigma_{\lambda_1}^i\cdots g_j\sigma_{\lambda_j}^ig_j^{-1}\cdots \sigma_{\lambda_k}^i\\
	&= \jacobi{D(\lambda_j)}{i} \, \sigma_{\lambda_1}^i\cdots \sigma_{\lambda_j}\cdots \sigma_{\lambda_k}^i.
\end{align*}
Suppose that $\lambda_j$ is odd. Then $\sigma_{\lambda_j}^i$ is conjugate to $\sigma_{\lambda_j}.$
Let $g_j\in \tilde S_{\lambda_j}$ such that $g_j \sigma_{\lambda_j}^i g_j^{-1} = \sigma_{\lambda_j}$.
It follows from~\cite[Theorem 2.5]{p2024cyclic} and Lemma~\ref{lemmma:conjugacy_class_splitting_index_2} that the parity of $g_j$ is equal to the Jacobi symbol $\jacobi{i}{\lambda_j}$.
Therefore, for $\lambda\in \DP_n^-$, we obtain that
\begin{align*}
	g_j \sigma_{\lambda_1}^i\cdots \sigma_{\lambda_j}^i\cdots \sigma_{\lambda_k}^i g_j^{-1} 
	&= \jacobi{i}{\lambda_j} \sigma_{\lambda_1}^i\cdots g_j\sigma_{\lambda_j}^ig_j^{-1}\cdots \sigma_{\lambda_k}^i\\
	&= \jacobi{i}{\lambda_j} \sigma_{\lambda_1}^i\cdots \sigma_{\lambda_j}\cdots \sigma_{\lambda_k}^i.
\end{align*}
Similarly, for $\lambda \in \DP_n^+$, we get that
$
	g_j \sigma_{\lambda_1}^i\cdots \sigma_{\lambda_j}^i\cdots \sigma_{\lambda_k}^i g_j^{-1} 
	= \sigma_{\lambda_1}^i\cdots \sigma_{\lambda_j}\cdots \sigma_{\lambda_k}^i.
$

Combining the above two cases, we obtain
\begin{align}\label{eq:sigma_la^i_p2024_used}
	g (\sigma_{\lambda_1}^i\cdots \sigma_{\lambda_k}^i) g^{-1}
	&= \begin{cases}
	     \left(\prod_{j:\lambda_j \text{ even}} \jacobi{D(\lambda_j)}{i} \right) \sigma_{\lambda_1}\cdots \sigma_{\lambda_k}  & \text{if } \lambda\in DP_n^+,\\
          \left(\prod_{j:\lambda_j \text{ odd}} \jacobi{i}{\lambda_j}\right)\left(\prod_{j:\lambda_j \text{ even}} \jacobi{D(\lambda_j)}{i} \right) \sigma_{\lambda_1}\cdots \sigma_{\lambda_k} & \text{if } \lambda\in DP_n^-,\\
	\end{cases}
\end{align}
with $\sgn(g) = \prod_{j:\lambda_j \text{ odd}} \jacobi{i}{\lambda_j}$, where $g = g_1 \dots g_k$.

The following lemma is immediate from Eq.~\eqref{eq:sigma_lambda^i_exprsession_in_sigma_la^i} and~\eqref{eq:sigma_la^i_p2024_used}.

\begin{lemma}
	\label{lemma:conjugacy_classes_of_powers_of_sigma_lambda}
	Let $\lambda=(\lambda_1,\dotsc,\lambda_k)$ be a partition of $n$ with length $k$.
	Let $N$ be the order of $\sigma_\lambda$.
	Then for any integer $i$, the power $\sigma_\lambda^i$ is conjugate to $\sigma_\alpha$ in $\tilde S_n$, where $\alpha=\lambda^i$, if and only if one of the following holds:
	\begin{enumerate}
		\item $\alpha \notin\OP_n\cup\DP^-_n$. In this case, $\sigma_\lambda^i$ is conjugate to both $\sigma_\alpha$ and $-\sigma_\alpha$.
		\item $\alpha \in\OP_n$ and the order of $\sigma_\alpha$ is equal to $N/(i,N)$.
		\item $\alpha \in \DP^-_n$; then $\alpha=\lambda$ and $(-1)^{\binom{i}{2}\binom{e(\lambda)}{2}}\left(\prod_{j:\lambda_j \text{ odd}} \jacobi{i}{\lambda_j}\right)\left(\prod_{j:\lambda_j \text{ even}} \jacobi{D(\lambda_j)}{i} \right)= 1$.
	\end{enumerate}
\end{lemma}
Since the conjugacy class $C_\lambda$ of $\tilde{S_n}$ splits in $\tilde A_n$ if and only if $\lambda \in \DP_n^+$, we consider only this case in the following lemma.
\begin{lemma}
	\label{lemma:conjugacy_classes_of_powers_of_sigma_lambda_alt}
	Let $\lambda \in \DP_n^+$ be a partition of $n$ with length $k$.
	Let $N$ be the order of $\sigma_\lambda$.
	Then for any integer $i$, the power $\sigma_\lambda^i$ is conjugate to $\sigma_\lambda$ in $\tilde A_n$ if and only if 
    $(i,N)=1$ and 
    \begin{displaymath}
        \left(\prod_{j:\lambda_j \text{ odd}} \jacobi{i}{\lambda_j}\right)\left(\prod_{j:\lambda_j \text{ even}} \jacobi{D(\lambda_j)}{i} \right) = 1.
    \end{displaymath}
    In particular, if $\lambda \in \OP_n \cap \DP_n$, the order of $\sigma_\lambda$ is equal to $2 \lcm(\lambda)$ and $(i,2\lcm(\lambda)) = 2$, then $\sigma_\lambda^i $ is conjugate to $-\sigma_\lambda$ if and only if $\prod_{j:\lambda_j \text{ odd}} \jacobi{i}{\lambda_j} = 1$.
\end{lemma}
\begin{proof}
Let $\lambda \in \DP_n^+$ and $C_\lambda$ be a conjugacy class of $\tilde S_n$.
Then $C_\lambda$ splits into two conjugacy classes of $\tilde A_n$.
Let $C$ be the conjugacy class of $\tilde A_n$ containing $\sigma_\lambda$. 
Let $i$ be a positive integer with $(i,N)=1$.
Since the orders of $\sigma_\lambda$ and $\sigma^i_\lambda$ are equal, there exists $g\in \tilde S_n$ such that $g \sigma_\lambda^i g^{-1} = \sigma_\lambda$.
By Lemma~\ref{lemmma:conjugacy_class_splitting_index_2} and the the discussion above with the fact that $\sgn(g) = \prod_{j:\lambda_j \text{ odd}} \jacobi{i}{\lambda_j}$, yields that $\sigma_\lambda^i$ is conjugate to $\sigma_\lambda$ in $\tilde{A}_n$ if and only if 
\begin{displaymath}
    \left(\prod_{j:\lambda_j \text{ odd}} \jacobi{i}{\lambda_j}\right)\left(\prod_{j:\lambda_j \text{ even}} \jacobi{D(\lambda_j)}{i} \right) = 1.
\end{displaymath}

The latter statement follows from~\cite[Theorem 2.5]{p2024cyclic}, Lemma~\ref{lemmma:conjugacy_class_splitting_index_2}, and the fact that $-\sigma_\lambda$ and $-\sigma_\lambda^i$ have odd order.
\end{proof}

With these tools, we can compute the minimal polynomial of any element of $\tilde S_n$ and $\tilde A_n$ using Sage~\cite{sagemath}; see the implementation at \url{https://github.com/velmurugan1066/Spin_char_min_poly/blob/main/spin_poly}.

\section{Minimal polynomials of $n$-cycles in $\tilde{S}_n$}
\label{section:minimal_polynomial_of_n_cycles}

In this section, we prove our main theorem for the $n$-cycles in $\tilde{S}_n$ 
and provide uniform lower bounds for multiplicities of its eigenvalues.

\subsection{Minimal polynomials of $p$-cycles}

Minimal polynomials of prime-order elements in quasi-simple groups have been studied in characteristic $0$ in~\cite{Zalesski_prime_projective_alternating_1996,Zalesski_distinct_eigenvalues_2006,Zalesski_quasi_simple_prime_2008,Tiep_Zalesski_2008}.
Recently, Zalesskii~\cite{Zalesskii_multiplicity_one_for_all_but_one} characterized prime-power order elements for which all but one eigenvalue have multiplicity $1$ in all irreducible spin representations of $\tilde A_n$.

We first compute the character values of $n$-cycles in irreducible spin representations of $\tilde S_n$.

\begin{lemma}\label{lemma:n_cycle_character_value_sym}
	Let $\lambda$ be a strict partition of $n$.
	Then the character value of an $n$-cycle $\sigma_{(n)}$ in the irreducible spin representation $V_\lambda$ of $\tilde S_n$ is
	\begin{displaymath}
		\phi_\lambda(\sigma_{(n)})=
		\begin{cases}
			(-1)^{\lambda_1-1} & \text{if } \lambda=(\lambda_1,\lambda_2)\text{ and }n\text{ is odd},\\
			i^k\sqrt{k} & \text{if } \lambda=(n)\text{ and }n=2k,\\
			0 & \text{otherwise}.
		\end{cases}
	\end{displaymath}
\end{lemma}

\begin{proof}
	The even case follows from Lemma~\ref{lemma:spin_character_values}.
	Assume that $n$ is odd.
	Apply Morris' recursion formula (Theorem~\ref{theorem:spin_Murnaghan_Nakayama_rule}) to $\sigma_{(n)}$.
    If $\ell(\lambda)>2$, there are no $n$-bars in $\lambda$, so $\phi_\lambda(\sigma_{(n)})=0$. If $\ell(\lambda)\le 2$, there is a unique $n$-bar, corresponding to the cell $(1,1)$ in the shifted diagram of $\lambda$.
	Thus, $L(b)=\lambda_2$ and $m(b)=0$. Hence
	\begin{displaymath}
		\phi_\lambda(\sigma_{(n)})=(-1)^{\lambda_2}\phi_{\varnothing}(\sigma_{(0)})=(-1)^{\lambda_2}=(-1)^{\lambda_1-1}.
	\end{displaymath}
\end{proof}

We also use the following standard lower bound for spin degrees, originally proved in~\cite{Kleidman_Wales_irr_restriction_o}; see also~\cite[Theorem A]{Kleshev_Tiep_modular_irr_restriction} for a characteristic-free version.

\begin{lemma}\cite[Theorem 3.1]{Kleidman_Wales_irr_restriction_o}
\label{lemma:lower_bound_on_dimension_of_spin_representation}
	Let $\lambda\in\DP_n$ with $n\geq 4$.
	Then
	\begin{displaymath}
		\dim(V_\lambda)\ge \dim(V_{(n)}) = 2^{\lfloor (n-1)/2 \rfloor}.
	\end{displaymath}
\end{lemma}

\begin{proof}
	Proceed by induction on $n$.
	The cases $4\le n\le 7$ are read from the character tables in~\cite{Hoffman_Humphreys}.
	Assume $n\ge 8$ and the claim holds for smaller $n$.

	If $\lambda=(n)$, the inequality is immediate.
	So assume $\ell(\lambda)\ge 2$.
	Write $n=m+4$, so $m\ge 4$.
	Since both $(4)$ and $(3,1)$ are contained in $\lambda$,
	Lemma~\ref{lemma:most_dominant_alpha_beta} implies that there exist $\mu,\nu\in\DP_m$ with
	\begin{displaymath}
		f_{\mu(4)}^\lambda>0,\qquad f_{\nu(3,1)}^\lambda>0.
	\end{displaymath}
	By induction,
	\begin{displaymath}
		\dim(V_\mu),\dim(V_\nu)\ge 2^{\lfloor (m-1)/2 \rfloor}.
	\end{displaymath}
	Using Theorem~\ref{theorem:Littlewood-Richardson-rule}, we get that
	\begin{align*}
	\dim(V_\lambda)
	&\ge \dim(V_\mu)\dim(V_{(4)})+\dim(V_\nu)\dim(V_{(3,1)})\\
	&\ge 2^{\lfloor (m-1)/2 \rfloor}\cdot 2+2^{\lfloor (m-1)/2 \rfloor}\cdot 2\\
	&=2^{\lfloor (m-1)/2 \rfloor+2}
	=2^{\lfloor (n-1)/2 \rfloor}.
	\end{align*}
\end{proof}

Now we determine the minimal polynomial of $p$-cycles.

\begin{lemma}
	\label{lemma:zalesskii_prime_sym}
	Let $w\in \tilde S_p$ be a $p$-cycle, where $p\ge 7$ is prime, with order $p$.
	Then the minimal polynomial of $w$ is $x^p-1$ in every irreducible spin representation of $\tilde S_p$.
	Moreover, if $p\ge 19$, each eigenvalue of $w$ appears with multiplicity at least $p$.
\end{lemma}

\begin{proof}
Let $\phi$ be an irreducible spin character and let $\delta$ be a linear character of $\langle w\rangle$.
Since $w$ is conjugate to $w^i$ for all $1\le i\le p-1$,
\begin{displaymath}
\left\langle \Res^{\tilde S_p}_{\langle w\rangle}\phi,\delta\right\rangle
=\frac1p\left(\phi(1)+\phi(w)\sum_{i=1}^{p-1}\overline{\delta(w^i)}\right).
\end{displaymath}
By Lemma~\ref{lemma:n_cycle_character_value_sym}, for odd $p$ we have $\phi(w)\in\{0,\pm1\}$.
Hence
\begin{displaymath}
\left\langle \Res^{\tilde S_p}_{\langle w\rangle}\phi,\delta\right\rangle
\ge \frac1p\bigl(\phi(1)-(p-1)\bigr).
\end{displaymath}
Using Lemma~\ref{lemma:lower_bound_on_dimension_of_spin_representation}, we have
\begin{displaymath}
\left\langle \Res^{\tilde S_p}_{\langle w\rangle}\phi,\delta\right\rangle
\ge \frac1p\bigl(2^{(p-1)/2}-p+1\bigr)>0
\end{displaymath}
for all $p\ge 7$.
Thus every linear character of $\langle w\rangle$ occurs, so the minimal polynomial is $x^p-1$.

For the latter statement of the lemma, it is enough to show that
\begin{displaymath}
\frac1p\bigl(2^{(p-1)/2}-p+1\bigr)\ge p,
\end{displaymath}
equivalently,
\begin{displaymath}
2^{(p-1)/2}-p^2-p+1\ge 0.
\end{displaymath}
The function $f(p)=2^{(p-1)/2}-p^2-p+1$ is increasing for $p\ge 19$, and $f(19)>0$.
\end{proof}
\subsection{Minimal polynomials of arbitrary $n$-cycles}~\label{subsection:arbitrary_n_cycles}
In this subsection, we prove the following theorem, which describes the minimal polynomial of $n$-cycles and provides uniform lower bounds for the multiplicity of its eigenvalues in all irreducible spin representations of $\tilde S_n$.
\begin{theorem}
	\label{theorem:cycle_double_S_n}
    The minimal polynomial of any irreducible spin representation $(\rho, V)$ of $\tilde S_n$ at an element $g$ with cycle type $(n)$ is $x^n-\epsilon$, where $\epsilon=g^n$, except when
	\begin{enumerate}
        \item   $\rho=\rho_{(2)}^+$ and the minimal polynomial of $\sigma_{(2)}^\pm$ is $ x\mp i$.
		\item   $\rho=\rho_{(2)}^-$ and the minimal polynomial of $\sigma_{(2)}^\pm$ is $ x \pm i$.
        \item $\rho=\rho_{(3)}$, where the minimal polynomial is $x^2 + \epsilon x+1$.
		\item $\rho=\rho_{(2,1)}^\pm$, where the minimal polynomial is $x - \epsilon$.
		\item $\rho  = \rho_{(4)}^\pm$ and $g$ is conjugate to $\sigma_{(4)}^+$, where the minimal polynomial is $x^2 \mp \sqrt{2} x +1$.
		\item $\rho  = \rho_{(4)}^\pm$ and $g$ is conjugate to $\sigma_{(4)}^-$, where the minimal polynomial is $x^2 \pm \sqrt{2} x +1$.
		\item $\rho=\rho_{(5)}$, where the minimal polynomial is $\frac{x^5 - \epsilon}{x - \epsilon}$.
		\item $\rho = \rho_{(6)}^\pm$ and $g$ is conjugate to $\sigma_{(6)}^+$, where the minimal polynomial is $x^4\mp\sqrt{3}x^3-2x^2\pm\sqrt{3}x+1$.
		\item $\rho = \rho_{(6)}^\pm$ and $g$ is conjugate to $\sigma_{(6)}^-$, where the minimal polynomial is $x^4\pm\sqrt{3}x^3-2x^2\mp\sqrt{3}x+1$.
		\item $\rho=\rho_{(8)}^\pm$ and $g$ is conjugate to $\sigma_{(8)}^+$, where the minimal polynomial is $\frac{x^8 -1}{(x \pm 1)}$.
		\item $\rho=\rho_{(8)}^\pm$ and $g$ is conjugate to $\sigma_{(8)}^-$, where the minimal polynomial is $\frac{x^8 -1}{(x \mp 1)}$.
	\end{enumerate}
	Moreover, for $n\geq 19$, each eigenvalue of $\rho(g)$ has multiplicity at least $n$.
\end{theorem}

For $H\leq S_m$ and $K\leq S_p$, the wreath product $H \wr K$ is defined as the semidirect product $H^{\times p} \rtimes K$, where $K$ acts on $H^{\times p}$ by permuting the $p$ factors.
The group $H \wr K$ can also be viewed as a subgroup of $S_n$ ($n=mp$), via the natural embedding~\cite[4.1.18]{JamesKerber}.

Let us recall a lemma concerning wreath products of cyclic groups.
\begin{lemma}[\cite{vel_dihedral}, Lemma 2.1]\label{lemma:the_n_cyle_in_S_n}
Let $n$ be a positive integer with $n=mp$, and let $C_n$ be a cyclic subgroup of $S_n$ generated by an $n$-cycle.
Then $C_m \wr C_p \leq S_n$ contains an $n$-cycle.
\end{lemma}

\begin{notation}
Let $n=mp$ be odd and $C_m^{\times p}$ be a subgroup of $S_n$ generated by $\tau_i=(m(i-1)+1~m(i-1)+2~\dots~m(i-1)+m)$ for $1 \leq i \leq p$.
Define $\theta$ to be a preimage in $\tilde S_n$ of order $p$ of the generator
$(1~m+1~\dots m(p-1)+1)~(2~m+2~\dots~m(p-1)+2)~\dots~(m~2m~\dots~mp)$
of $C_p$, where $C_p$ sits inside $S_n$ by permuting $p$-many copies of $C_m$.
\end{notation}

\begin{lemma}~\label{lemma:wreath_product_of_cyclic_groups_odd_odd_Sn}
	Let $n$ be an odd positive integer with $n=mp$.
	Then $\pi^{-1}(C_m \wr C_p)$ contains a subgroup isomorphic to $C_m\wr C_p\leq S_n$.
\end{lemma}

\begin{proof}
	Consider the elements $\tau_i=(m(i-1)+1~m(i-1)+2~\dots~m(i-1)+m)$ for all $1 \leq i \leq p$. Clearly, 
	$\tau_1,\tau_2,\dots,\tau_p$ generate the group $C_m^{\times p}\leq S_n$.
	Since $m$ is odd, $\pi^{-1}(\tau_i)$ contains an element $w_i$ of order $m$ for all $1\leq i\leq p$.
	Then the subgroup of $\tilde S_n$ generated by $w_1,w_2,\dots,w_p$ is isomorphic to $C_m^{\times p}$.
	Since $\theta$ normalizes the subgroup $\langle w_1,\dots,w_p\rangle$, the subgroup generated by $w_1,w_2,\dots,w_p$ and $\theta$ is isomorphic to $C_m \wr C_p$, and its image under the projection $\pi$ is equal to $C_m \wr C_p$.
\end{proof}

\begin{remark}
It is well known that every group of order $2k$ with odd $k$
contains a unique normal subgroup of order $k$ (a consequence of the existence of an odd permutation in the regular representation), so every subgroup of odd order in $S_n$ has an isomorphic copy in its preimage in $\tilde S_n$.
\end{remark}


By abuse of notation, we denote the isomorphic copy of $C_m\wr C_p$ in $\pi^{-1}(C_m\wr C_p)$ by $C_m \wr C_p$.

\begin{corollary}\label{corollary:the_n_cycle_in_S_n_odd}
    Let $w_1,\dots,w_p$ be as in the proof of Lemma~\ref{lemma:wreath_product_of_cyclic_groups_odd_odd_Sn}.
	The wreath product $C_m \wr C_p\leq \tilde S_n$ contains a cyclic subgroup $E_n$ of order $n$ generated by $\eta_n := w_1 \theta$.
	Moreover, $\eta_n^p=w_1 w_2 \dots w_p$.
\end{corollary}

\begin{proof}
	Using Lemma~\ref{lemma:the_n_cyle_in_S_n}, $\eta_n:= w_1\theta$ is the unique preimage of order $n$ of the $n$-cycle. The equality $\eta_n^p=w_1 w_2 \dots w_p$ follows from the fact that the equality holds when we apply $\pi$ which is an isomorphism when restricted to $C_m\wr C_p$.
\end{proof}

When $n$ is odd, all $n$-cycles of order $n$ in $\tilde S_n$ are conjugate, so we will consider the $n$-cycle $\eta_n$ and the cyclic subgroup $E_n$ generated by it.
\begin{lemma}\label{lemma:m_odd_p_nontrivial_atleast_two_distinct}
Let $n=mp>27$ be an odd composite integer, with $p$ being the smallest prime divisor of $n$.
For each linear character $\delta$ of $E_n$, there exist linear characters $\delta_1,\delta_2,\ldots,\delta_p$ of $C_m$ such that $\delta_i\neq \delta_j$ for some $i\neq j$, and 
\begin{displaymath}
    \Res^{E_n}_{\langle \eta_n^p \rangle } \delta = \Res^{C_m^{\times p}}_{ \langle w_1 w_2 \cdots w_p \rangle} (\delta_1 \times \delta_2 \times \cdots \times \delta_p).
\end{displaymath}
Moreover, if $n>27$, then there exist at least $np$ such tuples of linear characters.
\end{lemma}

\begin{proof}
	Recall that $\eta_n^p=w_1 w_2 \dots w_p$ and that $\eta_n$ has order $n$.
	Let $\zeta_m$ be a primitive $m$-th root of unity.
    Then $\delta(\eta_n^p) = \zeta_m^f$ for some $f \in \{1,\dots,m\}$.
    It is easy to see that the number of linear characters $\delta_1 \times\cdots \times \delta_p$ of $C_m^{\times p}$ with 
    \begin{displaymath}
        \Res^{E_n}_{\langle \eta_n^p \rangle } \delta = \Res^{C_m^{\times p}}_{ \langle w_1 w_2 \cdots w_p \rangle} (\delta_1 \times \delta_2 \times \cdots \times \delta_p)
    \end{displaymath}
    is $m^{p-1}$.
    Among these characters, the characters for which all $\delta_i$ are equal are given by $\delta_1(w_1)^p = \delta(\eta_n^p)$.
    Therefore, the number of such characters is at most $p$.
      Hence, we get at least $m^{p-1} - p$ many characters which satisfy all the hypothesis of our lemma. Since $m^{p-1}-p \geq np$, the latter statement of the lemma follows.
\end{proof}

We now give a useful proposition.

\begin{proposition}
	\label{prop:restriction_to_C_n_odd}
	Suppose that $n=mp$ is an odd positive integer with $m>1$ and $p$ a prime.
	Let $W$ be an irreducible spin representation of $\tilde S_n$, and let $V\otimes V_1 \otimes V_2 \otimes \cdots \otimes V_p$ be an irreducible spin representation of $\tilde S_{(m^p)}$ such that
	\begin{displaymath}
		\Res^{\tilde S_n}_{\tilde S_{(m^p)}} W\geq V\otimes V_1 \otimes V_2 \otimes \cdots \otimes V_p.
	\end{displaymath}
	If $\Res^{\tilde S_m}_{C_m} V_i \geq \delta_i$, where $\delta_i$ is a linear representation of $C_m$ for all $i\in [p]$ and at least two of the $\delta_i$ are distinct, then 
	\begin{displaymath}
		\Res^{\tilde S_{n}}_{E_n} W \geq \delta
	\end{displaymath}
	for every linear subrepresentation $\delta$ of $\Ind_{\langle \eta_n^p \rangle}^{E_n} \Res^{C_m^{\times p}}_{ \langle \eta_n^p \rangle} (\delta_1 \otimes \delta_2 \otimes \cdots \otimes \delta_p)$.
\end{proposition}

\begin{proof}
Recall that $\eta_n^p=w_1w_2\ldots w_p$. Since each $w_i$ is even, using Eq.~\eqref{eq:action_in_reduced_Clifford_product}, the action of $\eta_n^p$ on the reduced Clifford product $V\otimes V_1 \otimes V_2 \otimes \cdots \otimes V_p$ is given by 
	\begin{displaymath}
		w_1 w_2 \dots w_p (v \otimes v_1 \otimes v_2 \otimes \cdots \otimes v_p) = v \otimes w_1 v_1 \otimes w_2 v_2 \otimes \cdots \otimes w_p v_p.
	\end{displaymath}

	By hypothesis, we have $\Res^{\tilde S_m}_{C_m} V_i \geq \delta_i$ for all $i\in [p]$, so
	\begin{align*}
		\Res^{\tilde S_{(m^p)}}_{C_m^{\times p}} V\otimes V_1 \otimes V_2 \otimes \cdots \otimes V_p &\geq \delta_1 \otimes \delta_2 \otimes \cdots \otimes \delta_p.
	\end{align*}
	Since at least two $\delta_i$ are non-isomorphic and $p$ is prime, we obtain that
	\begin{displaymath}
		\Ind_{C_m^{\times p}}^{C_m \wr C_p} (\delta_1 \otimes \delta_2 \otimes \cdots \otimes \delta_p)
	\end{displaymath}
	is irreducible.
	Hence, using Frobenius reciprocity, we have
	\begin{align*}
		\Res^{\tilde S_{n}}_{C_m \wr C_p} W &\geq \Ind_{C_m^{\times p}}^{C_m \wr C_p} (\delta_1 \otimes \delta_2 \otimes \cdots \otimes \delta_p).
	\end{align*}
	Restricting further to $E_n$ and using Mackey's restriction formula, we obtain
	\begin{align*}
		\Res^{\tilde S_{n}}_{E_n} W &\geq \Res^{C_m \wr C_p}_{E_n} \Ind_{C_m^{\times p}}^{C_m \wr C_p} (\delta_1 \otimes \delta_2 \otimes \cdots \otimes \delta_p)\\
		&= \Ind_{\langle \eta_n^p \rangle}^{E_n} \Res^{C_m^{\times p}}_{ \langle \eta_n^p \rangle} (\delta_1 \otimes \delta_2 \otimes \cdots \otimes \delta_p).
	\end{align*}
	This completes the proof.
\end{proof}

\noindent\textbf{Suppose that $n=2m$ is even.}

In this case, we discuss the irreducible characters of $\widetilde{C_m\times C_m}:=\pi^{-1}(C_m\times C_m) $ and $\widetilde{C_m \wr C_2}:=\pi^{-1}(C_m \wr C_2).$ Choose elements $w_1, w_2 \in \widetilde{C_m \times C_m}$ such that $\pi(w_1) = (1~2~\dots~m),$ and $\pi(w_2) = (m+1~m+2~\dots~2m).$
Let $\eta_n\in \widetilde{C_m \wr C_2}$ have cycle type $(n)$ such that $\eta_n^2= \pm w_1 w_2,$ thanks to Lemma~\ref{lemma:the_n_cyle_in_S_n}.
Denote the subgroup generated $\eta_n$ by $E_n$.
Let $\theta$ be a preimage of an element of $C_m \wr C_2$ generating $C_2$.

This subsection divided into three cases, depending on the parity of $m$ and the orders of $w_1$ and $w_2$.
\subsubsection{Suppose that $n=2m$ and $m$ is odd.}\label{subsubsection:even_odd_odd_case}
In this case, we may replace $w_1$ (resp. $w_2$) by $-w_1$ (resp. $-w_2$) if necessary, so that $w_1$ and $w_2$ have order $m$.
Since $m$ is odd, 
$\widetilde{C_m\times C_m}=\langle w_1,w_2,-1 \rangle \cong C_m \times C_m \times \langle -1 \rangle .$
Hence, all irreducible spin characters of $\widetilde{C_m\times C_m}$ are one-dimensional and have the form
\begin{displaymath}
\delta_1 \otimes \delta_2 \otimes (-\mathbb{1}).
\end{displaymath}
Since $w_1$ and $-w_2$ (resp. $w_2$ and $-w_1$) have different orders, we infer that $\theta w_1 \theta^{-1} = w_2$ and $\theta w_2 \theta^{-1} = w_1$.
A simple application of Clifford theory yields the following lemma.
\begin{lemma}
	\label{lemma:Wreath_product_of_cyclic_groups_even_odd_odd_case}
The characters $\Ind_{\widetilde{C_m\times C_m}}^{\widetilde{C_m \wr C_2}} (\delta_1 \otimes \delta_2 \otimes (-\mathbb{1})) $ are irreducible whenever $\delta_1 \neq \delta_2$.
\end{lemma}

Recall that $\eta_n^2=\pm w_1 w_2$.
We now prove the analogue of Proposition~\ref{prop:restriction_to_C_n_odd} for the subgroup $E_n := \langle \eta_n\rangle$.
\begin{proposition}
	\label{prop:restriction_to_C_n_even_odd_odd_case}
	Let $W$ and $V\otimes V_1 \otimes V_2$ be irreducible spin representations of $\tilde S_n$ and $\tilde S_{(m^2)}$ respectively, with
	$\Res^{\tilde S_n}_{\tilde S_{(m^2)}} W \geq V \otimes V_1 \otimes V_2.$
	Suppose that $\Res^{\tilde S_m}_{C_m} V_i \geq \delta_i$, where $\delta_i$ is a linear representation of $C_m$ for $i=1,2$, and $\delta_1 \neq \delta_2$.
	Then
	\begin{displaymath}
	\Res^{\tilde S_{n}}_{E_n} W \geq  \delta
	\end{displaymath}
	for every linear subrepresentation $\delta$ of
	\begin{displaymath}
	\Ind_{\langle \eta_n^2 \rangle}^{E_n} \Res^{\langle w_1, w_2, -1 \rangle}_{ \langle \eta_n^2 \rangle} (\delta_1 \otimes \delta_2\otimes (-\mathbb{1})).
	\end{displaymath}
\end{proposition}
\begin{proof}
	Recall that $\eta_n^2=\pm w_1w_2$.
	Since each $w_i$ is even, the action of $\eta_n^2$ on the reduced Clifford product $V\otimes V_1 \otimes V_2$ using Eq.~\eqref{eq:action_in_reduced_Clifford_product} is given by
	\begin{displaymath}
		w_1 w_2 (v \otimes v_1 \otimes v_2) = v \otimes w_1 v_1 \otimes w_2 v_2.
	\end{displaymath}

	By hypothesis, $\Res^{\tilde S_m}_{C_m} V_i \geq \delta_i$ for $i=1,2$, so
	\begin{align}\label{eq:restriction_to_C_n_even_odd_odd_case}
		\Res^{\tilde S_{(m^2)}}_{\widetilde{C_m^{\times 2}}} V\otimes V_1 \otimes V_2
		\geq \delta_1 \otimes \delta_2 \otimes (-\mathbb{1}).
	\end{align}
	Since $\delta_1 \neq \delta_2$, Lemma~\ref{lemma:Wreath_product_of_cyclic_groups_even_odd_odd_case} implies that
	\begin{displaymath}
		\Ind_{\widetilde{C_m^{\times 2}}}^{\widetilde{C_m \wr C_2}} (\delta_1 \otimes \delta_2 \otimes (-\mathbb{1}))
	\end{displaymath}
	is irreducible.
	Hence, by Frobenius reciprocity and Eq.~\eqref{eq:restriction_to_C_n_even_odd_odd_case},
	\begin{displaymath}
		\Res^{\tilde S_{n}}_{\widetilde{C_m \wr C_2}} W
		\geq
		\Ind_{\widetilde{C_m^{\times 2}}}^{\widetilde{C_m \wr C_2}} (\delta_1 \otimes \delta_2 \otimes (-\mathbb{1})).
	\end{displaymath}
	Restricting further to $E_n$ and applying Mackey's restriction formula, we get
	\begin{align*}
		\Res^{\tilde S_{n}}_{E_n} W
		&\geq \Res^{\widetilde{C_m \wr C_2}}_{E_n}
		\Ind_{\widetilde{C_m^{\times 2}}}^{\widetilde{C_m \wr C_2}} (\delta_1 \otimes \delta_2 \otimes (-\mathbb{1}))\\
		&=
		\Ind_{\langle \eta_n^2 \rangle}^{E_n}
		\Res^{\widetilde{C_m^{\times 2}}}_{ \langle \eta_n^2 \rangle}
		(\delta_1 \otimes \delta_2\otimes (-\mathbb{1})).
	\end{align*}
	This completes the proof.
\end{proof}

We now prove an analogue of Lemma~\ref{lemma:m_odd_p_nontrivial_atleast_two_distinct} for the subgroup $E_n=\langle \eta_n\rangle$.
\begin{lemma}\label{lemma:m_odd_p_nontrivial_atleast_two_distinct_even_odd_case}
	For each linear spin character $\delta$ of $E_n$, there exist at least $m-2$ distinct pairs of linear spin characters $\delta_1,\delta_2$ of $\langle w_1 \rangle $ and $\langle w_2 \rangle $, respectively, such that $\delta_1 \neq \delta_2$ and
	\begin{displaymath}
		\Res^{E_n}_{\langle \eta_n^2 \rangle } \delta
		=
		\Res^{\langle w_1, w_2, -1 \rangle}_{ \langle \eta_n^2 \rangle}
		(\delta_1 \otimes \delta_2\otimes (-\mathbb{1})).
	\end{displaymath}
\end{lemma}
\begin{proof}
	Let $\delta$ be a linear spin character of $E_n$.
	Then there exist linear spin characters $\gamma_1,\gamma_2$ of $\langle w_1 \rangle $ and $\langle w_2 \rangle $, respectively,  such that
	\begin{displaymath}
	\Res^{E_n}_{\langle \eta_n^2 \rangle } \delta
	=
	\Res^{\langle w_1, w_2, -1 \rangle}_{ \langle \eta_n^2 \rangle}
	(\gamma_1 \otimes \gamma_2\otimes (-\mathbb{1})).
	\end{displaymath}
	Take any linear spin character $\delta_1$ of $C_m$, and define $\delta_2$ by letting
	\begin{displaymath}
	\delta_2(w_2) = \gamma_1(w_1)\gamma_2(w_2)\,\delta_1(w_1)^{-1}.
	\end{displaymath}
	There are $m$ choices for $\delta_1$, hence $m$ choices for $(\delta_1,\delta_2)$.
	Moreover, $\delta_1=\delta_2$ if and only if
	\begin{displaymath}
	\delta_1(w_1)^2=\gamma_1(w_1)\gamma_2(w_2),
	\end{displaymath}
	which has at most two solutions.
	Therefore, there are at least $m-2$ choices of $(\delta_1,\delta_2)$ with $\delta_1\neq\delta_2$.
	Each such pair satisfies
	\begin{displaymath}
	\Res^{E_n}_{\langle \eta_n^2 \rangle } \delta
	=
	\Res^{\langle w_1, w_2, -1 \rangle}_{ \langle \eta_n^2 \rangle}
	(\delta_1 \otimes \delta_2\otimes (-\mathbb{1})).
	\end{displaymath}
\end{proof}


\subsubsection{Suppose that $n=2m$ with $m$ even and, both $w_1$ and $w_2$ have order $m$.}\label{case:even_even_order=m}

We may multiply $w_1$ by $-1$, if necessary, and assume that $\eta_n^2=w_1w_2$.
Notice that $w_1 w_2 = -w_1w_2$ and $w_1 w_2^2 = w_2^2 w_1$.
Then $H:=\langle w_1,w_2^2,-1 \rangle$ is an index $2$ Abelian subgroup of $\widetilde{C_m\times C_m}$ and $H$ is isomorphic to $C_m \times C_{\frac{m}{2}} \times \langle -1 \rangle $, where $C_m=\langle w_1 \rangle$, and $C_{\frac{m}{2}}=\langle w_2^2 \rangle$.
	All irreducible spin characters of $H$ are of the form $\psi:=\delta_1 \otimes   {\delta_2} \otimes (-\mathbb{1})$, 
	where $\delta_1$ and $  {\delta_2}$ are linear characters of $\langle w_1 \rangle$ and $\langle w_2^2 \rangle$, respectively.
	Since $\psi(w_2 w_1 w_2^{-1})=\psi(- w_1)= -\psi(w_1)$, we have $\psi^{w_2} \neq \psi$.
	Hence, by Clifford theory, $\tilde\psi:=\Ind_{H}^{\widetilde{C_m\times C_m}} \psi$ is irreducible.
	Since there are $m^2$ linear characters of $\widetilde{C_m\times C_m}$, we obtain the following lemma.
	\begin{lemma}
		\label{lemma:Wreath_product_of_cyclic_groups_even_even_order=m_case}
		Let $H= \langle w_1,w_2^2,-1 \rangle$ be a subgroup of $\widetilde{C_m\times C_m}$.
		Then all irreducible spin characters $\tilde\psi$ of $\widetilde{C_m\times C_m}$ are two-dimensional and are given by
		\begin{displaymath}
		\tilde\psi:=\Ind_{H}^{\widetilde{C_m\times C_m}} (\delta_1 \otimes   {\delta_2} \otimes (-\mathbb{1})),
		\end{displaymath}
		where $\delta_1$ and $  {\delta_2}$ are linear characters of $\langle w_1 \rangle$ and $\langle w_2^2 \rangle$, respectively.
		All remaining irreducible characters of $\widetilde{C_m\times C_m}$ are linear, where $-1$ acts trivially.
		In particular, the irreducible characters of $\widetilde{C_m\times C_m}$ consist of $m^2$ linear characters and $m^2/4$ irreducible characters of degree $2$.
	\end{lemma}

	Recall that $\theta$ is an element of $\widetilde{C_m \wr C_2}$ whose image generates $C_2$.
	Then $\theta w_1^2 \theta^{-1}= w_2^2$ or $- w_2^2$.
	If $\delta_1(w_1^2) \neq \pm  {\delta_2}(w_2^2)$, then $\psi^{\theta} \ncong \psi$.
	Hence, $\Ind_{ \widetilde{C_m\times C_m}}^{\widetilde{C_m \wr C_2}} \tilde\psi$ is irreducible.
	Therefore, we obtained the following lemma.
	\begin{lemma}
		\label{lemma:Wreath_product_of_cyclic_groups_even_even_order=m_case_induced}
		Let $\tilde{\psi}~ (=\Ind_{H}^{\widetilde{C_m\times C_m}} (\delta_1 \otimes   {\delta_2} \otimes (-\mathbb{1})))$ be an irreducible spin character of $\widetilde{C_m\times C_m}$.
		If $\tilde\psi(w_1^2)\neq \pm \tilde \psi(w_2^2)$, then $\Ind_{ \widetilde{C_m\times C_m}}^{\widetilde{C_m \wr C_2}} \tilde\psi$ is irreducible.
	\end{lemma}
	\begin{proof}
    Since $\tilde\psi(w_1^2)= \pm 2\delta_1(w_1^2)$ and $\tilde\psi(w_2^2)=2  {\delta_2}(w_2^2)$, $\tilde\psi(w_1^2)= \pm \tilde\psi(w_2^2)$ if and only if $\delta_1(w_1^2) \neq \pm  {\delta_2}(w_2^2)$.
	As observed above, if $\delta_1(w_1^2) \neq \pm  {\delta_2}(w_2^2)$, then $\Ind_{ \widetilde{C_m\times C_m}}^{\widetilde{C_m \wr C_2}} \tilde\psi$ is irreducible.
	\end{proof}

	We also need the following lemma.
	\begin{lemma}\label{lemma:existence_of_characters_even_even_case}
		For each linear spin character $\delta$ of $E_n$, there exist at least $m-8$ distinct pairs of linear characters $\delta_1,  {\delta_2}$ of $\langle w_1 \rangle$ and $\langle w_2^2 \rangle$, respectively, such that $\delta_1(w_1^2) \neq \pm  {\delta_2}(w_2^2)$ and
		\begin{displaymath}
		\Res^{E_n}_{\langle \eta_n^4 \rangle} \delta
		=
		\Res^{\langle w_1, w_2^2, -1 \rangle}_{ \langle \eta_n^4 \rangle} (\delta_1 \otimes   {\delta_2} \otimes (-\mathbb{1})).
		\end{displaymath}
	\end{lemma}
	\begin{proof}
        Let $\delta_1$ be any linear spin character of $\langle w_1 \rangle$.
        Since $\eta_n^4 = -w_1^2 w_2^2$, we define $   \delta_2$ by letting 
        \begin{displaymath}
              \delta_2(w_2^2)= - \delta(-w_1^2 w_2^2)    \delta_1(w_1^2)^{-1}.
        \end{displaymath}
        Since $\delta(\eta_n^4) = \delta_1(w_1^2)    \delta_2(w_2^2) (-1)$, it follows that 
        \begin{displaymath}
		\Res^{E_n}_{\langle \eta_n^4 \rangle} \delta
		=
		\Res^{\langle w_1, w_2^2, -1 \rangle}_{ \langle \eta_n^4 \rangle} (\delta_1 \otimes   {\delta_2} \otimes (-\mathbb{1})).
		\end{displaymath}
    	There are $m$ choices for $  {\delta_1}$, and hence there are $m$ such pairs $(  {\delta_1},   {\delta_2})$.
		Note that $  {\delta_1} = \pm   {\delta_2}$ if and only if $  {\delta_1}(w_1)^4 = \mp  \delta(-w_1^2w_2^2)$, which has at most eight solutions.
		Therefore, there are at least $m-8$ choices for $(  {\delta_1},   {\delta_2})$ such that $  {\delta_1} \neq \pm   {\delta_2}$. This completes the proof.
	\end{proof}

\subsubsection{Suppose that $n=2m$ with $m$ even and, both $w_1$ and $w_2$ have order $2m$.}\label{Subcase:even_even_order_2m_case}
Similarly to the previous case,
we may assume that $\eta_n^2 = w_1 w_2$.
Let $H=\langle w_1,w_2^2 \rangle $. Then $H$ is isomorphic to $\langle w_1 \rangle \times \langle w_1^2 w_2^2 \rangle$.
Hence, all irreducible characters of $H$ are of the form $\delta_1 \otimes   {\delta_2}$, where $\delta_1$ and $  {\delta_2}$ are linear characters of $\langle w_1 \rangle$ and $\langle w_1^2 w_2^2 \rangle$, respectively.
As in the previous case, there are $m^2$ linear characters of $\widetilde{C_m\times C_m}$ on which $-1$ acts trivially.
Let $\psi:=\delta_1 \otimes   {\delta_2}$ be an irreducible character of $H$ with $\delta_1(-1)=-1$.
Observe that $\psi(w_2 w_1 w_2^{-1})=\psi(- w_1)= -\psi(w_1)$.
Therefore, $\psi^{w_2} \neq \psi$.
By Clifford theory, $\tilde\psi:=\Ind_{H}^{\widetilde{C_m\times C_m}} \psi$ is irreducible and has dimension $2$.
Thus, we have the following lemma.
\begin{lemma}
	\label{lemma:Wreath_product_of_cyclic_groups_even_even_order=2m_case}
	Let $H= \langle w_1,w_1^2 w_2^2 \rangle$ be a subgroup of $\widetilde{C_m\times C_m}$.
	Then all irreducible spin characters $\tilde\psi$ of $\widetilde{C_m\times C_m}$ are two-dimensional and are given by $\tilde\psi:=\Ind_{H}^{\widetilde{C_m\times C_m}} (\delta_1 \otimes   {\delta_2})$, where $\delta_1$ and $  {\delta_2}$ are linear characters of $\langle w_1 \rangle$ and $\langle w_1^2 w_2^2 \rangle$, respectively, with $\delta_1(-1)=-1$.
	All remaining irreducible characters of $\widetilde{C_m\times C_m}$ are linear, where $-1$ acts trivially.
	In particular, the irreducible characters of $\widetilde{C_m\times C_m}$ consist of $m^2$ linear characters and $m^2/4$ irreducible characters of degree $2$.
\end{lemma}

As before, $\theta w_1^2 \theta^{-1}= \pm w_2^2$.
Therefore, if $\delta_1(w_1^2) \neq \pm \delta_1(w_1^{-2})  {\delta_2}(w_1^2 w_2^2)$, then $\tilde\psi^{\theta} \ncong \tilde\psi$.
Once again, by Clifford theory, we have the following lemma.
\begin{lemma}
	\label{lemma:Wreath_product_of_cyclic_groups_even_even_order=2m_case_induced}
	Suppose that $n=2m$ with $m$ even.
	If $\tilde{\psi}$ is an irreducible spin representation of $\widetilde{C_m\times C_m}$ such that $\tilde{\psi}(w_1^2)\neq \pm \tilde{\psi}(w_2^2)$, then $\Ind_{ \widetilde{C_m\times C_m}}^{\widetilde{C_m \wr C_2}} \tilde\psi$ is irreducible.
\end{lemma}

\begin{lemma}\label{lemma:existence_of_characters_even_even_2m_case}
	Let $n=2m$ with $m\geq 8$ even, and let $w_1,w_2$ be of order $2m$.
	If $\delta$ is a linear spin character of $E_n=\langle \eta_n \rangle$, then there exist at least  $m-8$ distinct pairs of linear characters $\delta_1$ and $  {\delta_2}$ of $\langle w_1 \rangle$ and $\langle w_1^2w_2^2 \rangle$, respectively, such that
	\begin{displaymath}
	 \Ind_{\langle \eta_n^4 \rangle}^{E_n} \Res^{\langle w_1, w_2^2 \rangle}_{ \langle \eta_n^4 \rangle} (\delta_1 \otimes   {\delta_2}) \geq \delta
	\end{displaymath}
	and $\delta_1(w_1^2) \neq \pm \delta_1(w_1^{-2})  {\delta_2}(w_1^2 w_2^2)$.
\end{lemma}
\begin{proof}
	Let $\delta$ be a linear spin character of $E_n$.
	Take any linear spin character $\delta_1$ of $\langle w_1 \rangle$, i.e., $\delta_1(-1)=-1$.
    Since $\eta_n^4 = - w_1^2 w_2^2$, we define $  {\delta_2}$ by letting $  {\delta_2}(w_1^2 w_2^2) =   -\delta(-w_1^2 w_2^2)$.
	Then $\Res^{\langle w_1, w_2^2 \rangle}_{ \langle \eta_n^4 \rangle} (\delta_1 \otimes   {\delta_2}) = \Res^{E_n}_{\langle \eta_n^4 \rangle} \delta$.
	Note that this gives $m$ pairs $({\delta_1},   {\delta_2})$.
    The number of pairs $(\delta_1,   \delta_2)$ for which 
    $\delta_1(w_1^2) \neq \pm \delta_1(w_1^{-2})  {\delta_2}(w_1^2 w_2^2)$, equivalently,  $\delta_1(w_1)^4 = \pm {  \delta_2}(w_1^2 w_2^2)$ has at at most $8$ solutions.
	Therefore, there are at least $m-8$ choices for $(\delta_1,   {\delta_2})$ such that $\delta_1(w_1^2) \neq \pm \delta_1(w_1^{-2})  {\delta_2}(w_1^2 w_2^2)$ and $\Res^{\langle w_1, w_2^2 \rangle}_{ \langle \eta_n^4 \rangle} (\delta_1 \otimes   {\delta_2}) = \Res^{E_n}_{\langle \eta_n^4 \rangle} \delta$.
\end{proof}

\subsection*{Proof of Theorem~\ref{theorem:cycle_double_S_n}}
We prove the theorem by induction on $n$.
We verify the theorem directly using Sage for all $n\leq 27$.
Therefore, we may assume that $n>27$.
Assume that the theorem holds for all positive integers less than $n$, and we shall prove it for $n$.
Let $\lambda\in\DP_n$, and let $V_\lambda$ be an irreducible spin representation of $\tilde S_n$ indexed by $\lambda$.

If $n$ is prime, then the theorem follows from Lemma~\ref{lemma:zalesskii_prime_sym}.
So we may assume that $n$ is composite.

\textbf{Case 1:} Let $n$ be an odd positive integer with $n=mp$, where $p$ is the smallest prime divisor of $n$.
Since $n>27$, we have $m\geq 7$.
Let $\eta_n$ be an $n$-cycle in $\tilde S_n$ of order $n$ such that 
$\eta_n^p=w_1 w_2 \dots w_p$, with $\pi (w_i)=(m(i-1)+1~m(i-1)+2~\dots~mi)$ for all $1\leq i \leq p$ as in Corollary~\ref{corollary:the_n_cycle_in_S_n_odd}.

Let $\delta$ be an irreducible representation of $E_n=\langle\eta_n\rangle$.
Then, using Lemma~\ref{lemma:m_odd_p_nontrivial_atleast_two_distinct} and Frobenius reciprocity, we have linear characters $\delta_i$ of $\langle w_i \rangle$ for all $i\in \{1,\dots,p\}$ such that at least two of them are distinct and

\begin{displaymath}
	\Ind_{\langle \eta_n^p \rangle}^{E_n} \Res^{C_m^{\times p}}_{\langle \eta_n^p \rangle} (\delta_1 \otimes \delta_2 \otimes \cdots \otimes \delta_p) \geq \delta.
\end{displaymath}

Let $\Res^{\tilde S_n}_{\tilde S_{(m^p)}} V_\lambda \geq V \otimes V_1 \otimes V_2 \otimes \cdots \otimes V_p$ for some irreducible spin representation $V \otimes V_1 \otimes V_2 \otimes \cdots \otimes V_p$ of $\tilde S_{(m^p)}$.
Since $m$ is odd and $m\geq 7$, we have
$ \Res^{\tilde S_m}_{C_m} V_i \geq \delta_i $
for all $i\in [p]$ by induction.
It follows from Proposition~\ref{prop:restriction_to_C_n_odd} that
\begin{align}\label{equation:multiplicity_for_main_odd}
	\Res^{\tilde S_n}_{E_n} V_\lambda \geq \Ind_{\langle \eta_n^p \rangle}^{E_n} \Res^{C_m^{\times p}}_{\langle \eta_n^p \rangle} (\delta_1 \otimes \delta_2 \otimes \cdots \otimes \delta_p) \geq \delta.
\end{align}

By  Lemma~\ref{lemma:m_odd_p_nontrivial_atleast_two_distinct}, we can choose at least $np$ 
tuples $(\delta_i)_{i=1}^p$. Now Eq.~\eqref{equation:multiplicity_for_main_odd}
 implies that
\begin{displaymath}
	\Res^{\tilde S_n}_{E_n} V_\lambda \geq n \delta,
\end{displaymath}
for all $n > 27$ and for all linear characters $\delta$ of $E_n$.

\textbf{Case 2:} Let $n=2m$ be an even positive integer.
A direct check using Sage for $13\leq m \leq 18$ shows that
\begin{align}\label{equation:multiplicity_for_13_18_at_least_4}
	\Res^{\tilde S_m}_{E_m} V_\lambda \geq 4 \delta,
\end{align}
for all linear spin characters $\delta$ of $E_m$ and for all irreducible spin representations $V_\lambda$ of $\tilde S_m$.
We split this case into three subcases.

\textbf{Subcase 2.1:} Suppose that $m$ is odd. 
Let $\eta_n$ be an $n$-cycle in $\tilde S_n$ as in Subsection~\ref{subsubsection:even_odd_odd_case} such that
$\eta_n^2=- w_1 w_2$ or $w_1w_2$, with $\pi (w_1)=(1~2~\dots~m)$, $\pi(w_2)=(m+1~m+2~\dots~2m)$, and $w_1,w_2$ of order $m$.
Let $\delta$ be any linear spin representation of $E_n$.
Then, using Lemma~\ref{lemma:m_odd_p_nontrivial_atleast_two_distinct_even_odd_case}, there are at least $m-2$ pairs of linear spin characters $(\delta_1, \delta_2)$ of $\langle w_1 \rangle$ and $\langle w_2 \rangle$, respectively, such that $\delta_1 \neq \delta_2$ and
\begin{displaymath}
\Res^{E_n}_{\langle \eta_n^2 \rangle} \delta
=
\Res^{\langle w_1, w_2, -1 \rangle}_{ \langle \eta_n^2 \rangle}
(\delta_1 \otimes \delta_2\otimes (-\mathbb{1})).
\end{displaymath}

Let $\Res^{\tilde S_n}_{\tilde S_{(m^ 2)}} V_\lambda \geq V \otimes V_1 \otimes V_2$, where $V \otimes V_1 \otimes V_2$ is an irreducible spin representation of $\tilde S_{(m^2)}$.
Since $m> 13$, by the induction hypothesis, Eq.~\eqref{equation:multiplicity_for_13_18_at_least_4}, and Lemma~\ref{lemma:spectra_of_product_of_two_modules}, we have 
\begin{displaymath}
	\Res^{\tilde S_m}_{C_m} V_i \geq  4 \delta_i
\end{displaymath}
for all $i=1,2$.
Therefore, by Proposition~\ref{prop:restriction_to_C_n_even_odd_odd_case} and Lemma~\ref{lemma:m_odd_p_nontrivial_atleast_two_distinct_even_odd_case}, we have
\begin{displaymath}
	\Res^{\tilde S_n}_{E_n} V_\lambda \geq \tfrac{16(m-2)}{2} \delta \geq n \delta.
\end{displaymath}
This completes the proof for Subcase 2.1.

Now we shall consider the case where $m$ is even.
Recall that $\eta_n$ is an $n$-cycle in $\tilde S_n$ such that $\eta_n^2=w_1 w_2$, with $\pi (w_1)=(1~2~\dots~m)$ and $\pi(w_2)=(m+1~m+2~\dots~2m)$ as in Subsections~\ref{case:even_even_order=m} and \ref{Subcase:even_even_order_2m_case}.

\textbf{Subcase 2.2:} Suppose that $w_1$ and $w_2$ are of order $m$ and $m$ is even.

Let $\delta$ be an irreducible spin representation of $E_n$.
Using Lemma~\ref{lemma:existence_of_characters_even_even_case}, there exist at least $m-8$ linear spin characters $\delta_1$ and $  {\delta_2}$ of $\langle w_1 \rangle$ and $\langle w_2^2 \rangle$, respectively, such that 
\begin{align}\label{eq:existence_of_characters_even_even_case}
	\Ind_{\langle \eta_n^4 \rangle}^{E_n} \Res^{\langle w_1, w_2^2, -1 \rangle}_{ \langle \eta_n^4 \rangle} (\delta_1 \otimes   {\delta_2} \otimes (-\mathbb{1})) \geq \delta.
\end{align}

Suppose that $\Res^{\tilde S_n}_{\tilde S_{(m^{ 2})}} V_\lambda \geq V \otimes V_1 \otimes V_2$ for some irreducible spin representation $V \otimes V_1 \otimes V_2$ of $\tilde S_{(m^{ 2})}$.
By the induction hypothesis and Lemma~\ref{lemma:spectra_of_product_of_two_modules}, we have
\begin{displaymath}
\Res^{\tilde S_{(m^2)}}_{\langle w_1,w_2^2, -1 \rangle} V \otimes V_1 \otimes V_2 \geq \gamma_1 \otimes \gamma_2 \otimes (-\mathbb{1})
\end{displaymath}
for all linear representations $\gamma_1$ and $\gamma_2$ of $\langle w_1 \rangle$ and $\langle w_2^2 \rangle$, respectively.

Therefore,
\begin{displaymath}
	\Res^{\tilde S_{(m^2)}}_{\langle w_1,w_2^2,-1 \rangle} V \otimes V_1 \otimes V_2 \geq \psi,
\end{displaymath}
for all irreducible spin representations $\psi:=\delta_1 \otimes   {\delta_2} \otimes (-\mathbb{1})$ of $H:=\langle w_1,w_2^2,-1 \rangle$.

Using Lemma~\ref{lemma:Wreath_product_of_cyclic_groups_even_even_order=m_case} and Frobenius reciprocity, we have
\begin{displaymath}
	\Res^{\tilde S_{(m^{2})}}_{\widetilde{C_m\times C_m}} V \otimes V_1 \otimes V_2 \geq \tilde\psi
\end{displaymath}
for all irreducible spin representations $\tilde\psi:=\Ind_{H}^{\widetilde{C_m\times C_m}} \psi$ of $\widetilde{C_m\times C_m}$.
Using Lemmas~\ref{lemma:Wreath_product_of_cyclic_groups_even_even_order=m_case_induced} and \ref{lemma:existence_of_characters_even_even_case}, we find that $\Ind_{ \widetilde{C_m\times C_m}}^{\widetilde{C_m \wr C_2}} \tilde\psi$ is irreducible.

Thus, by Frobenius reciprocity, we have
\begin{displaymath}
	\Res^{\tilde S_n}_{\widetilde{C_m \wr C_2}} V \otimes V_1 \otimes V_2 \geq \Ind^{\widetilde{C_m \wr C_2}}_{\langle w_1,w_2^2,-1 \rangle} \delta_1 \otimes   {\delta_2} \otimes (-\mathbb{1}).
\end{displaymath}

Restricting further to $E_n$ and applying Mackey's restriction formula, we obtain
\begin{align*}
	\Res^{\tilde S_n}_{E_n} V \otimes V_1 \otimes V_2 
	&\geq \Res^{\widetilde{C_m \wr C_2}}_{E_n} \Ind^{\widetilde{C_m \wr C_2}}_{\langle w_1,w_2^2,-1 \rangle} \delta_1 \otimes   {\delta_2} \otimes (-\mathbb{1})\\
	&\geq \Ind_{\langle \eta_n^4 \rangle}^{E_n} \Res^{\langle w_1,w_2^2,-1 \rangle}_{ \langle \eta_n^4 \rangle} \delta_1 \otimes   {\delta_2} \otimes (-\mathbb{1}).
\end{align*}

Using Eq.~\eqref{eq:existence_of_characters_even_even_case} and the inequality above, we obtain
\begin{displaymath}
	\Res^{\tilde S_n}_{E_n} V_\lambda \geq \delta.
\end{displaymath}
By the induction hypothesis and Eq.~\eqref{equation:multiplicity_for_13_18_at_least_4}, we have $\Res^{\tilde S_m}_{\langle w_1 \rangle} V_1 \geq 4 \delta_1$ and $\Res^{\tilde S_m}_{\langle w_2^2 \rangle} V_2 \geq 8   {\delta_2}$.
Since there are at least $m-8$ pairs of linear spin characters $(\delta_1,   {\delta_2})$ such that
\begin{displaymath}
\Ind_{\langle \eta_n^4 \rangle}^{E_n} \Res^{\langle w_1, w_2^2, -1 \rangle}_{ \langle \eta_n^4 \rangle} (\delta_1 \otimes   {\delta_2} \otimes (-\mathbb{1})) \geq \delta,
\end{displaymath}
we have
\begin{displaymath}
	\Res^{\tilde S_n}_{E_n} V_\lambda \geq \tfrac{4 \times 8 \times (m-8)}{4} \delta \geq n \delta.
\end{displaymath}
This completes the proof for Subcase 2.2.

\textbf{Subcase 2.3:} Suppose that $w_1$ and $w_2$ are of order $2m$.
Let $\delta$ be an irreducible spin representation of $E_n$.
Using Lemma~\ref{lemma:existence_of_characters_even_even_2m_case}, there exist $m-8$ pairs of linear spin characters $\delta_1$ and $  {\delta_2}$ of $\langle w_1 \rangle$ and $\langle w_1^2 w_2^2 \rangle$, respectively, such that
\begin{align}\label{eq:existence_of_characters_even_even_2m_case}
	\Ind_{\langle \eta_n^4 \rangle}^{E_n} \Res^{\langle w_1, w_2^2 \rangle}_{ \langle \eta_n^4 \rangle} \delta_1 \otimes   {\delta_2} \geq \delta
\end{align}
and $\delta_1(w_1^2) \neq \pm \delta_1(w_1^{-2})  {\delta_2}(w_1^2 w_2^2)$.

Let $\Res^{\tilde S_n}_{\tilde S_{(m^2)}} V_\lambda \geq V \otimes V_1 \otimes V_2$ for some irreducible spin representation $V \otimes V_1 \otimes V_2$ of $\tilde S_{(m^2)}$.
As in the previous case, by the induction hypothesis and Eq.~\eqref{equation:multiplicity_for_13_18_at_least_4}, we have $\Res^{\tilde S_m}_{\langle w_1 \rangle} V_1 \geq 4 \delta_1$ and $\Res_{\langle w_1^2 w_2^2 \rangle}^{\tilde S_m} V_2 \geq 8   {\delta_2}$, for all linear spin representations $\delta_1$ and $  {\delta_2}$ of $\langle w_1 \rangle$ and $\langle w_2^2 \rangle$, respectively.

Now, using the same argument as in Subcase 2.2, we have
\begin{displaymath}
	\Res^{\tilde S_{(m^ 2)}}_{\langle w_1,w_2^2 \rangle} V \otimes V_1 \otimes V_2 \geq \psi,
\end{displaymath}
for all irreducible representations $\psi:=\delta_1 \otimes   {\delta_2}$ of $H:=\langle w_1,w_1^2 w_2^2 \rangle$ with $\delta_1(-1)=-1$.
Using Lemma~\ref{lemma:Wreath_product_of_cyclic_groups_even_even_order=2m_case} and Frobenius reciprocity, we have
\begin{displaymath}
	\Res^{\tilde S_{(m^{ 2})}}_{\widetilde{C_m\times C_m}} V \otimes V_1 \otimes V_2 \geq \tilde\psi,
\end{displaymath}
for all irreducible spin representations $\tilde\psi:=\Ind_{H}^{\widetilde{C_m\times C_m}} \psi$ of $\widetilde{C_m\times C_m}$.
Using Lemmas~\ref{lemma:Wreath_product_of_cyclic_groups_even_even_order=2m_case_induced} and \ref{lemma:existence_of_characters_even_even_2m_case}, we find that $\Ind_{ \widetilde{C_m\times C_m}}^{\widetilde{C_m \wr C_2}} \tilde\psi$ is irreducible.
Proceeding as in Subcase 2.2, we have


\begin{align*}
	\Res^{\tilde S_n}_{E_n} V \otimes V_1 \otimes V_2 
	&\geq \Ind_{\langle \eta_n^4 \rangle}^{E_n} \Res^{\langle w_1,w_2^2 \rangle}_{ \langle \eta_n^4 \rangle} \delta_1 \otimes   {\delta_2}.
\end{align*}
Using Eq.~\eqref{eq:existence_of_characters_even_even_2m_case} and the inequality above, we have
\begin{displaymath}
	\Res^{\tilde S_n}_{E_n} V_\lambda \geq \delta.
\end{displaymath}
Using the induction hypothesis and the fact that there are at least $m-8$ pairs of linear spin characters $(\delta_1,   {\delta_2})$ such that
\begin{displaymath}
\Ind_{\langle \eta_n^4 \rangle}^{E_n} \Res^{\langle w_1, w_2^2 \rangle}_{ \langle \eta_n^4 \rangle} \delta_1 \otimes   {\delta_2} \geq \delta,
\end{displaymath}
we have
\begin{displaymath}
	\Res^{\tilde S_n}_{E_n} V_\lambda \geq \tfrac{4 \times 8 \times (m-8)}{4} \delta \geq n \delta.
\end{displaymath}

This completes the proof for Subcase 2.3, and hence the theorem. \qed

Now we state some small corollaries that will be helpful in the upcoming proofs.
\begin{corollary}\label{corollary:lower_bound_non_basic}
	For all $n\geq 11$ and for all irreducible spin representations $(\rho,V)$ of $\tilde S_n$ other than the basic spin representations, the minimal polynomial of any $n$-cycle in $V$ has $n$ distinct roots, each appearing as an eigenvalue of $\rho(\eta_n)$ with multiplicity at least $n$.
\end{corollary}
\begin{proof}
	Let $\lambda\in\DP_n \setminus \{(n)\}$.
	Then, by Theorem~\ref{theorem:cycle_double_S_n}, we have $\Res^{\tilde S_n}_{E_n} V_\lambda \geq n \delta$ for all linear spin characters $\delta$ of $E_n$ when $n\geq 19$.
    Since any $n$-cycle $g$ is either conjugate to $\eta_n$ or $-\eta_n$, it follows that $\rho_\lambda(g)$ has $n$ distinct eigenvalues for all $n\geq 19$.
	Using direct computations in Sage for $11\leq n \leq 18$, the result follows.
\end{proof}
The following corollary is an easy consequence of Theorem~\ref{theorem:cycle_double_S_n}.
\begin{corollary}\label{corollary:even_n_cycle_completeness}
	Let $g$ be an element of $\tilde S_n$ with cycle type $(n)$ and $n$ even.
	If $V$ is an irreducible spin representation of $\tilde S_n$, then
    \begin{displaymath}
    \{\pm 1\} \times \Sp_V(g)= \{\pm i\} \times \Sp_V(g) =\Sp_V(g) \cup \Sp_{V'}(g)= \Omega(n,\varepsilon_g).
    \end{displaymath}
    \end{corollary}
\begin{lemma}
	\label{lemma:lower_bound_for_(p,q)_(p,q)}
	Let $p\geq 12$ and $q < p$ be two even positive integers.
	Then $\sigma_{(p,q)}$ has $\lcm(p,q)$ distinct eigenvalues in $V_{(p,q)}$ and each of them appears with multiplicity at least $p$.
\end{lemma}
\begin{proof}
	Using the shifted Littlewood-Richardson rule (Theorem~\ref{theorem:Littlewood-Richardson-rule}), we have 
	$
		\Res^{\tilde S_{p+q}}_{\tilde S_p \times \tilde S_q} V_{(p,q)} \geq V\otimes V_{(p-1,1)} \otimes V_{(q)}^\pm.
	$
	Using Corollary~\ref{corollary:lower_bound_non_basic}, we have that $\sigma_{(p)}$ has $p$ distinct eigenvalues in $V_{(p-1,1)}$ with each having multiplicity at least $p$.
	Lemma~\ref{lemma:spectra_of_product_of_two_modules} implies that $\sigma_{(p,q)}$ has spectrum 
    \begin{displaymath}
    \{\pm i\}\times\Sp_{V_{(p-1,1)}}(\sigma_{(p)})\times\Sp_{V_{(q)}^\pm}(\sigma_{(q)}) = \{\pm i\}\times\Sp_{V_{(p-1,1)}}(\sigma_{(p)})\times\{\zeta \mid \zeta^q=\varepsilon_{\sigma_{(q)}}\},
    \end{displaymath}
    which has $\lcm(p,q)$ distinct elements by Proposition~\ref{proposition:product_of_two_sets_cardiality_in_complex_numbers} and each of them appears with multiplicity at least $p$.
\end{proof}

\section{Proof of Theorem~\ref{theorem:main}}
\label{section:minimal_polynomial_some_family}

In this section, we prove Theorem~\ref{theorem:main}.
Let $(\rho,V)$ be an irreducible spin representation of $\tilde S_n$ indexed by a strict partition $\lambda$.
The proof is divided into two parts: first, we consider the case $\lambda=(n)$, that is, the basic spin representation; then we treat the remaining cases.

For $g\in \tilde S_n$ with cycle type $\alpha$, define
$
\varepsilon_g:=
\begin{cases}
	1 & \text{if } g^{\lcm(\alpha)}=1,\\
	-1 & \text{if } g^{\lcm(\alpha)}=-1.
\end{cases}
$

Recall that for $\varepsilon\in\{\pm1\}$, $\Omega(r,\varepsilon)$ denotes the set of complex roots of the equation $x^r-\varepsilon=0$.

\subsection{The case $\lambda=(n)$}

In this subsection, we describe the minimal polynomials of elements of $\tilde S_n$ in the basic spin representations.
\begin{lemma}\label{lemma:basic_spin_representation_eigenvalues}
Theorem~\ref{theorem:main} holds when $\lambda=  (n)$.
\end{lemma}

\begin{proof}
If $g$ is a preimage of an $n$-cycle, then the assertion follows from Theorem~\ref{theorem:cycle_double_S_n}.
Therefore, we may assume that $\ell(\alpha)\geq 2$.
If $n\leq 18$, then the minimal polynomial of $\rho(g)$ can be found using Sage, so we may assume that $n\geq 19$.
We now prove the statement by induction on $n$.

Suppose that $g$ belongs to the Young subgroup $\tilde S_{n_1}\times_c\tilde S_{n_2}$ with $n_1+n_2=n$, and that $g=g_1g_2$, where $g_1\in\tilde S_{n_1}$ and $g_2\in\tilde S_{n_2}$.
It follows from Lemma~\ref{lemma:shifted_Littlewood_Richardson_coefficients} that if $\mu\in\DP_{n_1}$ and $\nu\in\DP_{n_2}$ satisfy $f_{\mu\nu}^\lambda>0$, then $\mu=(n_1)$ and $\nu=(n_2)$ (see Fig.~\ref{f:shifted-tableaux-for-basic-spin}).
Denote by $\rho_{(n_1)}$ and $\rho_{(n_2)}$ the representations of $\tilde S_{n_1}$ and $\tilde S_{n_2}$ affording characters $\phi_{(n_1)}$ and $\phi_{(n_2)}$, respectively.

\begin{figure}[!h]
\ytableausetup{boxsize=1.3em}
\begin{center}
\begin{ytableau}
*(blue) & \ldots & *(blue) & 1 & \ldots & 1 \\
\end{ytableau}
\end{center}
\caption{Shifted tableaux for $\lambda=(n)$}
\label{f:shifted-tableaux-for-basic-spin}
\end{figure}

Since $\lambda$ is not the union of $(n_1)$ and $(n_2)$, Theorem~\ref{theorem:Littlewood-Richardson-rule} implies that the set of eigenvalues of $\rho(g)$ coincides with that of $\Sp_{V_{(n_1)} \times_c V_{(n_2)}}(g)$, or with the union of the corresponding sets for $\Sp_{({V_{(n_1)} \times_c V_{(n_2)}})^\pm}(g)$ if $n_1$ and $n_2$ have different parity.
We now consider several cases depending on $\alpha$.

{\bf Case 1.} Suppose that $3\in\alpha$, but $5$ is not the unique part of $\alpha$ divisible by $5$.
Then we may assume that $n_1=3$ and take $g_1=-\sigma_{(3)}$.
Recall that $g_1^3=1$ by Theorem~\ref{c:value_of_powers}.
This implies that $\varepsilon_g=\varepsilon_{g_2}$.
Let $\beta\vdash n_2$ be the cycle type of $g_2$, and let $t=\lcm(\beta)$.
By induction, either $\Sp_{(n_2)}(g_2)=\Omega(t,\varepsilon_g)$, or $3$ is the unique part of $\beta$ divisible by $3$ and $\Sp_{(n_2)}(g_2)=\Omega(t/3,\varepsilon_g)\times (\Omega(3,1) \setminus \{1\})$.
Note that $V_{(3)}$ is self-associate and $g_1$ is even.
Moreover, if $n_2$ is even, then
$
\Res^{\tilde S_n}_{\tilde S_{(3)}\times_c\tilde S_{(n_2)}}V_\lambda \geq V_{(3)}\otimes_c V_{(n_2)},\, V_{(3)}\otimes_c V_{(n_2)}'.
$
By Lemma~\ref{lemma:spectra_of_product_of_two_modules}, we obtain
$
\Sp_{(n)}(g)=\Sp_{(3)}(g_1)\times\Sp_{(n_2)}(g_2),
$
or
$
\Sp_{(n)}(g)=\{\pm1\}\times\Sp_{(3)}(g_1)\times\Sp_{(n_2)}(g_2)
$
if $g_2$ is odd and $n_2$ is even.
Moreover, by Table~\ref{table:exceptional_cases_1_8}, we have
$
\Sp_{(3)}(g_1)=\{\eta\in\mathbb{C}\mid \eta^3= 1,\ \eta\neq1\}.
$

First, consider the case $\Sp_{(n_2)}(g_2)=\Omega(t,\varepsilon_{g_2})$.
If $g_2$ is odd, then $t$ is even, so $\{\pm1\}\times \Omega(t,\varepsilon_g)=\Omega(t,\varepsilon_g)$.
Therefore,
$
\Sp_{(n)}(g)=\Sp_{(3)}(g_1)\times \Omega(t,\varepsilon_g),
$
regardless of the parity of $g_2$.
This implies that the minimal polynomial of $\rho(g)$ equals $x^t-\varepsilon_g$ if $3\mid t$, and $\frac{x^{3t}-\varepsilon_g}{x^t-\varepsilon_g}$ if $(3,t)=1$, as required.

It remains to consider the case when $3$ is the unique part of $\beta$ divisible by $3$.
In this case, write $g = h_1 h_2$, with $h_1$ having cycle type $(3,3)$ and $h_2$ having cycle type $\beta$ which obtained from $\alpha$ by removing all $3$'s from $\alpha$.
Using induction and the fact that $|\beta| \geq 9$, we see that $h_1$ (resp. $h_2$) has $3$ (resp. $\lcm(\beta)$) distinct eigenvalues in $V_{(9)}$ (resp. $V_{(n-9)}$). Now the result follows from Corollary~\ref{corollary:product_of_two}.

{\bf Case 2.} Suppose that $5\in\alpha$, but $3$ is not the unique part divisible by $5$. Then we can argue exactly as in Case 1 by replacing $3$ with $5$.

{\bf Case 3.} Suppose that $3,5\in\alpha$, and that all other parts of $\alpha$ are coprime to $15$.
Then we may assume that $n_1=3$ and that $g_1=-\sigma_{(3)}$.
Since $g_1^3=1$, we get $\varepsilon_g=\varepsilon_{g_2}$.
Denote by $\beta$ the partition of $n_2$ corresponding to the cycle type of $g_2$, and let $t=\lcm(\beta)$.
By assumption, $5$ is the unique part of $\beta$ divisible by $5$.

By induction, we know that $\Sp_{(n_2)}(g_2) = \Omega(t/5,\varepsilon_g) \times (\Omega(5,1)\setminus \{1\})$.
By Lemma~\ref{lemma:spectra_of_product_of_two_modules}, we find that $\Sp_{(n)}(g)=\Sp_{(3)}(g_1)\times\Sp_{(n_2)}(g_2),$ or $\Sp_{(n)}(g)=\{\pm1\}\times\Sp_{(3)}(g_1)\times\Sp_{(n_2)}(g_2)$ if $g_2$ is odd and $n_2$ is even.
If $g_2$ is odd, then $t$ is even, so $\{\pm1\}\times\bigl(\Omega(t/5,\varepsilon_g) \times (\Omega(5,1)\setminus \{1\}))\bigr)=\Omega(t/5,\varepsilon_g) \times (\Omega(5,1)\setminus \{1\}).$
Therefore, $\Sp_{(n)}(g)=\Sp_{(3)}(g_1)\times\bigl(\Omega(t/5,\varepsilon_g) \times (\Omega(5,1)\setminus \{1\}))\bigr)$
regardless of the parity of $g_2$.
Since $\Sp_{(3)}(g_1)=\Omega(3,1)\setminus\{1\}$, it follows that $\Sp_{(n)}(g)=\frac{(x^{k}-\varepsilon_g)(x^{k/15}-\varepsilon_g)}{(x^{k/3}-\varepsilon_g)(x^{k/5}-\varepsilon_g)}$.

{\bf Case 4.} Finally, suppose that $3,5\notin\alpha$.
We start with the case that $\alpha $ has an odd part $n_1$.
Write $g = g_1 g_2$ with the cycle type of $g_1$, and $g_2$ are $(n_1)$ and $\beta$, respectively, such that $\alpha = (n_1) \cup \beta$ and the order of $g_1$ is $n_1$.
By induction, either $\Sp_{V_{(n-n_1)}}(g_2)$ has $\lcm(\beta)$ many distinct eigenvalues or $\beta = (n_2)$ and $n_2 \in \{2,4,6,8\}$.
By induction, $\Sp_{V_{(n_1)}}(g_1)$ has $n_1$ distinct elements.
So the result follows in the former case by Corollary~\ref{corollary:product_of_two}.
In the latter case, we see that $\Sp_{(n)}(g) = \{\pm 1\} \times \Sp_{{(n_1)}}(g_1) \times \Sp_{{(n_2)}}(g_2)$.
Using Corollary~\ref{corollary:even_n_cycle_completeness}, $\{\pm 1\} \times   \Sp_{(n_2)}(g_2) = \Omega(n_2,\varepsilon_g)$.
Therefore, we get that $\Sp_{(n)}(g) = \Omega(n_1,1) \times \Omega(n_2,\varepsilon_g)$.
Now the result follows from Proposition~\ref{proposition:product_of_two_sets_cardiality_in_complex_numbers}.

Therefore, we may assume that all parts of $\alpha$ are even.
Let $n_1\in \alpha$ be an even part.
Denote by $\beta$ the partition of $n_2$ corresponding to the cycle type of $g_2$, and let $t=\lcm(\beta)$.
Then $n_1$ and $n_2$ are even, while $g_1$ is odd.
By Lemma~\ref{lemma:spectra_of_product_of_two_modules}, we find that $\Sp_{(n)}(g)=\{\pm i\}\times\Sp_{(n_1)}(g_1)\times\Sp_{(n_2)}(g_2)$
if $g_2$ is odd, and $\Sp_{(n)}(g)=\{\pm1\}\times\Sp_{(n_1)}(g_1)\times\Sp_{(n_2)}(g_2)$ if $g_2$ is even.
Using Corollary~\ref{corollary:even_n_cycle_completeness} and induction, we obtain that $\Sp_{(n)}(g) = A \times \Omega(n_1,\pm1) \times \Omega(t,\pm 1)$, where $A = \{\pm 1\} \text{ or } \{\pm i\}$ and both $\pm1$ depends on $\varepsilon_g,\varepsilon_{g_1}$ and $\varepsilon_{g_2}$.
In any case, using Corollary~\ref{corollary:product_of_two}, it follows that $\Sp_{(n)}(g)$ has $\lcm(\alpha)$ many distinct elements.
This completes the proof.
\end{proof}

\subsection{The general case}

In this subsection, we finish the proof of Theorem~\ref{theorem:main}. 
We assume that $\rho$ (or $\rho^\pm$) is 
an irreducible spin representation of $\tilde S_n$ 
corresponding to $\lambda\in\DP_n$.
Suppose that $g\in \tilde S_n$ has cycle type $\alpha$. 
By Theorem~\ref{theorem:cycle_double_S_n} and Lemma~\ref{lemma:basic_spin_representation_eigenvalues},
we can assume that $\alpha\neq(n)$ and $\lambda\neq(n)$.
Using Sage, we verify that Theorem~\ref{theorem:main}
is true for all $n\leq 13$, so we assume that $n\geq14$.
We split the proof into several cases for $\alpha$ that are motivated by exceptions in Theorem~\ref{theorem:main} (compare with Lemma~\ref{lemma:basic_spin_representation_eigenvalues}):

\begin{itemize}
\item $3\in\alpha$ but it is not true that $5$ is the unique part of $\alpha$ divisible by 5;
\item $5\in\alpha$ but it is not true that $3$ is the unique part of $\alpha$ divisible by 3;
\item $\alpha$ contains $3$ and $5$, while other parts of $\alpha$ are not divisible by $3$ and $5$;
\item $\alpha$ contains neither 3 nor 5.
\end{itemize}

In fact, first two cases are similar, and we consider them together.
On the other hand, the third case differs for $\lambda=(n-1,1)$ and
$\lambda\neq(n-1,1)$ in the statement of Theorem~\ref{theorem:cycle_double_S_n}, so we consider these two subcases separately. Thus, the proof of the theorem is divided into the following four lemmas.

\begin{lemma}\label{lemma:general-case-part1}
Suppose that $\rho$ is 
an irreducible spin representation of $\tilde S_n$ 
corresponding to $\lambda\in\DP_n$,
where $n\geq14$ and $\lambda\neq(n)$. 
Let $g\in \tilde S_n$ have cycle type $\alpha$ and $k=\lcm(\alpha)$.
If one of the following statements holds
\begin{itemize}
\item $3\in\alpha$ but it is not true that $5$ is the unique part of $\alpha$ divisible by $5$;
\item $5\in\alpha$ but it is not true that $3$ is the unique part of $\alpha$ divisible by $3$, 
\end{itemize}
then the minimal polynomial of $\rho(g)$ equals 
$x^{k}-\varepsilon_g$.
\end{lemma}
\begin{proof}
Suppose that $3\in\alpha$ but it is not true that $5$ is the unique part of $\alpha$ divisible by 5. Consider $g$ as an element of the Young subgroup $S_3\times_c S_{n-3}$. Write $g=g_1g_2$,
$g_1=-\sigma_{(3)}$ is a preimage of $(1,2,3)$ in $\tilde S_n$
and $g_2\in\tilde S_{n-3}$ is the corresponding element in the second factor. Note that $g_1^3=1$, so $\varepsilon_g=\varepsilon_{g_2}$.
Denote by $\beta$ the partition of $n-3$ corresponding to the cycle type of $g_2$. Let $t=\lcm(\beta)$ and $\omega=e^{2\pi i/3}$ be a primitive cube root of unity.
According to Table~\ref{table:exceptional_cases_1_8},
we see that $\Sp_{(3)}(g_1)=\{\omega,\overline{\omega}\}$
and $\Sp_{(2,1)^\pm}(g_1)=\{1\}$. 
By Lemma~\ref{lemma:most_dominant_alpha_beta},
there exist strict partitions $\mu,\nu\in\DP_{n-3}$ such that
$f^\lambda_{(3)\mu}, f^\lambda_{(2,1)\nu}>0$. 
It follows from Theorem~\ref{theorem:Littlewood-Richardson-rule} that
$V_{(3)}\otimes_c V_\mu,V_{(2,1)}^\epsilon\otimes_c
V_{\nu}\leq\Res^{\tilde S_n}_{\tilde S_{3} \times_c \tilde S_{n-3}}V_\lambda$, where $\epsilon\in\{+,-\}$.
Denote these reduced Clifford products by 
$U_1$ and $U_2$, respectively. 

Assume that it is not true that $3$ is the unique part of $\beta$
dividing by 3.
Using induction and $n-3>8$, $\Sp_\mu(g_2)=\Sp_\nu(g_2)=\Omega(t,\varepsilon_g)$.
It follows from Lemma~\ref{lemma:spectra_of_product_of_two_modules}
that $\Sp_{U_1}(g) = \{\omega,\overline{\omega}\}\times \Omega(t,\varepsilon_g)$ 
regardless of the parity of $g_2$.
Similarly, since $\Sp_{(2,1)}(g_1)=\{1\}$, we see that $\Sp_{U_2}(g) = \{1\} \times \Omega(t,\varepsilon_{g})$.
Therefore, $\Sp_{\lambda}(g) \supseteq \Omega(3,1) \times \Omega(t,\varepsilon_g)$, which has $\lcm(\alpha)$ many distinct elements by Proposition~\ref{proposition:product_of_two_sets_cardiality_in_complex_numbers}. So the result follows in this case.


Therefore, we may assume that $3\in\beta$ and other parts of $\beta$
are not divisible by 3. Then $\lcm(\alpha)=t$.
By induction and $n-3>8$, $ \{w,\bar{w}\} \times \Omega(t/3,\epsilon_g) $ contained in both  $\Sp_\mu(g_2)$ and $\Sp_\nu(g_2)$.
Arguing as in the previous paragraph, we obtain that $\Sp_\lambda(g) \supset \Omega(3,1) \times \{w,\bar{w}\} \times \Omega(t/3,\epsilon_g) $.
Since $\Omega(3,1)  \times \{w,\bar{w}\} \times \Omega(t/3,\epsilon_g)= \Omega(3,1)  \times \Omega(t/3,\epsilon_g) $, which by Proposition~\ref{proposition:product_of_two_sets_cardiality_in_complex_numbers}, has $t$ ($=\lcm(\alpha)$) many distinct elements and we are done.

Suppose  that $5\in\alpha$ but it is not true that $3$ is the unique part divisible by 3. We can write $g=g_1g_2$, where $g_1=-\sigma_{(5)}$
is a preimage of the cycle $(1,2,3,4,5)$ and $g_2\in\tilde{S}_{n-5}$. By Theorem~\ref{c:value_of_powers}, we see that $g_1^5=1$, so $\varepsilon_g=\varepsilon_{g_2}$.
Denote by $\beta$ the partition of $n-5$ corresponding to the cycle type of $g_2$. Let $t=\lcm(\beta)$. By Lemma~\ref{lemma:most_dominant_alpha_beta} and Theorem~\ref{theorem:Littlewood-Richardson-rule},
there exist $\epsilon\in\{\pm\}$ and 
$\mu\in\DP_{n-5}$ such that $V_{(4,1)}^\epsilon\otimes_c V_\mu\leq
\Res^{\tilde S_n}_{\tilde S_{5} \times_c \tilde S_{n-5}}V_\lambda$.
By induction, $\Sp_{(4,1)}(g_1)$ and $\Sp_\mu(g_2)$ has $5$ and $\lcm(\mu)$ many distinct elements. It follows from Corollary~\ref{corollary:product_of_two} that $\Sp_\lambda(g)$ has $\lcm(\sigma)$ many distinct elements.
This completes the proof.
\end{proof}

\begin{lemma}
Suppose that $\rho$ is 
an irreducible spin representation of $\tilde S_n$ 
corresponding to $\lambda=(n-1,1)$,
where $n\geq14$. 
Let $g\in \tilde S_n$ have cycle type $\alpha$ and $k=\lcm(\alpha)$.
If $\alpha$ contains $3$ and $5$, while other parts of $\alpha$ are not divisible by $3$ and $5$, then the minimal polynomial of $\rho(g)$ equals $\frac{(x^{k}-\varepsilon_g)}{(x^{k/15}-\varepsilon_g)}$.
\end{lemma}
\begin{proof}
We prove the statement using induction on $n$.
Consider $g$ as an element of the Young subgroup $S_3\times_c S_{n-3}$. Write $g=g_1g_2$,
$g_1=-\sigma_{(3)}$ is a preimage of $(1,2,3)$ in $\tilde S_n$
and $g_2\in\tilde S_{n-3}$ is the corresponding element in the second factor. By Theorem~\ref{c:value_of_powers}, we see that $g_1^3=1$, so $\varepsilon_g=\varepsilon_{g_2}$. 
Denote by $\beta$ the partition of $n-3$ corresponding to the cycle type of $g_2$. Let $t=\lcm(\beta)$. By assumption, we can write $\lcm(\beta)=5l$, where $l$ is an integer such that $(l,15)=1$. 

It is easy to see that there are only four shifted tableaux of shape  $\lambda/\mu$ with content $\nu$ such that $f^\lambda_{\mu\nu}>0$:
two for $(\mu,\nu)=((3),(n-3))$,
one for $(\mu,\nu)=((3),(n-4,1))$, and 
one for $(\mu,\nu)=((2,1),(n-3))$ (see, Fig.\ref{f:shifted-tableaux-for-basic-standard-spin}).

\begin{figure}[!h]
\ytableausetup{boxsize=1.3em}

\begin{ytableau}
*(blue) & *(blue) & *(blue) & 1 & \ldots & 1 \\
\none & 1
\end{ytableau}
\quad
\begin{ytableau}
*(blue) & *(blue) & *(blue) & 1' & \ldots & 1 \\
\none & 1
\end{ytableau}
\quad
\begin{ytableau}
*(blue) & *(blue) & *(blue)  & 1 & \ldots & 1 \\
\none & 2
\end{ytableau}
\quad
\begin{ytableau}
*(blue) & *(blue) & 1 & \ldots & 1 \\
\none & *(blue)
\end{ytableau}
\caption{Shifted tableaux for $\lambda=(n-1,1)$}
\label{f:shifted-tableaux-for-basic-standard-spin}
\end{figure}

By Theorem~\ref{theorem:Littlewood-Richardson-rule},
we infer that each irreducible constituent of 
$\Res^{\tilde S_n}_{\tilde S_3 \times_c \tilde S_{n-3}} V_{(n-1,1)}$ is equivalent to one of the following (up to associate):
$V_{(3)}\otimes_c V_{(n-3)}$,
$V_{(2,1)}\otimes_c V_{(n-3)}$, or
$V_{(3)}\otimes_c V_{(n-4,1)}$, moreover each of this reduced Clifford product (and their associates) appears in the restriction since $\lambda\neq\mu\cup\nu$ in all these cases. This implies that 
to find the set of eigenvalues of $\rho(g)$,
we need to find eigenvalues of $g$ 
in each of these submodules.

First, consider the module $V_{(3)} \otimes_c V_{(n-3)}$.
By induction, $\Sp_{(n-3)}(g_2) = (\Omega(5,1)\setminus \{1\}) \times  \Omega(l,\epsilon_g)$.
Since $\Sp_{(3)}(g_1) = \Omega(3,1) \setminus \{1\}$, using Lemma~\ref{lemma:spectra_of_product_of_two_modules}, we obtain that 
$\Sp_{V_{(3)} \otimes_c V_{(n-3)}} (g) = (\Omega(3,1) \setminus \{1\}) \times (\Omega(5,1) \setminus \{1\}) \times \Omega(l,\epsilon_g)$ regardless of the parity of $g_2$.

Second, consider the module $V_{(3)} \otimes_c V_{(n-4,1)}$.
By induction, $\Sp_{(n-4,1)}(g_2) = \Omega(5,1) \times \Omega(l,\epsilon_g)$.
Since $\Sp_{(3)}(g_1) = \Omega(3,1) \setminus \{1\}$, using Lemma~\ref{lemma:spectra_of_product_of_two_modules}, we obtain that 
$\Sp_{V_{(3)} \otimes_c V_{(n-3)}} (g) = (\Omega(3,1) \setminus \{1\}) \times \Omega(5,1) \times \Omega(l,\epsilon_g)$ regardless of the parity of $g_2$.

Finally, we consider the module $V_{(2,1)} \otimes_c V_{(n-3)}$.
By induction, $\Sp_{(n-3)}(g_2) = (\Omega(5,1) \setminus \{1\}) \times \Omega(l,\epsilon_g)$.
Since $\Sp_{(2,1)}(g_1) = \{1\}$, using Lemma~\ref{lemma:spectra_of_product_of_two_modules}, we obtain that 
$\Sp_{V_{(2,1)} \otimes_c V_{(n-3)}} (g) =  \{1\} \times (\Omega(5,1) \setminus \{1\}) \times \Omega(l,\epsilon_g)$ regardless of the parity of $g_2$.
Hence, we obtain that 
\begin{align*}
    \Sp_{(n-1,1)}(g) = \big((\Omega(3,1) \setminus \{1\}) \times \Omega(5,1) \times \Omega(l,\epsilon_g) \big) \cup 
    \big( \{1\} \times (\Omega(5,1) \setminus \{1\}) \times \Omega(l,\epsilon_g) \big).
\end{align*}
Therefore, the minimal polynomial of $g$ in $V_{(n-1,1)}$ is $\frac{x^k-1}{x^{k/15}-1}$.
\end{proof}

\begin{lemma}
Suppose that $\rho$ is 
an irreducible spin representation of $\tilde S_n$ 
corresponding to $\lambda\in\DP_n$,
where $n\geq14$ and $\lambda\neq(n),(n-1,1)$. 
Let $g \in \tilde S_n$ have cycle type $\alpha$ and $k=\lcm(\alpha)$.
If $3,5\in\alpha$ and other parts of $\alpha$ are coprime to $15$,
then the minimal polynomial of $\rho(\sigma)$ equals 
$x^{k}-\varepsilon_g$.
\end{lemma}
\begin{proof}
Consider $g=g_1g_2$ as an element of the Young subgroup $\tilde S_{3}\times_c\tilde S_{n-3}$ where $g_1=-\sigma_3\in\tilde S_{3}$, $g_2\in\tilde S_{n-3}$.
Let $\beta$ be the cycle type of $g_2$ and $k=\lcm(\beta)$.
Using Lemma~\ref{lemma:(3)_(2,1)_one_different_from_basic}, there exist partitions $\mu,\nu \in \DP_{n-3}$, different from $(n-3)$, such that $f_{(3)\mu}^\lambda>0$ and $f_{(2,1)\nu}^\lambda>0$.
Therefore, $\Res^{\tilde S_n}_{\tilde S_3 \times_c \tilde S_{n-3}} V_\lambda \geq V_{(3)} \otimes_c V_\mu, V_{(2,1)} \otimes_c V_\nu$.
By induction and $n-3>8$, we obtain that $\Sp_\mu(g_2) = \Sp_\nu(g_2) = \Omega(k, \varepsilon_g)$.

Recall that $\Sp_{(3)}(g_1) = \Omega(3,1) \setminus \{1\}$.
Using Lemma~\ref{lemma:spectra_of_product_of_two_modules}, we see that $\Sp_{V_{(3)} \otimes V_\mu}(g) = \big( \Omega(3,1) \setminus \{1\} \big) \times \Omega(t,\varepsilon_g)$ regardless of parity of $g_2$.
Similarly, using Lemma~\ref{lemma:spectra_of_product_of_two_modules}, we get that $\Sp_{V_{(2,1)} \otimes V_\mu}(g) = \{1\} \times \Omega(t,\varepsilon_g)$ regardless of parity of $g_2$.
Hence, $\Sp_\lambda(g) = \Omega(3,1) \times \Omega(t,\varepsilon_g) = \Omega(3t,\varepsilon_g)$.
This completes the proof.
\end{proof}

\begin{lemma}
Suppose that $\rho$ is 
an irreducible spin representation of $\tilde S_n$ 
corresponding to $\lambda\in\DP_n$,
where $n\geq14$ and $\lambda\neq(n)$. 
Let $g\in \tilde S_n$ have cycle type $\alpha$ and $k=\lcm(\alpha)$.
If $3,5\not\in\alpha$, then the minimal polynomial of $\rho(g)$ equals $x^{k}-\varepsilon_g$.
\end{lemma}
\begin{proof}
We start by verifying the lemma for $n=15,16$ directly using Sage.
Therefore, we may assume that $n\geq 17$.
Let $n_1$ be the least part of $\alpha$ and therefore, $n_2:= n-n_1\geq 9$.
Consider $g$ as an element of the Young subgroup $\tilde S_{n_1}\times_c\tilde S_{n_2}$ with $n_1+n_2=n$
and $g=g_1g_2$, where $g_1\in\tilde S_{n_1}$ and  $g_2\in\tilde S_{n_2}$. 
We take $g_1 = \sigma_{n_1}$. Let $\beta$ be the cycle type of $g_2$ and $k=\lcm(\beta)$.
Let $\mu \in \DP_{n_1}$ be the partition contained in $\lambda$ which is maximum in lexicographic order.
Using Lemma~\ref{lemma:most_dominant_alpha_beta}, there exists a partition $\nu \in \DP_{n_2}$ such that $f_{\mu\nu}^\lambda>0$.
By induction and $n_2 \geq 9$, we obtain that $\Sp_{\nu}(g_2) = \Omega(k,\varepsilon_{g_2})$.

Notice that since $\lambda_1\geq 6$, either $\mu = (n_1)$ or $\mu_1\geq 6$.
Therefore, by induction, we get that either $\Sp_{\mu}(g_1) = \Omega(n_1,\varepsilon_{g_1})$ or $\mu = (n_1)$ and $n_2 \in \{2,4,6,8\}$.
In the former case, the theorem follows by Corollary~\ref{corollary:product_of_two}.
In the latter case, we can simply change $\mu$ with $(n_1-1,1)$ when $n_1\in \{4,6,8\}$.
By Lemma~\ref{lemma:most_dominant_alpha_beta}, there exists $\nu \in \DP_{n_2}$ such that $f_{\mu\nu}^\lambda>0$.
By induction, it follows that $\Sp_{\mu}(g_1) = \Omega(n_1,\varepsilon_{g_1})$ and $\Sp_{\nu}(g_2) = \Omega(t,\varepsilon_{g_2})$. Now the result follows from Corollary~\ref{corollary:product_of_two}.

The remaining case is $n_2=2$.
Using Lemma~\ref{lemma:spectra_of_product_of_two_modules}, we have
$\Sp_{V_{(2)} \otimes_c V_\nu}(g) = A \times \Sp_{(2)}(g_1) \times \Sp_{\nu}(g_2) $, where $A \in\{ \{1\}, \{\pm 1\} , \{\pm i\}\}.$

If $\beta$ has an even part, then by induction, we have $\Sp_\nu(g_2) = \{\pm 1\} \times \Omega(t,\varepsilon_{g_2})$.
Therefore, $\Sp_{V_{(2)} \otimes_c V_\nu}(g) = A \times \{\pm i\} \times \Sp_{\nu}(g_2) $. Now the result follows from Proposition~\ref{proposition:product_of_two_sets_cardiality_in_complex_numbers}.

If $\lambda \neq (2) \cup \beta$, then using Theorem~\ref{theorem:Littlewood-Richardson-rule}, we obtain that $\Res^{\tilde S_n}_{\tilde S_2 \times_c \tilde S_{n-2}} V_\lambda \geq ((V_{(2)} \otimes_c V_\nu) \oplus (V_{(2)}' \otimes_c V_\nu))$.
Hence, $\Sp_\lambda (g) \supseteq A \times \{\pm i\} \times \Omega(t,\varepsilon_{g_2})$ and we are done by Proposition~\ref{proposition:product_of_two_sets_cardiality_in_complex_numbers}.

Finally, we left out with $\lambda = (2) \cup \beta$ with $\beta \in \OP_n$.
Let $\gamma:= (\lambda_1-1,\lambda_2,\dots,\lambda_{\ell-1},\lambda_\ell-1) \in \DP_{n-2}$, where $\ell$ is the length of $\lambda$.
Since $\gamma\subset \lambda$, by Lemma~\ref{lemma:most_dominant_alpha_beta}, $f_{(2)\gamma}^\lambda>0$.
Now we apply the previous case, as $\gamma \cup (2) \neq \lambda$.
This completes the proof.
\end{proof}

\section{Proof of results on double cover of alternating groups}
\label{section:minimal_poly_of_any_element_in_A_n}
The goal of this section is to prove Theorem~\ref{theorem:main_alt}.

\begin{notation}
	For the double cover of a Young subgroup
	$
	\tilde A_{\lambda_1}\times_c \cdots \times_c \tilde A_{\lambda_l},
	$
	the irreducible representations are not, in general, ordinary tensor products, since this subgroup is a central product rather than a direct product.
    However, by convention, if $V_i$ is an irreducible representation of $\tilde A_{\lambda_i}$ for each $i$, we write
	$
	V := V_1 \otimes \cdots \otimes V_l
	$
	for the corresponding irreducible representation of the central product $\tilde A_{\lambda_1}\times_c \cdots \times_c \tilde A_{\lambda_l}$, with the central element $-1$ acting as $-I$ on $V$.
\end{notation}

The following lemmas are used in the proof of Theorem~\ref{theorem:main_alt} for $n$-cycles.
As in Subsection~\ref{subsection:arbitrary_n_cycles}, for $n=mp$ with $p$ prime, let $\eta_n$ be an $n$-cycle such that $\eta_n^p=w_1w_2\cdots w_p$, where $\pi(w_i)=(m(i-1)+1~m(i-1)+2~\dots~m(i-1)+m)$. Define $E_n=\langle \eta_n \rangle.$
\begin{lemma}\label{lemma:cycle_linear_characters_two_distinct_A_n}
	Let $n=mp>27$ be an odd integer, where $m\in\{7,9\}$ and $m \geq p>1$.
	Suppose that $V = V_{(m)}^{i_1} \otimes V_{(m)}^{i_2} \otimes \cdots \otimes V_{(m)}^{i_p}$ is an irreducible spin representation of $\tilde A_{(m^{ p})}$, where $i_j \in \{+,-\}$ for all $j\in [p]$.
	Let $\delta$ be a linear character of $E_n=\langle \eta_n\rangle$.
	Then there exist linear characters $\delta_1,\delta_2,\ldots,\delta_p$ of $C_m$, with at least two of them distinct, such that $\Res^{\tilde A_m}_{C_m} V_{(m)}^{i_j} \geq \delta_j$ for all $j\in [p]$, and
	\begin{displaymath}
	\Res^{E_n}_{\langle \eta_n^p \rangle} \delta = \Res^{C_m \times \cdots \times C_m}_{\langle w_1w_2\cdots w_p \rangle} \delta_1 \otimes \cdots \otimes \delta_p,
	\end{displaymath}
	where $\eta_n^p=w_1w_2\cdots w_p$.
\end{lemma}

\begin{proof}
	Let $\omega_m = e^{2\pi i/m}$ be the primitive $m$-th root of unity.
	To prove the lemma, it is enough to find $m$-th roots of unity $\zeta_1,\ldots,\zeta_p$, with at least two distinct, such that $\zeta_1\zeta_2\cdots \zeta_p = \delta(\eta_n^p)$ and $\zeta_j$ is an eigenvalue of $w_j$ in $V_{(m)}^{i_j}$ for all $j\in [p]$.


	We begin with the case $m=7$, so $p\in\{5,7\}$.
	The minimal polynomial of $\sigma_7^{+\pm}$ in $V_{(7)}^\pm$ is given by:
	\begin{center}
		\begin{tabular}{c|c|c}
			 & $\sigma_7^{++}$ & $\sigma_7^{+-}$ \\
			\hline
			$V_{(7)}^+$ & $(x-1)(x-\omega_7^3)(x-\omega_7^5)(x-\omega_7^6)$ & $(x-1)(x-\omega_7)(x-\omega_7^2)(x-\omega_7^4)$ \\
			\hline
			$V_{(7)}^-$ & $(x-1)(x-\omega_7)(x-\omega_7^2)(x-\omega_7^4)$ & $(x-1)(x-\omega_7^3)(x-\omega_7^5)(x-\omega_7^6)$
		\end{tabular}
	\end{center}
	Let $\zeta:= \delta(\eta_n^p)$.
	A direct computation shows that there exist $7$-th roots of unity $\zeta_1,\zeta_2,\zeta_3$, with at least two of them distinct, such that $\zeta=\zeta_1\zeta_2\zeta_3$ and $\zeta_j$ an eigenvalue of $w_j$ in $V_{(7)}^{i_j}$ for all $j\in [3]$.
	We set $\zeta_j = 1$ for all $4 \leq j \leq 7$.
	Then $\{\zeta_j\}_{j=1}^p$ satisfies the required condition.

	Finally, consider the case $m=9$, so $p\in\{5,7\}$.
	The minimal polynomial of $\sigma_9^{+\pm}$ in $V_{(9)}^\pm$ is given by:
	\begin{center}
		\begin{tabular}{c|c|c}
			 & $\sigma_9^{++}$ & $\sigma_9^{+-}$ \\
			\hline
			$V_{(9)}^+$ & $\tfrac{x^9-1}{x^2+x+1}$ & $\tfrac{x^9-1}{x-1}$ \\
			\hline
			$V_{(9)}^-$ & $\tfrac{x^9-1}{x-1}$ & $\tfrac{x^9-1}{x^2+x+1}$
		\end{tabular}
	\end{center}
	Let $\zeta:= \delta(\eta_n^p)$.
	A direct computation shows that there exist $9$-th roots of unity $\zeta_1,\zeta_2,\zeta_3$, with at least two of them distinct, such that $\zeta=\zeta_1\zeta_2\zeta_3$ and $\zeta_j$ an eigenvalue of $w_j$ in $V_{(9)}^{i_j}$ for all $j\in [3]$.
	For the remaining factors, we set $\zeta_4(w_4) = \zeta_6 (w_6) = \cdots = \zeta_{p-1} (w_{p-1}) = \omega_m$ and $\zeta_5(w_5) = \zeta_7(w_7) = \cdots = \zeta_p(w_p) = \omega_m^{-1}$.
	Then $\{\zeta_j\}_{j=1}^p$ satisfies the required condition.
\end{proof}

The following corollary is an immediate consequence of Proposition~\ref{proposition:product_of_two_sets_cardiality_in_complex_numbers}.
\begin{lemma}
	\label{lemma:alt_full_times_full_is_full}
	Let $m,k\in\mathbb{N}$ and $n=m+k$.
	Let $V_1$ and $V_2$ be irreducible spin modules for $\tilde A_m$ and $\tilde A_k$, respectively.
    Suppose that $\pi_1\in \tilde A_m$ and $\pi_2 \in \tilde A_k$.
	If $\pi_i$ has $o(\pi_i)$ distinct eigenvalues in $V_i$ for $i=1,2$, then $\pi_1\pi_2$ has $o(\pi_1\pi_2)$ distinct eigenvalues in the irreducible module $V_1 \otimes V_2$ for $\tilde A_m \times_c \tilde A_k$.
\end{lemma}

The following reduction lemma simplifies the computation of minimal polynomials in most cases.

\begin{lemma}\label{lemma:double_A_n_not_DP_n^-}
	Let $g \in \tilde A_n$ with cycle type $\alpha$, and let $\lambda \in\DP_n$.
	Then, for any irreducible spin representation $V_\lambda$ of $\tilde A_n$, the minimal polynomial of $g$ in $V_\lambda$ is the same as in the irreducible representation $V$ of $\tilde S_n$ containing $V_\lambda$ upon restriction to $\tilde A_n$, provided either $\lambda \notin\DP_n^+$ or $\lambda \in\DP_n^+$ with $\lambda \neq \alpha$.
\end{lemma}

\begin{proof}
	Let $\phi_\lambda$ (respectively, $\phi_\lambda^\pm$) be the character of the irreducible representation $V_\lambda$ (respectively, $V_\lambda^\pm$) of $\tilde A_n$.
	By Lemmas~\ref{lemma:conjugacy_classes_A_n} and~\ref{lemma:restriction_to_A_n}, and Theorem~\ref{theorem:Delta_details_double_alternating}, we have that $\phi_\lambda(g^i)$ (respectively, $\phi_\lambda^\pm(g^i)$) is either equal to $\phi_\lambda(g^i)$ or to $\phi_\lambda(g^i)/2$ for all $i$.
	By computing inner products of characters, the minimal polynomial of $g$ in $V_\lambda$ (or $V_\lambda^\pm$) is identical to the minimal polynomial of $g$ in $V_\lambda$ viewed as a representation of $\tilde S_n$.
	This completes the proof.
\end{proof}

In order to prove Theorem~\ref{theorem:main_alt}, by Lemma~\ref{lemma:double_A_n_not_DP_n^-}, it suffices to consider the case $\lambda = \alpha \in \DP_n^+$.
We split the proof into three cases: first $\alpha=(n)$, second all parts of $\alpha$ even, finally the general case.

When $n$ is prime, we have the following result, proved by Zalesskii~\cite{Zalesski_quasi_cyclic_sylow_1999} in a more general context concerning minimal polynomials of elements of prime order in quasi-simple groups.

\begin{lemma}
	\label{lemma:zalesskii_prime_alt}
	Let $g \in \tilde A_p$ with cycle type $(p)$, where $p\geq 7$ is prime, such that the order of $g$ is~$p$.
	Then the minimal polynomial of $g$ is $x^p-1$ in all irreducible spin representations $V$ of $\tilde A_p$, except when $V=V_{(7)}^\pm$.
\end{lemma}

\begin{proof}
	We may assume that $V$ is irreducible.
	Using Sage, we may verify the lemma for primes $p\leq 13$, so we assume that $p\geq 17$.
	Suppose that $V$ is an irreducible spin representation of $\tilde S_p$ with character $\phi$.
	Let $\psi$ be an irreducible spin character of $\tilde S_p$ such that $\psi\leq \Ind_{\tilde A_p}^{\tilde S_p} \phi$.
	If $V\neq V_{(p)}^\pm$, then using Theorem~\ref{theorem:Delta_details_double_alternating}, we have $\phi(g^i)=\psi(g^i)/2$ or $\psi(g^i)$ for all $i$.
	Now the result follows from a straightforward inner-product computation and Lemma~\ref{lemma:zalesskii_prime_sym}.

	Now suppose that $V=V_p^+$ or $V_p^-$.
	Let $\phi$ denote its character.
	Then
	\begin{displaymath}
	\phi(g^i)=\frac{\pm 1 \pm i^{(p-1)/2} \sqrt{p}}{2}.
	\end{displaymath}
	In particular, $|\phi(g^i)|\leq \frac{1+\sqrt{p}}{2}$.
	Also, $\phi(1)=2^{\frac{p-3}{2}}> p \sqrt{p}$ for all $p\geq 17$.
	Then, for any irreducible representation $\delta$ of $\langle g \rangle$, we get
	\begin{align*}
	\langle \Res^{\tilde A_p}_{\langle g \rangle} \phi, \delta \rangle 
	& \geq 2^{\frac{p-3}{2}} - (p-1) \tfrac{(\sqrt{p}+1)}{2}
	\geq 2^{\frac{p-3}{2}} - p\sqrt{p}
	>0.
	\end{align*}
	This completes the proof.
\end{proof}

\begin{theorem}\label{theorem:cycle_double_A_n_odd}
Let $n\geq 3$ be an odd integer and $\lambda \in \DP_n^+$.
Then the minimal polynomial of any $n$-cycle in $V_{(n)}^\pm$ has $n$ distinct eigenvalues, except for the cases listed in Table~\ref{table:exceptional_cases_1_8_alt}.
\end{theorem}
\begin{proof}
	The proof is by induction on $n$. We can verify the theorem directly using Sage for $n\leq 27$, so we assume that $n>27$.
	If $n$ is prime, the result follows from Lemma~\ref{lemma:zalesskii_prime_alt}.
	Otherwise, write $n=mp$, where $p$ is the smallest prime divisor of $n$.
	In particular, $m\geq 7$. 
    Let $V = V_{(n)}^+$ or $V_{(n)}^-$.
	Then
	\begin{displaymath}
	\Res^{\tilde A_n}_{\tilde A_{(m^{p})}} V \geq V_1 \otimes V_2 \otimes \cdots \otimes V_p,
	\end{displaymath}
	for some irreducible spin representations  $V_j$  of $\tilde A_m$ for all $j\in [p]$.
	By repeated use of the shifted Littlewood--Richardson rule~\ref{theorem:Littlewood-Richardson-rule}, each $V_j$ is either $V_{(m)}^+$ or $V_{(m)}^-$.

    It suffices to show that the $n$-cycle $\eta_n$ has $n$-distinct eigenvalues in $V$.
	Let $\delta$ be a linear character of $E_n=\langle \eta_n\rangle$.
	By induction and Lemmas~\ref{lemma:m_odd_p_nontrivial_atleast_two_distinct} and~\ref{lemma:cycle_linear_characters_two_distinct_A_n}, there exist linear characters $\delta_1,\delta_2,\ldots,\delta_p$ of $C_m$, with at least two distinct, such that $\Res^{\tilde A_m}_{C_m} V_j \geq \delta_j$ for all $j\in [p]$ and
	\begin{align}\label{equation:restriction_to_C_m_w_1w_2cdots w_p_alt}
		\Res^{C_n}_{\langle \eta_n^p \rangle} \delta
		=
		\Res^{C_m \times \cdots \times C_m}_{ \langle w_1w_2\cdots w_p \rangle}
		\delta_1 \otimes \cdots \otimes \delta_p,
	\end{align}
	where $\eta_n^p=w_1w_2\cdots w_p$.

	We proceed as in Case 1 of the proof of Theorem~\ref{theorem:cycle_double_S_n}.
	Since $\Res^{\tilde A_{(m^p)}}_{C_m \times \cdots \times C_m} V_1 \otimes V_2 \otimes \cdots \otimes V_p \geq \delta_1 \otimes \cdots \otimes \delta_p$ and $\Ind_{C_m \times \cdots \times C_m}^{C_m \wr C_p} \delta_1 \otimes \cdots \otimes \delta_p$ is irreducible, we have 
	\begin{displaymath}
		\Res^{\tilde A_n}_{C_m \wr C_p} V
		\geq
		\Ind_{C_m \times \cdots \times C_m}^{C_m \wr C_p}
		\delta_1 \otimes \cdots \otimes \delta_p.
	\end{displaymath}

	Restricting further to $\langle \eta_n \rangle$ and applying~\eqref{equation:restriction_to_C_m_w_1w_2cdots w_p_alt}, we get
	\begin{displaymath}
		\Res^{\tilde A_n}_{\langle \eta_n \rangle} V \geq \delta.
	\end{displaymath}
	Thus, the minimal polynomial of $\eta_n$ in $V$ equals $x^n-1$.
\end{proof}

To proceed with the proof of the general case, we need the following lemmas.

\begin{lemma}~\label{lemma:main_alt_two_even_cycles}
	Theorem~\ref{theorem:main_alt} holds when $\lambda=\alpha=(p,q)$ and $p,q$ are even positive integers.
\end{lemma}

\begin{proof}
	If $p\leq 10$, there are only finitely many such $\lambda$, and we can verify these cases using Sage.
	Therefore, we may assume that $p\geq 12$.
	Let $D_{(p,q)}$ denote the cyclic subgroup of $\tilde A_n$ generated by $\sigma_{(p,q)}$, and let $m$ denote the order of $\sigma_{(p,q)}$.
	
	For a linear spin character $\delta$ of $D_{(p,q)}$, Lemma~\ref{lemma:lower_bound_for_(p,q)_(p,q)} gives
	\begin{align}~\label{equation:lower_bound_for_(p,q)_(p,q)_alt}
	\langle \Res^{\tilde A_n}_{D_{(p,q)}}(  \phi_\lambda^+ + \phi_\lambda^-), 	\delta \rangle_{D_{(p,q)}} = \langle \Res^{\tilde S_n}_{D_{(p,q)}}  \phi_\lambda , 	\delta \rangle_{D_{(p,q)}}  >p.
	\end{align}

	Now compute, using Theorem~\ref{theorem:Delta_details_double_alternating}:
	\begin{align*}
		\langle \Res^{\tilde A_n}_{D_{(p,q)}}  \phi_\lambda^+ - \phi_\lambda^-, 	\delta \rangle_{D_{(p,q)}}
		&=
		\frac{1}{m} \sum_{i=0}^m \Delta_\lambda(\sigma_{(p,q)}^i) \overline{\delta(\sigma_{(p,q)}^i)} \\
		&=
		\frac{1}{m} \sum_{i=0 \text{ and } (i,m)=1}^m (\pm)i^{(n-2)/2} \sqrt{pq} \overline{\delta(\sigma_{(p,q)}^i)} \\
		&\leq \frac{1}{m} \varphi(m) \sqrt{pq}
		\leq \sqrt{pq}
		\leq p,
	\end{align*}
	where $\varphi$ is the Euler's totient function.
	The lemma now follows from this estimate and~\eqref{equation:lower_bound_for_(p,q)_(p,q)_alt}.
\end{proof}

\begin{corollary}\label{corollary:main_alt_all_parts_even}
	Theorem~\ref{theorem:main_alt} holds when $\lambda=\alpha$ and all parts of $\lambda$ are even.
\end{corollary}

\begin{proof}
	Let $W$ be an irreducible spin representation of $\tilde A_n$.
    Let $l=\ell(\lambda)$.
	Then
	\begin{displaymath}
	\Res^{\tilde A_n}_{\tilde A_{(\lambda_1 +\lambda_2, \dots, \lambda_{l-1}+\lambda_l)}} W \geq V_1 \otimes V_2 \otimes \dots \otimes V_{l/2},
	\end{displaymath}
	for some irreducible spin representations $V_j$ of $\tilde S_{\lambda_{2j-1}+\lambda_{2j}}$, where $j\in [l/2]$.
	By Lemma~\ref{lemma:main_alt_two_even_cycles}, $\Sp_{V_j}(\sigma_{(\lambda_{2j-1},\lambda_{2j})})$ has $\lcm(\lambda_{2j-1},\lambda_{2j})$ distinct elements for all $j\in [l/2]$.
	A repeated application of Lemma~\ref{lemma:alt_full_times_full_is_full} now shows that $\Sp_W(\sigma_\lambda)$ has $\lcm(\lambda)$ distinct elements.
\end{proof}

We now treat the remaining cases for double covers of alternating groups.


\begin{proof}[Proof of Theorem~\ref{theorem:main_alt}]
	We prove the theorem by induction on $n$.
	We verify the theorem directly for $n\leq 18$, so assume $n\geq 19$.

    By lemma~\ref{lemma:double_A_n_not_DP_n^-} and Theorem~\ref{theorem:main}, the theorem follows if $\lambda \in \DP_n^- $ or $\lambda \in \DP_n^+$ with $\lambda \neq \alpha$.
    Therefore, we may assume that $\lambda = \alpha$ and $\lambda\in \DP_n^+$.
    Recall that it suffices to show that $g$ has $\lcm(\alpha)$ many distinct eigenvalues in $V_\lambda$.

	If $l:= \ell(\lambda)=1$, then the result follows from Theorem~\ref{theorem:cycle_double_A_n_odd}.
	Thus, assume that $\lambda$ has at least two parts.
	If all parts of $\lambda$ are even, then the result follows from Corollary~\ref{corollary:main_alt_all_parts_even}.
	Therefore, $\lambda$ has at least one odd part.
	If $\lambda=(p,q)$ has two odd parts, then
	\begin{displaymath}
	\Res^{\tilde S_n}_{\tilde S_p \times_c \tilde S_q} V_\lambda \geq V_{(p)} \otimes (V_{(q)} \oplus V_{(q-1,1)})
	\end{displaymath}
	by the shifted Littlewood--Richardson rule.
    Therefore, $V_{(p)}^+ \otimes V_{(q)}^+ \leq \Res^{\tilde A_n}_{\tilde A_p \times_c \tilde A_q} V_\lambda^\epsilon$ for some $\epsilon \in \{\pm\}$.
    Conjugating this equation by $t:=t_1$, yields  
    $V_{(p)}^- \otimes V_{(q)}^+ \leq \Res^{\tilde A_n}_{\tilde A_p \times_c \tilde A_q} V_\lambda^{\tilde\epsilon}$, where $\tilde\epsilon= \begin{cases}
        - & \text{if } \epsilon=+,\\
        +  & \text{otherwise.}\\
    \end{cases}$
    Here, we have used the fact that $t_1$ commutes with every element of $\tilde A_q$ and $V_{(p)}^{\epsilon} = V_{(p)}^{\tilde \epsilon}$.
    Clearly, $V_{(p)}^\epsilon \otimes V_{(q-1,1)} \leq \Res^{\tilde A_n}_{\tilde A_p \times_c \tilde A_q} V_\lambda^+$ for some $\epsilon\in \{\pm\}$.

    In any case, we obtain that for each $\epsilon \in \{\pm\}$, there exist $\epsilon_1,\epsilon_2, \epsilon_3 \in \{\pm\}$, such that
    \begin{equation}~\label{equation:alt_(p,q)_both_odd}
        \Res^{\tilde A_n}_{\tilde A_p \times_c \tilde A_q} V_\lambda^\epsilon \geq \big(V_{(p)}^{\epsilon_1} \otimes V_{(q-1,1)} \big), \big(V_{(p)}^{\epsilon_2} \otimes V_{(q)}^+ \big),\big(V_{(p)}^{\epsilon_3} \otimes V_{(q)}^- \big).
    \end{equation}
    Since $p$ and $q$ are odd, we may assume, multiplying $g$ by $-1$ if necessary, that $g$ has order $\lcm(p,q)$.
    Write $g = g_1 g_2$ with $g_1 \in \tilde A_p$ have order $p$ and $g_2 \in \tilde A_q$ have order $q$.
    Since $p\geq 11$, by induction $\Sp_{(p)^{\epsilon_1}}(g_1) = \Sp_{(p)^{\epsilon_2}}(g_1) = \Sp_{(p)^{\epsilon_3}}(g_1) = \Omega(p,1)$.
    Once again by induction, we obtain that $\Sp_{(q)^{+}}(g_2) \cup \Sp_{(q)^{-}}(g_2) \cup \Sp_{(q-1,1)}(g_2) = \Omega(q,1)$.
    Therefore, by Eq.~\eqref{equation:alt_(p,q)_both_odd}, 
    $\Sp_\lambda(g) \supset \Omega(p,1) \times \Omega(q,1)$ which by Proposition~\ref{proposition:product_of_two_sets_cardiality_in_complex_numbers} contains $\lcm(p,q)$ many distinct elements.

	
	Now consider the case where $\lambda$ has at least three parts.
	Define
	\begin{displaymath}
	p=
	\begin{cases}
		3 & \text{if } 3\in \lambda, \\
		5 & \text{if } 5\in \lambda \text{ but } 3\notin \lambda, \\
		\max\{\lambda_i : \lambda_i \text{ is odd}\} & \text{otherwise}.
	\end{cases}
	\end{displaymath}
    Write $g=g_1 g_2$ with $g_1 \in \tilde S_{p}$ and $g_2 \in \tilde S_{n-p}$ such that $g_1$ has cycle type $(p)$ with order $p$.
    Let $\beta$ be the cycle type of $g_2$ and $t=\lcm(\beta)$.
	By Lemma~\ref{lemma:combinatorial_lemma_for_double_A_n}, we have
	\begin{displaymath}
	\Res^{\tilde S_n}_{\tilde S_p \times_c \tilde S_{n-p}} V_\lambda \geq V \otimes V_{(p)} \otimes V_{\mu},
	\end{displaymath}
    and if $p>2$,
    \begin{displaymath}
	\Res^{\tilde S_n}_{\tilde S_p \times_c \tilde S_{n-p}} V_\lambda \geq V \otimes V_{(p-1,1)} \otimes V_{\nu},
	\end{displaymath}
    where $\mu,\nu$ are different from $(n-p)$.

    By induction, we get that for each $\epsilon \in \{\pm\}$, there exist $\epsilon,\epsilon_1,\epsilon_2,\epsilon_3 \in \{\pm\}$, such that
    \begin{displaymath}
	\Res^{\tilde A_n}_{\tilde A_p \times_c \tilde A_{n-p}} V_\lambda^{\epsilon} \geq V_{(p)}^{\epsilon_1} \otimes V_{\mu}^{\epsilon_2},  V_{(p-1,1)} \otimes V_{\nu}^{\epsilon_3}.
	\end{displaymath}
    By induction, we have $\Sp_{\mu^{\epsilon_2}}(g_2) = \Sp_{\nu^{\epsilon_3}}(g_2) = \Omega(t,\varepsilon_g)$ and $\Sp_{(p)^{\epsilon_1}}(g_1) \cup \Sp_{(p-1,1)}(g_1) = \Omega(p,1)$.
    Hence, $\Sp_{\lambda^\epsilon}(g) \supseteq \Omega(p,1) \times \Omega(t,\epsilon_g)$ which by Proposition~\ref{proposition:product_of_two_sets_cardiality_in_complex_numbers} contains $\lcm(\alpha)$ distinct elements.
    This completes the proof.
\end{proof}
\section{Concluding remarks}
We close the paper with further concluding remarks and open problems, 
complementing the one posed in the introduction.

A complete characterization of the minimal polynomials of elements of a group in all irreducible representations has been obtained only for the symmetric and alternating groups (see~\cite{Staroletov_all_eigenvalues}) and their double covers (see Theorems~\ref{theorem:main} and~\ref{theorem:main_alt}).
For all other finite simple groups, and more generally for all quasi-simple groups, the problem remains open.
For related problems in this direction, see~\cite{survey_Zalesski_eigenvalue_1_2025}.

We also note that, as mentioned in the introduction, the analogous problem in positive characteristic~$q$ is considerably more difficult and remains open even for the symmetric and alternating groups.
The strongest known result in this direction is due to Thompson~\cite[Lemma~3.2]{Thompson_composition_factors_2008}, who handles the case $G = A_n$ with $g$ an $n$-cycle, where $n$ is prime and $q < n < 2q$.


For the symmetric group $S_n$, the multiplicities of eigenvalues of permutations admit combinatorial interpretations: in terms of the major index~\cite{Kraskiewicz_Weyman} for $n$-cycles, and in terms of the multi-major index~\cite{JAM} for general permutations, on standard and semistandard Young tableaux, respectively.
It would be interesting to find an analogous combinatorial interpretation for the multiplicities of eigenvalues of elements of $\tilde{S}_n$ in the irreducible spin representations, especially given that the combinatorics of shifted Young tableaux has been well developed by Sagan~\cite{Sagan_shifted_tableaux_1987} and Worley~\cite{Worley_shifted_tableaux_1984}.

\begin{problem}
	Find a statistic on shifted Young tableaux that gives a combinatorial interpretation of the multiplicities of the eigenvalues of elements of $\tilde{S}_n$ in the irreducible spin representations of $\tilde{S}_n$.
\end{problem}

Based on Sage computations, we have observed that the eigenvalue multiplicities, when nonzero, appear to be nearly evenly distributed for elements of $\tilde{S}_n$ in irreducible spin representations when $n$ is large.
We note that Theorem~\ref{theorem:cycle_double_S_n}, Corollary~\ref{corollary:lower_bound_non_basic}, and Lemma~\ref{lemma:lower_bound_for_(p,q)_(p,q)} provide linear lower bounds for certain specific elements of $\tilde{S}_n$.
This suggests the following question, analogous to~\cite[Theorem~1.9]{Swanson}.

\begin{problem}
Does an analogue of~\cite[Theorem~1.9]{Swanson} hold for $\tilde{S}_n$? Moreover, is it true that for every irreducible spin representation $V$ of $\tilde{S}_n$ and every element $g\in\tilde{S}_n$ of sufficiently large order, each eigenvalue of $g$ occurs with multiplicity at least
\[
\left\lfloor \frac{\dim(V)}{o(g)} \right\rfloor?
\]
\end{problem}

Finally, we conclude with a question about the distribution of orders of elements in $\tilde{S}_n$.
Let $u_n$ (resp.\ $v_n$) denote the number of conjugacy classes of $\tilde{S}_n$ whose elements have order $\operatorname{lcm}(\alpha)$ (resp.\ $2\operatorname{lcm}(\alpha)$), where $\alpha$ denotes the cycle type of the elements in the conjugacy class.
The orders of $\sigma_\alpha$ can be computed using Lemma~\ref{lemma:powers_of_pi_lambda}.

\begin{conjecture}
	$u_n > v_n$ for all $n \geq 30$.
\end{conjecture}

\begin{problem}
	Do $u_n$ and $v_n$ admit natural combinatorial interpretations? Does the difference $u_n - v_n$? Does $\frac{v_n}{u_n}$ tend to $\frac{1}{2}$ as $n \to \infty$?
\end{problem}

\begin{table}[H] 
\caption{Exceptional cases for $1\leq n\leq 8$ in $\tilde S_n$}
\label{table:exceptional_cases_1_8}
    \centering
    \renewcommand{\arraystretch}{1.0}
    \begin{tabular}{|c|cc|}
    \hline
        & \multicolumn{2}{c|}{$(\text{conjugacy class}, \text{minimal polynomial})$} \\
        \hline
        $n=2$  &   &  \\
        $\rho_{(2)}^\pm$ & $((2)^+, \frac{x^2+1}{x\pm i})$, & $((2)^-, \frac{x^2+1}{x\mp i})$ \\
        \hline
        $n=3$  &   &  \\
        $\rho_{(3)}$ & $((3)^+, \frac{x^3+1}{x+1})$, & $((3)^-, \frac{x^3-1}{x-1})$ \\
        $\rho_{(2,1)}^\pm$ & $((2,1)^+, \frac{x^2+1}{x\pm i})$, & $((2,1)^-, \frac{x^2+1}{x\pm i})$ \\
        & $((3)^+, \frac{x^3+1}{x^2-x+1})$ & $((3)^-, \frac{x^3-1}{x^2+x+1})$ \\
        \hline
        $n=4$  &   &  \\
        $\rho_{(4)}^\pm$ & $((3,1)^+, \frac{x^3+1}{x+1})$, & $((3,1)^-, \frac{x^3-1}{x-1})$ \\
        & $((4)^+, x^2\mp \sqrt{2}x+1)$, & $((4)^-, x^2\pm \sqrt{2}x+1)$ \\ \hline
        $n=5$  &   &   \\
        $\rho_{(5)}$ & $((3,1^2)^+, \frac{x^3+1}{x+1})$, & $((3,1^2)^-, \frac{x^3-1}{x-1})$, \\
        & $((3,2)^\pm, \frac{x^6+1}{x^2+1})$ & $((5)^\pm, \frac{x^5\pm1}{x\pm1})$ \\
        $\rho_{(3,2)}^\pm$ & $((5)^+, \frac{x^5+1}{x+1})$, & $((5)^-, \frac{x^5-1}{x-1})$ \\ & $((3,2)^+, x^4\pm\sqrt{3}x^3+2x^2\pm\sqrt{3}x+1)$ & $((3,2)^-, x^4\mp\sqrt{3}x^3+2x^2\mp\sqrt{3}x+1)$.  \\ \hline
        $n=6$  &   &   \\
        $\rho_{(6)}^\pm $ & $((6)^+, x^4\mp\sqrt{3}x^3-2x^2\pm\sqrt{3}x+1)$,
        & $((6)^-, x^4\pm\sqrt{3}x^3-2x^2\mp\sqrt{3}x+1)$, \\
        & $((3,1^3)^+, \frac{x^3+1}{x+1})$, & $((3,1^3)^-, \frac{x^3-1}{x-1})$, \\
        & $((5,1)^+, \frac{x^5+1}{x+1})$, & $((5,1)^-, \frac{x^5-1}{x-1})$, \\
        & $((3,2,1)^\pm, \frac{x^6+1}{x^2+1})$. & \\
        $\rho_{(3,2,1)}^\pm$ & $((3,2,1)^+, x^4\mp\sqrt{3}x^3+2x^2\mp\sqrt{3}x+1)$, & $((3,2,1)^-, x^4\pm\sqrt{3}x^3+2x^2\pm\sqrt{3}x+1)$, \\
        & $((3,3)^+, \frac{x^3-1}{x-1})$, & $((3,3)^-, \frac{x^3+1}{x+1})$, \\
        & $((5,1)^+, \frac{x^5+1}{x+1})$, & $((5,1)^-, \frac{x^5-1}{x-1})$, \\
         & $((6)^\pm, \frac{x^6-1}{x^2-1})$. & \\ \hline
	    $n=7$ & & \\
	    $\rho_{(7)}$ & $((3,1^4)^\pm, \frac{x^3\pm1}{x\pm1})$, & $((5,1^2)^\pm, \frac{x^5\pm1}{x\pm5})$,  \\
	      & $((3,2,1^2), \frac{x^6+1}{x^2+1})$, & $((3,2,2), \frac{x^6+1}{x^2+1})$, \\
	      & $((4,3)^\pm, \frac{x^{12}+1}{x^4+1})$, & $((5,1^2)^\pm, \frac{x^{5}\pm1}{x\pm1})$, \\
	      & $((5,2)^\pm, \frac{x^{10}+1}{x^2+1})$. & \\
	    \hline
	    $n=8$ & & \\
	    $\lambda=(8)^\pm$ & $((5,1^3)^+, \frac{x^{5}+1}{x+1})$, & $((5,1^3)^-, \frac{x^5-1}{x-1})$,  \\
	    & $((3,1^5)^+, \frac{x^3+1}{x+1})$, & $((3,1^5)^-, \frac{x^3-1}{x-1})$,  \\
	    & $((5,3)^+, \frac{(x^{15}-1)(x-1)}{(x^5-1)(x^3-1)})$, & $((5,3)^-, \frac{(x^{15}+1)(x+1)}{(x^5+1)(x^3+1)})$,  \\
        & $((8)^+, \frac{x^{8}-1}{x\pm1})$, & $((8)^-, \frac{x^{8}-1}{x\mp1})$,  \\
        & $((5,2,1)^\pm, \frac{x^{10}+1}{x^2+1})$, & $((3,2^2,1), \frac{x^6+1}{x^2+1})$, \\
        & $((3,2,1^3), \frac{x^{6}+1}{x^2+1})$, & $((4,3,1)^\pm, \frac{x^{12}+1}{x^4+1})$. \\
        $\lambda=(7,1)$ & $((5,3)^\pm, \frac{x^{15}\mp1}{x\mp1})$. &   \\   \hline
    \end{tabular}
\end{table}
\begin{table}[H] 
\caption{Non-trivial exceptions for $1\leq n\leq 9$ in $\tilde A_n$}
\label{table:exceptional_cases_1_8_alt}
    \centering
    \renewcommand{\arraystretch}{1.2}
    \begin{tabular}{|c|cc|}
    \hline
        & \multicolumn{2}{c|}{$(\text{conjugacy class}, \text{minimal polynomial})$} \\
        \hline
        $n=3$  &   &  \\
        $\rho_{(3)}^+$ & $((3)^{\pm+}, x\pm \omega_3^2)$, & $((3)^{\pm-}, x \pm \omega_3)$ \\
        $\rho_{(3)}^-$ & $((3)^{\pm+}, x\pm \omega_3)$, & $((3)^{\pm-}, x \pm \omega_3^2)$ \\
        \hline
        $n=4$  &   &  \\
		$\rho_{(3,1)}^+$ & $((3,1)^{\pm+}, (x\pm 1)(x \pm \omega_3)$, & $((3,1)^{\pm-}, (x\pm 1)(x \pm \omega_3^2))$\\
       $\rho_{(3,1)}^-$ & $((3,1)^{\pm+}, (x\pm 1)(x \pm \omega_3^2))$, & $((3,1)^{\pm-}, (x\pm 1)(x \pm \omega_3))$\\
       \hline
        $n=5$  &   &   \\
        $\rho_{(5)}^+$ & $((5)^{\pm+}, (x \pm \omega_5)(x \pm \omega_5^4))$, & $((5)^{\pm-},(x \pm \omega_5^2)(x \pm \omega_5^3) )$, \\
        $\rho_{(5)}^-$ & $((5)^{\pm+},(x \pm \omega_5^2)(x \pm \omega_5^3) )$, & $((5)^{\pm-}, (x \pm \omega_5)(x \pm \omega_5^4))$, \\
         \hline
	    $n=7$ & &  \\ $\rho_{(7)}^+$
		& $((7)^{\pm+}, (x \mp 1)(x \mp \omega_7^3)(x \mp \omega_7^5)(x \mp \omega_7^6))$, & $((7)^{\pm-}, (x\mp 1)(x \mp \omega_7)(x \mp \omega_7^2)(x \mp \omega_7^4))$,  \\
	     $\rho_{(7)}^-$
		& $((7)^{\pm+}, (x\mp 1)(x \mp \omega_7)(x \mp \omega_7^2)(x \mp \omega_7^4)))$, & $((7)^{\pm-}, (x \mp 1)(x \mp \omega_7^3)(x \mp \omega_7^5)(x \mp \omega_7^6))$,  \\
	    \hline
     $n=9$  & & \\
		$\rho_{(9)}^+$ & $((9)^{\pm+}, \frac{(x^9 \mp 1)}{(x^2\pm x+1)}))$,  & $((9)^{\pm-}, \frac{(x^9\mp 1)}{(x \mp 1)})$, \\
        $\rho_{(9)}^-$ & $((9)^{\pm+}, \frac{(x^9\mp 1)}{(x \mp 1)})$,  & $((9)^{\pm-}, \frac{(x^9 \mp 1)}{(x^2\pm x+1)}))$, \\
        \hline
    \end{tabular}
\end{table}

\subsection*{Acknowledgements}
We thank Maria Grechkoseeva for her help in proving Lemma~\ref{l:sign_of_powers_of_cycles}.
The second named author is supported by ANRF National Postdoctoral Fellowship (PDF/2025/003158).
The third named author is supported by by the Russian Science Foundation, project 24-11-00119, \url{https://rscf.ru/en/project/24-11-00119/}.

\bibliographystyle{abbrv}
\bibliography{refs}
\end{document}